\documentclass[11pt, reqno, a4paper]{amsart}
\usepackage{amsmath}
\usepackage{bibentry}
\usepackage{xcolor}
\usepackage{silence}
\usepackage{hyperref}
\hypersetup{
  pdftitle={Symplectic log Kodaira dimension -infty, affine-ruledness and unicuspidal rational curves},
  pdfauthor={Tian-Jun Li and Shengzhen Ning}
}
\makeatletter
\def\Hy@Warning#1{}
\makeatother
\usepackage{todonotes}
\usepackage{xcolor}
\usepackage{dynkin-diagrams}
\usepackage{dsfont}
\usepackage{romannum}
\usepackage{todonotes}
\usepackage{adjustbox}
\usepackage{xcolor}
\usepackage{float}
\usepackage{pgfplots}
\pgfplotsset{compat=1.18}
\usepackage{leftindex}
\usepackage{caption}  
\usepgfplotslibrary{fillbetween}
 \newtheorem{thm}{Theorem}[section]
\newtheorem{theorem}[thm]{Theorem}

\newtheorem{fact}[thm]{Fact}

\newtheorem{proposition}[thm]{Proposition}
\newtheorem{lemma}[thm]{Lemma}

\newtheorem{rmk}[thm]{Remark}
\newtheorem{question}[thm]{Question}

\newtheorem{definition}[thm]{Definition}
\newtheorem{corollary}[thm]{Corollary}

\newtheorem{example}[thm]{Example}

\usetikzlibrary{cd}
\RequirePackage{tikz-cd}
\RequirePackage{amssymb}
\usetikzlibrary{calc}
\usetikzlibrary{decorations.pathmorphing}

\tikzset{curve/.style={settings={#1},to path={(\tikztostart)
			.. controls ($(\tikztostart)!\pv{pos}!(\tikztotarget)!\pv{height}!270:(\tikztotarget)$)
			and ($(\tikztostart)!1-\pv{pos}!(\tikztotarget)!\pv{height}!270:(\tikztotarget)$)
			.. (\tikztotarget)\tikztonodes}},
	settings/.code={\tikzset{quiver/.cd,#1}
		\def\pv##1{\pgfkeysvalueof{/tikz/quiver/##1}}},
	quiver/.cd,pos/.initial=0.35,height/.initial=0}

\tikzset{tail reversed/.code={\pgfsetarrowsstart{tikzcd to}}}
\tikzset{2tail/.code={\pgfsetarrowsstart{Implies[reversed]}}}
\tikzset{2tail reversed/.code={\pgfsetarrowsstart{Implies}}}
\tikzset{no body/.style={/tikz/dash pattern=on 0 off 1mm}}
\newcommand{\RomanNumeralCaps}[1]
    {\MakeUppercase{\romannumeral #1}}
\newcommand{\ZZ}{\mathbb{Z}}

\newcommand{\CC}{\mathbb{C}}

\newcommand{\RR}{\mathbb{R}}
\newcommand{\PP}{\mathbb{P}}
\newcommand{\dd}{\delta}
\newcommand{\lk}{\overline{\kappa}}

\title[Symplectic log Kodaira dimension $-\infty$, affine-ruledness and unicuspidal curves]{Symplectic log Kodaira dimension $-\infty$, affine-ruledness and unicuspidal rational curves}

\date{\today}

\author{Tian-Jun Li and Shengzhen Ning}
\AtEndDocument{\bigskip{\footnotesize
		\textsc{School of Mathematics, University of Minnesota, Minneapolis, MN, US} \par
		\textit{E-mail address}: \texttt{tjli@math.umn.edu} \par
		\addvspace{\medskipamount}
		\textsc{School of Mathematics, University of Minnesota, Minneapolis, MN, US} \par
		\textit{E-mail address}: \texttt{ning0040@umn.edu} \par
		
}}
\begin{document}
\sloppy
\raggedbottom

\begin{abstract}
    Given a closed symplectic $4$-manifold $(X,\omega)$, a collection $D$ of embedded symplectic submanifolds satisfying certain normal crossing conditions is called a symplectic divisor. In this paper, we consider the pair $(X,\omega,D)$ with symplectic log Kodaira dimension $-\infty$ in the spirit of Li-Zhang. We introduce the notion of symplectic affine-ruledness, which characterizes the divisor complement $X\setminus D$ as being foliated by symplectic punctured spheres. We establish a symplectic analogue of a theorem by Fujita-Miyanishi-Sugie-Russell in the algebraic settings which describes smooth open algebraic surfaces with $\lk=-\infty$ as containing a Zariski open subset isomorphic to the product between a curve and the affine line. When $X$ is a rational manifold, the foliation is given by certain unicuspidal rational curves of index one with cusp singularities located at the intersection point in $D$. We utilize the correspondence between such singular curves and embedded curves in its normal crossing resolution recently highlighted by McDuff-Siegel, and also a criterion for the existence of embedded curves in the relative settings by McDuff-Opshtein. Another main technical input is Zhang's curve cone theorem for tamed almost complex $4$-manifolds, which is crucial in reducing the complexity of divisors. We also investigate the symplectic deformation properties of divisors and show that such pairs are deformation equivalent to K\"ahler pairs. As a corollary, the restriction of the symplectic structure $\omega$ on an open dense subset in the divisor complement $X\setminus D$ is deformation equivalent to the standard product symplectic structure. 
    
\end{abstract}

\maketitle
\pagenumbering{arabic}

\tableofcontents

\section{Introduction}

Pseudoholomorphic curve theory and Seiberg-Witten theory are two phenomenal tools opening the door for exploring the topology of closed symplectic $4$-manifolds. In \cite{McDuffrationalruled}, McDuff shows that the existence of an embedded symplectic sphere of self-intersection number $\geq 0$ implies the ambient manifold to be symplectic rational or ruled. By symplectic rational or ruled manifold, we mean a closed symplectic $4$-manifold diffeomorphic to the blowup of $\CC\PP^2$ or the blowup of an $S^2$-bundle over a Riemann surface. On the other hand, the groundbreaking work (\cite{TaubesSWandGr}) by Taubes which relates the Seiberg-Witten invariants to Gromov invariants on symplectic manifolds allows us to find embedded symplectic surfaces in certain cases. Encapsulating these two celebrated work by McDuff and Taubes, along with the $b_2^+=1$ Seiberg-Witten wall crossing formula by Kronheimer-Mrowka \cite{KM94} and more generally Li-Liu \cite{LiLiuWallcrossing}, leads to the following important result discovered independently by Liu and Ohta-Ono. It provides a simple characterization of symplectic rational ruled manifolds only involving the numerical pairing between the cohomology class of $\omega$ and the symplectic canonical class $K_{\omega}$.
Here $K_{\omega}\in H^2(X;\ZZ)$ is defined by $K_{\omega}:=-c_1(TX,J)$ for any $\omega$-compatible almost complex structure $J$; this is well-defined (independent of $J$) up to deformation of $\omega$.

\begin{theorem}[\cite{LiuAiKo,OhtaOnoc_1positive}]\label{thm:rationalruled}
    Let $(X,\omega)$ be a closed symplectic $4$-manifold such that $[\omega]\cdot K_{\omega}<0$, then $(X,\omega)$ must be a symplectic rational or ruled manifold.
\end{theorem}

 We remark that $(X,\omega)$ in Theorem \ref{thm:rationalruled} is not assumed to be minimal (without any embedded symplectic exceptional sphere). This theorem is a symplectic reminiscence of the more classical result in the theory of classification of compact complex surfaces which says that complex surfaces with Kodaira dimension $-\infty$ must be biholomorphic to either rational or ruled surfaces (see for example \cite{BHPV}). Recall that one way to interpret the Kodaira dimension $\kappa(X)=-\infty$ for a complex surface $X$ is that the multiple of the canonical class $nK_X$ is not effective for any $n\geq 1$. From this point of view, for symplectic $4$-manifolds the parallel condition $[\omega]\cdot K_{\omega}<0$ in Theorem \ref{thm:rationalruled} is quite natural since it implies that $K_{\omega}$ can never be represented by symplectic submanifolds. 
 
 Theorem \ref{thm:rationalruled}, together with another result concerning the sign of $K_{\omega}^2$ in \cite{LiuAiKo} which answers Gompf's conjecture, eventually inspires the development of the notion of symplectic Kodaira dimension for closed symplectic 4-manifolds. Such a notion first appeared in \cite{MSsurvey} for minimal manifolds with $b_2^+=1$ and was later completed in \cite{kod06,kod15}. It is defined by first reducing the manifold to its minimal model and then applying the following classification scheme.
\begin{equation}\label{equ:kod}\tag{\textdagger}
\kappa^s(X,\omega) = 
\begin{cases}
   -\infty & \text{if }[\omega]\cdot K_{\omega}<0\text{ or }K_\omega^2<0,\\
   0 & \text{if }[\omega]\cdot K_{\omega}=0\text{ and }K_\omega^2=0,\\
   1 & \text{if }[\omega]\cdot K_{\omega}>0\text{ and }K_\omega^2=0,\\
   2 & \text{if }[\omega]\cdot K_{\omega}>0\text{ and }K_\omega^2>0.
\end{cases}
\end{equation}
 Consequently, all symplectic $4$-manifolds that satisfy the condition in Theorem \ref{thm:rationalruled} are categorized as having $\kappa^s=-\infty$. Now, we are interested in the following question.
 
\begin{question}\label{question:relativeversion}
    What is the relative version of Theorem \ref{thm:rationalruled}?
\end{question}

By `relative', we mean the context of considering the pair $(X,\omega,D)$ where $D$ is a union of embedded symplectic surfaces in $(X,\omega)$. On the algebraic side, a series of works by Fujita, Miyanishi, Sugie and Russell (see also the summary in the monograph \cite{Miyanishibook}) has established a remarkable result in the relative setting.

\begin{theorem}[\cite{Fujita,Miysugiekodinf,Russelaffineruled,Miyanishi82affineruled}]\label{thm:algbraicaffineruled}
    Let $V$ be an open nonsingular algebraic surface with log Kodaira dimension $\lk(V)=-\infty$. When the compactification of $V$ is a rational surface, further assume that the compactifying divisor is connected. Then $V$ is affine-ruled, i.e. $V$ contains a Zariski open subset isomorphic to $C\times \mathbb{A}^1$ where $C$ is a smooth quasi-projective curve and $\mathbb{A}^1$ is the affine line.
\end{theorem}

The above theorem is proved over algebraically closed field of arbitrary characteristic without assuming $V$ is affine\footnote{So the `affine' in the terminology `affine-ruled' means the open algebraic surface $V$ has a ruling by affine lines, rather than $V$ itself is affine.} and applied to answer Zariski's cancellation problem (see \cite{Fujita}). The condition of log Kodaira dimension $\lk(V)=-\infty$ for the open algebraic surface $V$ implies that after taking some simple normal crossing compactification $(X,D)$ with $V=X\setminus D$, any multiple of the \textbf{adjoint class} $K_X+D$ is not effective.  

\subsection[Symplectic log Kodaira dimension $-\infty$ and affine-ruledness]{Symplectic log Kodaira dimension $-\infty$ and affine-ruledness}

By the previous discussion, to define and study the notion of log Kodaira dimension $-\infty$ in the symplectic context, it is natural to first explore the implications of the condition that the adjoint class has negative symplectic area. Motivated by the earlier work of Li-Yau \cite{LiYau} on a single embedded symplectic surface, Li-Zhang \cite{LiZhangrelativeKod} investigate the case where $D$ is a disjoint union of embedded symplectic surfaces with genus $\geq 1$. They introduce a definition of symplectic log Kodaira dimension in the relative setting analogous to (\ref{equ:kod}) under the assumption on the genus. In this paper, we drop this assumption and consider the more general situation where $D$ is a symplectic divisor in the following sense.

\begin{definition}\label{def:sympdivisor}
    A \textbf{symplectic divisor} $D$ in a closed symplectic 4-manifold $(X,\omega)$ refers to a non-empty configuration of finitely many closed embedded symplectic surfaces $D=\cup D_i\subseteq (X,\omega)$ such that all intersections among $D_i$'s are positively transverse and $\omega$-orthogonal. 
\end{definition}

 $(X,\omega,D)$ in the above definition will also be simply called a pair. Note that the $\omega$-orthogonal condition guarantees that no three components in a symplectic divisor can intersect at one point. We say that the divisor $D$ or the pair $(X,\omega,D)$ is \textbf{connected} if the configuration $\cup D_i\subseteq (X,\omega)$ is connected. The above definition was also called the `singular set' in \cite{MO15} and should be thought of as the analogue of simple normal crossing divisors in algebraic settings. We will write $[D]$ as the total homology class of the symplectic divisor $D$, which is the sum of the homology classes of its components $\sum[D_i]$. 

Guided by the classification scheme outlined in \cite{LiZhangrelativeKod}, symplectic log Kodaira dimension $\lk^s=0$ has been extensively studied in works \cite{LiMakLCY,LiMakICCM,LiMinMak,Enumerate}, focusing on the deformation, contact and enumerative properties with connections to symplectic fillings, toric actions and almost toric fibrations. Additionally, \cite{Ouyang} also investigates specific divisors within $\lk^s>0$ relating them with Hamiltonian circle actions. In this paper, we turn our attention to $\lk^s=-\infty$, aiming to establish a symplectic analogue of Theorem \ref{thm:algbraicaffineruled}. In the algebraic context, an affine-ruled open algebraic surface is defined as one that contains a Zariski open subset isomorphic to the product of the affine line and a smooth quasi-projective curve. Here, we first provide a symplectic analogue of the definition for affine-ruledness.

\begin{definition}\label{def:sympaffruled}
    The pair $(X,\omega,D)$ is called {\bf symplectic affine-ruled} if there are finitely many embedded symplectic submanifolds $S_1,\cdots,S_l\subseteq X\setminus D$ and a diffeomorphism \[\Phi:\mathring{\Sigma}_g\times (S^2\setminus\{pt\})\rightarrow X\setminus (D\cup S_1\cup\cdots \cup S_l) \] 
    such that each fiber $\{*\}\times (S^2\setminus\{pt\})$ is a symplectic submanifold with respect to $\Phi^*\omega$ and has equal symplectic area, where $S^2\setminus\{pt\}$ denotes the one-punctured sphere and $\mathring{\Sigma}_g$ denotes a Riemann surface of genus $g$ with finitely many punctures.
\end{definition}

\begin{rmk}
     The result of Greene-Shiohama \cite{openmoser} implies that the symplectic forms on the punctured Riemann surface with the same finite volume are diffeomorphic. This is generalized into a family version by Pelayo-Tang \cite{PelayoTang,PelayoTang2}. In particular, we can choose $\Phi$ in Definition \ref{def:sympaffruled} to satisfy that, for any compact subset $K\subseteq \mathring{\Sigma}_g$, the restriction of $\Phi^*\omega$ on $\{*\}\times (S^2\setminus\{pt\})$ for $*\in K$ is a fixed standard symplectic form. 
\end{rmk}

Now we can state our main result. 

\begin{theorem}\label{thm:firstmain}
    Let $D\subseteq (X,\omega)$ be a connected symplectic divisor with $[\omega]\cdot(K_{\omega}+[D])<0$. Then the pair $(X,\omega,D)$ is symplectic affine-ruled. Moreover, the same is true when $D$ is not connected and $X$ is irrational ruled. 
\end{theorem}

Since $[\omega]\cdot K_{\omega}<[\omega]\cdot(K_{\omega}+[D])$ for any symplectic divisor $D$, Theorem \ref{thm:rationalruled} implies that the ambient manifold $X$ must be rational or ruled under the assumption $[\omega]\cdot(K_{\omega}+[D])<0$. The main Theorem \ref{thm:firstmain} is then a combination of Theorem \ref{thm:rationalmain} and \ref{thm:ruledmain} which handles the rational case and the irrational ruled case respectively. As we will see in the proof, we can provide more precise characterizations of those embedded symplectic submanifolds $S_1,\cdots,S_l\subseteq X\setminus D$ in Definition \ref{def:sympaffruled}. Actually, when $X$ is a rational manifold, they must be either a sphere or a $k$-punctured sphere whose closure in $X$ intersects $D$ at exactly $k$ points; when $X$ is an irrational ruled manifold, other than spheres and punctured spheres as the rational case, there might be one $S_i$ whose closure is the section of the ruling. In the rational case, the closure $\overline{S_i}$ could have singularities and may intersect two components of $D$ at their intersection point. As a result,  $D\cup \overline{S_1}\cup\cdots\cup \overline{S_l}$ may not generally qualify as a symplectic divisor under Definition \ref{def:sympdivisor}. However, it can still be considered as the completion of $D$ in a broader sense (non-simple normal crossing divisor). Furthermore, this union can be realized as $J$-holomorphic subvarieties (see Definition \ref{def:Jsubvariety}) by certain tame almost complex structure $J$. This is also called the holomorphic shadow in \cite{kesslershadow} where a Zariski-type structure is established. Therefore, the complement $X\setminus (D\cup S_1\cup\cdots \cup S_l)$ can be interpreted as the Zariski open subset in the symplectic setting, parallel to Theorem \ref{thm:algbraicaffineruled}.

\subsection{Unicuspidal rational curves}\label{section:mainidea}

The strategy for proving Theorem \ref{thm:firstmain} is in the same spirit as Theorem \ref{thm:rationalruled}. We wish to find some embedded symplectic sphere $C\subseteq (X,\omega)$ with $[C]^2=0$ by Gromov-Taubes-Seiberg-Witten invariants. In the relative setting, this can be achieved by appealing to McDuff-Opshtein's criterion \cite{MO15} (Theorem \ref{thm:MO15}) and there will be some tame almost complex structure $J$ realizing both $C$ and the divisor $D$ as $J$-holomorphic submanifolds. Ideally, if $C$ only transversally intersects one component in $D$, say $D_0$, at exactly one point, then the Lefschetz fibration produced by the moduli space of $J$-holomorphic curves in class $[C]$ might be expected to have $D_0$ as a section and other components in $D$ contained in the singular fibers. When $X$ is irrational ruled, this idea works well by taking $C$ to be the spherical fiber and applying Zhang's result on the fiber class \cite{Zhangmoduli} (Theorem \ref{thm:fiberclass}). The details are spelled out in Section \ref{section:ruled}. 

However, the situation becomes more intriguing when $X$ is a rational manifold. Such an embedded sphere may not exist simply for homological reasons. As a consequence, we allow $C$ to have exactly one cusp singularity and be embedded elsewhere, which is called a {\bf unicuspidal rational curve}. Such curves have recently been studied by McDuff-Siegel (\cite{mcduffsiegelellipsoidalsuperpotentialssingularcurve,MSinfstair}), which give (stable) symplectic embedding obstructions of ellipsoids. When requiring the index to be $0$, a correspondence between $C$ in $X$ and the exceptional curve $\tilde{C}$ coming from the normal crossing resolution in the blowup of $X$ is established in \cite[Section 4]{mcduffsiegelellipsoidalsuperpotentialssingularcurve}. For our purpose of foliating the divisor complement by punctured spheres, we need to consider unicuspidal rational curves with index $1$ instead. If the cusp is exactly located at an intersection point between two components in $D$, the complement of the total transform $\tilde{D}$ of $D$ under the normal crossing resolution in the blowup of $X$ will be diffeomorphic to $X\setminus D$. The condition on index will guarantee the smooth curve $\tilde{C}$ after the resolution to behave like the fiber class, which would give the desired fibration on the divisor complement. See Figure \ref{fig:mainidea} for a cartoon of the idea.

\begin{figure}[ht]
		\centering\includegraphics*[height=5cm, width=13cm]{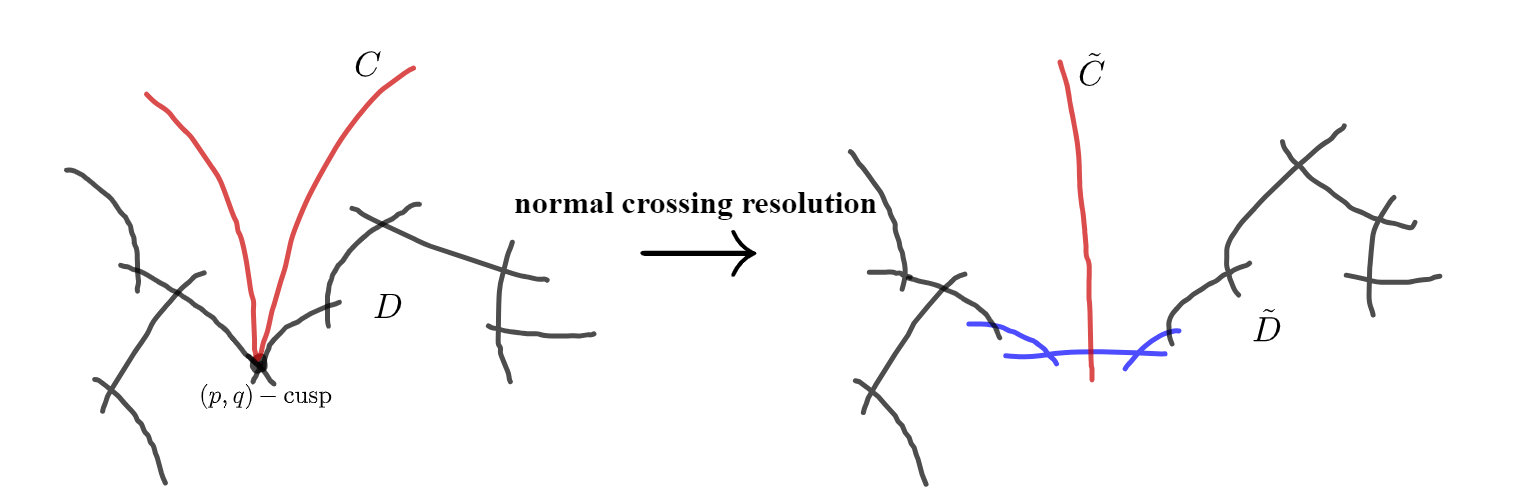}
 
		\caption{After the normal crossing resolution, \(\tilde{C}\) will be an embedded sphere of self-intersection 0.  \label{fig:mainidea}}
	\end{figure}

The majority of Section \ref{section:rational} is dedicated to the search for such a unicuspidal rational curve $C$. Initially, we only have the condition $[\omega]\cdot (K_{\omega}+[D])<0$ so that the configuration of $D$ could be quite complicated. However, by performing a sequence of blowdowns, we reduce the complexity, arriving at a model birationally equivalent to $(X,\omega,D)$, where the existence of $C$ becomes more evident.  The diagram accompanying the proof of Theorem \ref{thm:rationalmain} outlines this strategy. After the first reductions, three possible configurations emerge:
\begin{itemize}
    \item quasi-minimal pairs of first kind;
    \item quasi-minimal pairs of second kind;
    \item the ambient manifold has $b_2\leq 2$.
\end{itemize} When $b_2\leq 2$, it is possible to enumerate all the configurations so as to find the curve $C$ through a case-by-case discussion. The quasi-minimal pair of second kind would have configurations which look like a trident and can be further reduced to the case of $b_2\leq 2$, as detailed in Section \ref{section:psedolcy}. 

The quasi-minimal pair of first kind is more sophisticated but still manageable (Section \ref{section:LCY}), as it can be completed into a {\bf symplectic log Calabi-Yau divisor} introduced by Li-Mak \cite{LiMakLCY}. We then further proceed to reduce the complexity by applying the minimal reduction procedure for log Calabi-Yau divisors, obtaining a partially minimal model with a configuration that forms a chain of spheres. The homology class $[C]$ for the unicuspidal curve is taken to be a positive linear combination of the components in a subchain (called `admissible' in Section \ref{section:findcusp}). To show the existence of the unicuspidal rational curve $C$ in such a class, we transition to the blowup manifold via the normal crossing resolution. Using McDuff-Opshtein’s criterion to get the embedded curve $\tilde{C}$ with $[\tilde{C}]^2=0$ in the blowup manifold, there will be the desired singular curve $C$ in $X$ by the curve correspondence highlighted by McDuff-Siegel \cite{mcduffsiegelellipsoidalsuperpotentialssingularcurve} (Proposition \ref{prop:modulirelation}). This approach allows us to bypass the intricate analysis of the moduli space for singular $J$-holomorphic curves.

\subsection{Comparison with K\"ahler pairs}
Note that our definition of symplectic affine-ruledness only requires the divisor complement $X\setminus D$ to contain an open dense subset foliated by symplectic one-punctured spheres with equal symplectic area. This is the main distinction with the result in the algebraic setting, where it is shown that the Zariski open subset is isomorphic to a product between a curve and an affine line. However, in the symplectic setting, one can not naively expect the open dense subset in the divisor complement is symplectomorphic to the product. Consider the simplest example of the complement of a line in $\CC\PP^2$ which is biholomorphic to $\CC\times\CC$ but symplectomorphic to a ball. It is known there are many restrictions for symplectic embeddings of a polydisk into a ball other than the volume constraints (\cite{HindLisi,Hutchingsbeyond,Jopolydisk}). 

Nevertheless, we investigate the deformation aspect of symplectic divisors and obtain the following result in Section \ref{section:deformation}, where various notions of deformations between symplectic divisors are introduced.

\begin{theorem}[=Theorem \ref{thm:deformation}]\label{thm:maindeformation}
    Let $D\subseteq (X,\omega)$ be a connected symplectic divisor with $[\omega]\cdot(K_{\omega}+[D])<0$. Then $(X,\omega,D)$ is symplectic deformation equivalent to a K\"ahler pair with $\overline{\kappa}(X\setminus D)=-\infty$. Moreover, the same is true when $D$ is not connected and $X$ is irrational ruled. 
\end{theorem}

Combined with the result in algebraic geometry, we have the following corollary.

\begin{corollary}
    Let $D\subseteq (X,\omega)$ be a connected symplectic divisor with $[\omega]\cdot(K_{\omega}+[D])<0$. Then $X\setminus D$ contains an open dense subset symplectic deformation equivalent to $\mathring{\Sigma}_g\times (S^2\setminus\{pt\})$ equipped with the product symplectic structure. Moreover, the same is true when $D$ is not connected and $X$ is irrational ruled. 
\end{corollary}
\begin{proof}
    By Theorem \ref{thm:maindeformation}, we may assume there is a deformation of symplectic forms $\omega_t$ of $\omega$ such that $(X,\omega_1,D)$ is a K\"ahler pair (this is called a $D$-symplectic homotopy in Section \ref{section:deformation}). This implies $\omega_1$ is a K\"ahler form of some complex structure $J$ on $X$. Now we can apply Theorem \ref{thm:algbraicaffineruled} to the pair $(X,J,D)$ with $\overline{\kappa}(X\setminus D)=-\infty$ to obtain a Zariski open subset biholomorphic to $\mathring{\Sigma}_g\times \CC$, where a product symplectic form $\omega_2$ is also K\"ahler. Then the linear interpolation between two K\"ahler forms $\omega_1|_{\mathring{\Sigma}_g\times \CC}$ and $\omega_2$ will be a symplectic deformation. Thus we also have a symplectic deformation from  $\omega|_{\mathring{\Sigma}_g\times \CC}$ and $\omega_2$.
\end{proof}

When $D$ is empty, the above result is well-known: any symplectic form on a rational ruled manifold is deformation equivalent to a K\"ahler form (\cite{spaceofsymp}). Indeed, under the condition $[\omega]\cdot K_{\omega}<0$ of Theorem \ref{thm:rationalruled}, the deformation can be strengthened into symplectomorphism.

\begin{proposition}\label{prop:kahlermain}
    Let $(X,\omega)$ be a symplectic $4$-manifold with $[\omega]\cdot K_\omega<0$, then $\omega$ is a K\"ahler form.
\end{proposition}
When $X$ is a rational manifold, the conclusion of the above proposition can be seen by considering the `good generic' complex structures described in Friedman-Morgan \cite{FM88}. This means the anti-canonical class $-K$ is effective and smooth, and there is no smooth rational curve of self-intersection $-2$. It follows that the K\"ahler cone will agree with the symplectic cone with a fixed canonical class characterized by Li-Liu \cite{LiLiucone}. In Section \ref{section:appendix}, we will give a proof for irrational ruled surfaces using $J$-compatible inflation technique for manifolds with $b_2^+=1$. We remark that the condition $[\omega]\cdot K_\omega<0$ is necessary since Cascini-Panov \cite{CP12} point out that $(T^2\times S^2)\#\overline{\CC\PP}^2$ admits non-K\"ahler symplectic forms $\omega$ with $[\omega]\cdot K_{\omega}>0$.  The above discussion motivates us to ask a similar question in the relative setting. 
\begin{question}
    Is any pair $(X,\omega,D)$ with $[\omega]\cdot(K_{\omega}+[D])<0$ actually a K\"ahler pair?
\end{question}

An affirmative answer to the above question will make Theorem \ref{thm:firstmain} subsumed by Theorem \ref{thm:algbraicaffineruled}, since it implies that all the symplectic divisors considered in this paper actually correspond to some algebraic models. Such a question is also important for defining a reasonable notion of symplectic log Kodaira dimension. In the absolute setting, when $(X,\omega)$ is also K\"ahler, the definition (\ref{equ:kod}) involving the symplectic structure agrees with the holomorphic Kodaira dimension using the complex structure. Ideally, in the relative setting, a suitable notion of symplectic log Kodaira dimension for a K\"ahler pair is supposed to also agree with its holomorphic counterpart.

\subsection{Acknowledgment}
The authors would like to thank Jun Li for discussions on $b_2^+=1$ $J$-compatible inflation, and Jie Min for long-term collaborations on  symplectic log Calabi-Yau divisors and some helpful conversations during Rutgers Symplectic Summer School 2024 and Yamabe Memorial Symposium 2024. We are also grateful to Steven Lu, Ruiran Sun and Weiyi Zhang for their interest  and many useful discussions. Finally, we thank the anonymous referees for their valuable suggestions, which clarified several arguments and improved the presentation.




\section{Symplectic divisors in dimension 4}
In this section, we gather various techniques for handling symplectic divisors in symplectic $4$-manifolds that will be used in subsequent discussions. 

\subsection{Operations on symplectic divisors}\label{sec:divisoroperation}

Given the pair $(X,\omega,D)$, the number of components $D_i$'s in $D$ will be denoted by $l(D)$.  we can consider its dual graph $\Gamma(D)$ whose vertices are given by each component $D_i$ and edges are given by the intersection points of its components. The \textbf{total genus} of $D$ is defined to be 
\[g(D):=\frac{1}{2}([D]^2+K_{\omega}\cdot [D])+1.\]
 A symplectic embedding $I$ from the standard ball $B(\delta)$ of radius $\delta$ in $(\CC^2,\omega_{\text{std}})$ to $(X,\omega)$ is said to be {\bf relative to $D$} if $I^{-1}(D)$ is either empty or the union of (one or two) coordinate planes intersecting with $B(\delta)$. If $p=D_a\cap D_b$ is an intersection point between two components in $D$, by the $\omega$-orthogonal assumption, there will be a symplectic embedding $I$ near $p$ and relative to $D$ such that $D_a,D_b$ correspond to two coordinate planes. As described in \cite[Section 4.1]{Weinsteincomplement}, we can then perform the {\bf smoothing operation} by locally replacing two transversally intersecting disks with an annulus. The outcome is another symplectic divisor $D'$ with $l(D')=l(D)-1$, as the components $D_a$ and $D_b$ will be `summed' into a new component with genus $g(D_a)+g(D_b)$ and homology class $[D_a]+[D_b]$. We have the relation between $g(D)$ and the genus of each component $g(D_i)$ given by the lemma below. 

\begin{lemma}\label{lem:g(D)}
    \[g(D)=\sum_{i=1}^{l(D)}g(D_i)+\text{dim}H_1(\Gamma(D))-\text{dim}H_0(\Gamma(D))+1.\]
\end{lemma}

\begin{proof}
We can perform the smoothing operations at all the intersection points in $D$ to get a disjoint union of embedded symplectic submanifolds $C_1,\cdots,C_k$. Note that $k$ and $\sum_{i=1}^{k}g(C_i)$ are given by $\text{dim}H_0(\Gamma(D))$ and $\sum_{i=1}^{l(D)}g(D_i)+\text{dim}H_1(\Gamma(D))$ respectively. The above equation is then a direct consequence of the adjunction formula for these $C_1,\cdots,C_k$. 
\end{proof}

Next, we discuss the birational equivalence between symplectic divisors. When there is a symplectic embedding of the ball $B(\delta)$, the symplectic blowup construction (\cite[Section 7.1]{MSintro}) will remove $I(B(\delta))$ and collapse the boundary by Hopf fibration into a symplectic exceptional sphere $e$ with symplectic area $\delta$. When the embedding is relative to the divisor $D$, the intersection between $D$ and the boundary of $I(B(\delta))$ is exactly the circle fibers of the Hopf fibration. The {\bf proper transform} $\tilde{D}_i$ of $D_i$, a component in $D$, is then the symplectic submanifold in the blowup manifold $(\tilde{X},\tilde{\omega})$ obtained by removing the interior of the disk $I(B(\delta))\cap D_i$ and then collapsing its boundary circle. The proper transform $\tilde{D_i}$ will have an $\tilde{\omega}$-orthogonal intersection with the exceptional sphere $e$. Therefore, the union of all proper transforms $\tilde{D_i}$ and $e$ will be a symplectic divisor in $(\tilde{X},\tilde{\omega})$ which is called the {\bf total transform} of $D$.

Depending on the position of the symplectic embedding $I$ relative to the divisor $D$, we say $(\tilde{X},\tilde{\omega},\tilde{D})$ is a $\star$-blowup of $(X,\omega,D)$ by using the embedding $I$, where
\begin{itemize}
    \item $\star$=\textbf{exterior}, if $I(B(\delta))$ is disjoint from $D$ and $\tilde{D}$ is either the total transform of $D$ or the union of all proper transforms $\tilde{D_i}$;
    \item $\star$=\textbf{toric}, if $I(B(\delta))$ is centered at the intersection point and $\tilde{D}$ is the total transform of $D$;
    \item $\star$=\textbf{non-toric}, if $I(B(\delta))$ only intersects one component and $\tilde{D}$ is the union of all proper transform $\tilde{D_i}$;
    \item $\star$=\textbf{half-toric}, if $I(B(\delta))$ only intersects one component and $\tilde{D}$ is the total transform of $D$.
\end{itemize}
The symplectic exceptional sphere $e$ is called exterior/toric/non-toric/half-toric with respect to $\tilde{D}$ accordingly. See Figure \ref{fig:blowup}.

\begin{figure}[ht]
		\centering\includegraphics*[height=6cm, width=13cm]{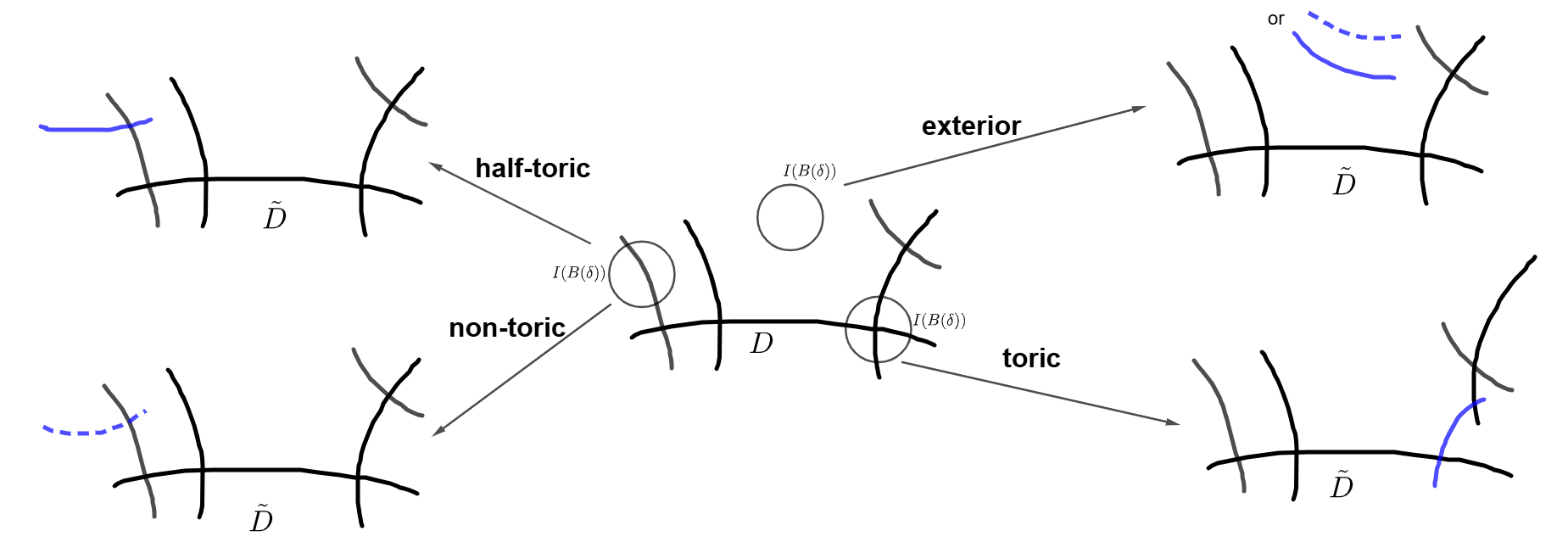}
 
		\caption{Four types of blowups, where the dashed curve indicates a component not contained in $\tilde{D}$. To distinguish the types, one must consider the full data $(X,\omega,D)$ and $(\tilde{X},\tilde{\omega},\tilde{D})$. For example, comparing $(X,\omega)$ with $(\tilde{X},\tilde{\omega})$ determines the exceptional class, and comparing $D$ with $\tilde{D}$ then distinguishes the half-toric and non-toric cases. \label{fig:blowup}}
	\end{figure}

We can analyze the effect of toric blowups on self-intersection sequence $([D_1]^2,[D_2]^2,\cdots)$ of the divisor $D$. For a sequence of integers $(a_1,\cdots,a_n)$ and some $1\leq k\leq n-1$, we define the toric blowup of the sequence at the position $k$ as $(a_1,\cdots,a_{k-1},a_k-1,-1,a_{k+1}-1,a_{k+2},\cdots,a_n)$ .

\begin{definition}\label{def:toricblowupseq}
A {\bf toric blowup sequence} $(t_1,\cdots,t_n)$ is a sequence that can be obtained through successive toric blowups starting from the initial sequence $(0,0)$.
\end{definition}

The following direct observation will be needed later. 

\begin{fact}\label{fact:blowupseq}
    Any toric blowup sequence other than $(0,0)$ contains at least two entries equal to $-1$. Moreover, except for $(-1,-1,-1)$, no toric blowup sequence contains two consecutive $-1$'s.
\end{fact}

The inverse procedure of the symplectic blowup construction yields the symplectic blowdown operation. Note that the $\omega$-orthogonality is preserved under blowdown. This is obvious for exterior, non-toric and half-toric blowdown. For toric blowdown, we can apply \cite[Proposition 3.5]{Symingtonblowdown} to identify a neighborhood of the exceptional sphere with the toric model, where the blowdown can be viewed as adding a corner to the moment polygon and the preimage of two edges under the moment map must intersect orthogonally. An important observation, which we will frequently rely on later, is that half-toric and exterior blowdowns decrease the pairing $[\omega]\cdot(K_{\omega}+[D])$ by $\omega(E)$, where $E$ is the exceptional class for blowdown; while non-toric and toric blowdown leave this paring unchanged. Consequently, all these four types of blowdowns will preserve the condition $[\omega]\cdot(K_{\omega}+[D])<0$.

Finally, we can compare symplectic blowup of divisors with the model of almost complex blowup, analogous to the absolute case. Let $J$ be an $\omega$-compatible almost complex structure which is integrable in $I(B(\delta))$. If the symplectic embedding $I$ is also holomorphic with respect to $J$, then we can take the almost complex blowup $(\overline{X},\overline{J})$ by replacing $I(0)$ by $\CC\PP^1$ and the total transform $\overline{D}\subseteq\overline{X}$ of $D$ under the projection $p:\overline{X}\rightarrow X$. As described in \cite[Theorem 7.1.21]{MSintro}, there is a symplectic form $\overline{\omega}$ compatible with $\overline{J}$ such that $(\overline{X},\overline{\omega})$ is symplectomorphic to $(\tilde{X},\tilde{\omega})$ obtained by the symplectic construction. If $\tilde{D}\subseteq (\tilde{X},\tilde{\omega})$ is the total transform of $D$ in the symplectic blowup model, this symplectomorphism will naturally map $\overline{D}$ to $\tilde{D}$ by its construction. In particular, we have the following corollary.

\begin{corollary}\label{cor:divisorblowupcompare}
    Let $D$ be a symplectic divisor in a K\"ahler manifold $(X,\omega,J)$ whose components are also $J$-holomorphic. If the symplectic embedding $I:B(\delta)\rightarrow X$ relative to $D$ is also holomorphic and $(\tilde{X},\tilde{\omega},\tilde{D})$ is the symplectic exterior/toric/non-toric/half-toric blowup of $(X,\omega,D)$ by using $I$, then there exists a complex structure $\tilde{J}$ compatible with $\tilde{\omega}$ such that all components in $\tilde{D}$ are $\tilde{J}$-holomorphic.
\end{corollary}

\subsection{Taubes-Seiberg-Witten theory}

In this section, we give a brief introduction to Taubes' Seiberg-Witten theory on a closed symplectic $4$-manifold $(X,\omega)$. For the entire story, see \cite{Taubes1,Taubes2,Taubes3,Taubes4}.

 To define Seiberg-Witten invariant, we firstly need a Riemannian metric $g$ and a spin$^c$ structure which gives the spinor bundles $S^{\pm}$ with their determinant line bundle $\mathcal{L}:=\text{det}(S^+)=\text{det}(S^-)$. By the choice of the canonical class $K_{\omega}$, there is a bijection between spin$^c$ structures on $X$ and $H^2(X;\ZZ)$ by associating $\mathcal{L}$ to $e:=\frac{1}{2}(c_1(\mathcal{L})+K_{\omega})$. Then for a self-dual $2$-form $\eta$, the Seiberg-Witten equations are defined for a pair $(A,\phi)$ consisting of a connection $A$ of $\mathcal{L}$ and a section $\phi$ of $S^+$. They are given by
\[D_A\phi=0\]
\[F_A^+=iq(\phi)+i\eta,\]
where $q:\Gamma(S^+)\rightarrow \Omega_+^2(X)$ is a canonical map. When the choice of $(g,\eta)$ is generic, the quotient of the space of solutions by $C^{\infty}(M;S^1)$ is a compact manifold $\mathcal{M}_X(\mathcal{L},g,\eta)$ of dimension 
\[I(e):=e^2-K_{\omega}\cdot e.\]
We may also call $I(e)$ the {\bf SW index} of the class $e$. There is a principal $S^1$-bundle $\mathcal{M}^0_X(\mathcal{L},g,\eta)$ over $\mathcal{M}_X(\mathcal{L},g,\eta)$ if we take the quotient by only elements in $C^{\infty}(M;S^1)$ mapping a base point of $X$ to $1\in S^1$. Now we can get a number by pairing the maximal cup product of Euler class of the principal $S^1$-bundle with the fundamental class of $\mathcal{M}_X(\mathcal{L},g,\eta)$. 

When $b_2^+(X)=1$, the number defined above depends on the chamber where the pair $(g,\eta)$ live. More precisely, let $\omega_g$ denote the $g$-harmonic self-dual $2$-form. Then it depends on the sign of the discriminant 
\[\Delta_{\mathcal{L}}(g,\eta):=\int_X (2\pi c_1(\mathcal{L})+\eta)\wedge \omega_g.\]
As a result, there are two maps
\[SW_{\omega,\pm}:H^2(X;\ZZ)\rightarrow\ZZ\]
where the $\pm$ in the subscript indicates the sign.

To relate the Seiberg-Witten invariants with $J$-holomorphic curves, we consider the invariants defined using the the pair $(g,\eta)$ with negative discriminant. After identifying $H_2(X;\ZZ)$ with $H^2(X;\ZZ)$ by Poincar\'e duality, we will just write the map
\[SW:H_2(X;\ZZ)\rightarrow \ZZ\]
to denote $SW_{\omega,-}$ for convenience throughout our discussion.

We will be mostly concerned with rational or ruled manifolds in this paper. When $X$ is a ruled manifold, denote by $g$ the genus of the base surface which is equal to $\frac{1}{2}b_1(X)$. If $g>0$ (irrational ruled), there will be a distinguished class $F\in H_2(X;\ZZ)$ which is the class of the fiber sphere. The key properties about Seiberg-Witten invariants we will frequently use are listed below.

\begin{theorem}[\cite{LiLiuWallcrossing,LiLiuIMRN}]\label{thm:liliu}
    Let $(X,\omega)$ be a closed symplectic $4$-manifold with $b_2^+(X)=1$ and $A\in H^2(X;\mathbb{Z})$ with $I(A)\geq 0$, then
    \begin{itemize}
        \item (Symmetry lemma) $|SW_{\omega,+}(A)|=|SW_{\omega,-}(K_{\omega}-A)|.$
        \item (Wall crossing formula) $|SW_{\omega,+}(A)-SW_{\omega,-}(A)|=\begin{cases}
            1, X\text{ is rational,}\\
            |1+A\cdot F|^g, X\text{ is irrational ruled.}
        \end{cases}$
        \item (Blowup formula) Let $(\tilde{X},\tilde{\omega})$ be the symplectic blowup of $(X,\omega)$ with exceptional class $E$, then $SW_{\tilde{\omega},-}(A+lE)=SW_{\omega,-}(A)$ when $I(A+lE)\geq0$.
    \end{itemize}
\end{theorem}

We now provide a brief overview of how Taubes counts pseudo-holomorphic curves on the symplectic 4-manifold $(X,\omega)$, following the exposition of \cite{McDuffGT}. Let $A\in H_2(X;\ZZ)$ be a non-zero class with $I(A)\geq 0$.
By choosing a generic tame almost complex structure $J$ and $\frac{1}{2}I(A)$ generic points on $X$, the Gromov-Taubes invariant $Gr_{\omega}(A)$ is defined by a delicate weighted counting of the moduli space $\mathcal{H}_J(A)$ consisting of $\{(C_i,m_i)\}$ such that
\begin{itemize}
    \item $C_i$'s are disjoint embedded connected $J$-holomorphic submanifolds, $m_i$'s are positive integer numbers, $m_i=1$ unless $C_i$ is a torus with trivial normal bundle.
    \item $\sum m_i[C_i]=A$, $I([C_i])\geq 0$ for all $i$, each $C_i$ passes through $\frac{1}{2}I([C_i])$ chosen points. 
\end{itemize}

 An immediate observation by adjunction formula and index constraint is that when $[C_i]^2<0$, then $C_i$ must be an exceptional sphere. Now if one further defines $Gr_{\omega}$ to be $1$ for the zero class and $0$ for the classes with negative index, there will be a well-defined map
 $$Gr_{\omega}:H_2(X;\ZZ)\rightarrow \ZZ.$$

 When $b_2^+(X)>1$, Taubes' work identifies $Gr_{\omega}$ with $SW$. This implies that if $SW(A)\neq 0$ and $A\neq 0$, then $A$ can be represented by $J$-holomorphic submanifolds for generic $J$. When $b_2^+(X)=1$, such as rational or ruled manifolds mainly considered in this paper, one has to take the modified version of Gromov-Taubes invariants $Gr'_{\omega}$, defined in \cite{McDuffGT}, by discarding multiply-covered exceptional spheres so as to establish the identification with $SW$ (see \cite{LiLiuIMRN}). Since we will primarily work in relative settings and take a divisor-adapted almost complex structure, which may not be generic for defining Gromov-Taubes invariants, it is necessary to introduce the following notions.

\begin{definition}\label{def:Jsubvariety}
    An {\bf irreducible $J$-holomorphic subvariety} is a closed subset $C\subset M$ such that
    \begin{itemize}
        \item its $2$-dimensional Hausdorff measure is finite and non-zero;
        \item it has no isolated points;
        \item away from finitely many singular points, $C$ is a connected smooth submanifold with $J$-invariant tangent space.
    \end{itemize}
    A {\bf $J$-holomorphic subvariety} $\Theta$ is a finite set of pairs $\{(C_i,m_i)\}$ such that each $C_i$ is an irreducible $J$-holomorphic subvariety, $m_i$ is a positive integer and $C_i\neq C_j$ for $i\neq j$.
\end{definition}

Every irreducible $J$-holomorphic subvariety is the image of a $J$-holomorphic map $\phi:\Sigma\rightarrow X$ from a connected Riemann surface $\Sigma$. Therefore we can define the class of the subvariety $[C]$ as $\phi_*([\Sigma])$. For a class $A\in H_2(X;\ZZ)$, we let the moduli space $\mathcal{M}_A$ be the space of subvarieties $\Theta=\{(C_i,m_i)\}$ such that $[\Theta]:=\sum m_i[C_i]=A$, which can be naturally equipped with a topology in the Gromov-Hausdorff sense. We will also write $\mathcal{M}_{X,A}^J$ if we want to address the ambient manifold and the almost complex structure. Elements of $\mathcal{M}_A$ should be thought as unparametrized curves, while the moduli space parametrized curves will be denoted by the notation $\mathcal{M}(A;J)$ (following \cite{MSJcurve}) in Section \ref{section:ruled}.

\begin{definition}
    A class $A\in H_{2}(X;\ZZ)$ is called {\bf $J$-effective} if $\mathcal{M}_A$ is non-empty; is called {\bf $J$-nef} if its intersection pairings with all $J$-effective classes are non-negative.
\end{definition}
The key fact we will frequently use is that the non-vanishing of Seiberg-Witten invariants will still imply that the class is $J$-effective, though may not be represented by $J$-holomorphic submanifolds. To be more precise, if $SW(A)\neq 0$ and $A\neq 0$, then for {\bf any} tame almost complex structure $J$, there exists a $J$-holomorphic subvariety passing through any $\frac{1}{2}I(A)$ given points on $X$. As a result, if the class $A$ has $\omega(A)<0$, then we know $SW_{\omega,-}(A)=0$ since $A$ has no $J$-holomorphic representative. 

\begin{corollary}\label{cor:SWnonzero}
Let $(X,\omega)$ be a symplectic rational or ruled manifold. For any class $A\in H_2(X;\ZZ)$ with $I(A)\geq 0$, $[\omega]\cdot(K_{\omega}-A)<0$ and $A\cdot F\neq -1$ if $X$ is irrational ruled, we have $SW(A)\neq 0$. In particular, the Seiberg-Witten invariants of the classes represented by embedded symplectic exceptional spheres and the fiber class $F$ of ruled manifolds are non-zero.
\end{corollary}
\begin{proof}
   The condition $[\omega]\cdot(K_{\omega}-A)<0$ indicates that $SW_{\omega,-}(K_\omega-A)=0$. By the symmetry lemma and the wall crossing formula in Theorem \ref{thm:liliu}, it follows immediately that $A$ has non-vanishing Seiberg-Witten invariant. If $A$ is a symplectic exceptional class or fiber class, it is straightforward to check $I(A)\geq 0$ and $A\cdot F\neq -1$. Moreover, one can choose a symplectic structure $\omega$ with small blowup sizes to make $[\omega]\cdot K_\omega<0$ while keeping $\omega(A)>0$, which guarantees $[\omega]\cdot(K_{\omega}+A)<0$. Thus, $SW(A)\neq 0$ follows directly.
\end{proof}

 We have another straightforward corollary concerning the genus of a class in the following.

\begin{corollary}\label{cor:SWgenus}
    Let $(X,\omega)$ be a symplectic rational or ruled manifold. Suppose $A\in H_2(X;\ZZ)$ is a class with $[\omega]\cdot A>0$, $[\omega]\cdot (K_{\omega}+A)<0$ and $A\cdot F\neq 1$ if $X$ is irrational ruled, then $A\cdot (A+K_{\omega})<0$.
\end{corollary}

\begin{proof}
    The key observation is that $A\cdot (A+K_{\omega})$ is the SW index for both $-A$ and $K_{\omega}+A$. If $A\cdot (A+K_{\omega})\geq 0$, apply Theorem \ref{thm:liliu} to the classes $-A$ and $K_{\omega}+A$ to see that one of them must have non-trivial Seiberg-Witten invariant. But this contradicts with the assumption that both of them have negative pairing with $\omega$.
\end{proof}

\subsection{Divisor-adapted almost complex structures and McDuff-Opshtein's criterion}
In this section we follow \cite{MO15} to introduce the almost complex structures which we will work with for the symplectic divisor $D=\cup D_i\subseteq (X,\omega)$.
\begin{definition}
    An $\omega$-tame almost complex structure $J$ is said to be \textbf{$D$-adpated} if there exists some plumbed closed fibered neighborhood $\overline{\mathcal{N}}(D)=\cup\overline{\mathcal{N}}(D_i)$, where each $\overline{\mathcal{N}}(D_i)$ is a closed neighborhood of $D_i$ modeled on a closed neighborhood of the zero section in a holomorphic line bundle over $D_i$ with Chern number $[D_i]^2$, such that $J$ is integrable in $\overline{\mathcal{N}}(D)$ and makes both $D_i$ and the projection $\overline{\mathcal{N}}(D_i)\rightarrow D_i$ be $J$-holomorphic. The collection of all such $J$'s will be denoted by $\mathcal{J}(D)$.
\end{definition}

\begin{definition}\label{def:Dgood}
    Let $D\subseteq (X,\omega)$ be a symplectic divisor. A nonzero class $A\in H_2(X;\ZZ)$ is said to be \textbf{$D$-good} if the followings are satisfied:
    \begin{enumerate}
        \item $SW(A)\neq 0$;
        \item if $A^2=0$, then $A$ is primitive;
        \item $A\cdot E\geq 0$ for any class $E$ which is not equal to $A$ and can be represented by an embedded symplectic exceptional sphere;
        \item $A\cdot[D_i]\geq 0$ for all $i$.
    \end{enumerate}
\end{definition}

Recall that once we have fixed the symplectic canonical class $K$, there will be an identification between $H_2(X;\ZZ)$ and spin$^c$-structures on $X$ and thus $SW$ is a well-defined function on $H_2(X;\ZZ)$. Also, the classes represented by embedded symplectic exceptional spheres only depend on the symplectic canonical class $K$ when $b_2^+=1$ by \cite[Theorem A]{LiLiuruled} (and only depend on the deformation equivalent class of symplectic forms for all symplectic $4$-manifolds, see for example \cite[Theorem B]{Wendlbook}). Consequently, for $b_2^+=1$ manifold $X$, the notion of $D$-goodness actually only requires a topological divisor $D\subseteq X$ and the choice of a symplectic canonical class. When $(X',D')$ is the topological model of the blowup of $(X,\omega,D)$ with exceptional class $E$, we can still talk about the notion of $D'$-goodness by choosing the canonical class $K_\omega+E\in H_2(X';\ZZ)$ without specifying a particular symplectic form on $X'$. In this sense, we have the following lemma of $D'$-goodness regarding the blowup.

\begin{lemma}\label{lem:criterionforDgood}
    Let $(X',D')$ be the topological model for the blowup of $(X,\omega,D)$ with $b_2^+(X)=1$ and $E$ be the exceptional class. If $A\in H_2(X;\ZZ)$ is $D$-good for $(X,\omega,D)$ and $A^2\geq 0$, then $A$ is also $D'$-good when viewed as a class in $H_2(X';\ZZ)=H_2(X;\ZZ)\oplus \ZZ E$.
\end{lemma}

\begin{proof}
Conditions (2) and (4) in Definition \ref{def:Dgood} is obvious. (1) follows from the blowup formula in Theorem \ref{thm:liliu}. To see (3), observe that the exceptional classes $B+kE$ of $X'$ where $B\in H_2(X;\ZZ)$ satisfy either $k=0$ or $B^2\geq 0$. Again, by blowup formula in Theorem \ref{thm:liliu}, we know that $SW_{\omega,-}(B)=SW_{\tilde{\omega},-}(B+kE)\neq 0$, which implies $B\cdot[\omega]>0$. The non-negative intersection property then immediately follows from the light cone lemma since we also have $A^2\geq 0$, $A\cdot [\omega]>0$ and $b_2^+=1$.
\end{proof}

To show the $D$-goodness for some special classes, it follows from the lemma below that we only need to check condition (4) in Definition \ref{def:Dgood}.

\begin{lemma}\label{lem:specialclassDgood}
    Let $D\subseteq(X,\omega)$ be a symplectic divisor. Suppose $[\omega]\cdot K_{\omega}<0$. If $A\in H_2(X;\ZZ)$ is a primitive class represented by an embedded symplectic sphere with $A^2\geq -1$, then $A$ is $D$-good if and only if $A\cdot[D_i]\geq 0$ for all $i$.
\end{lemma}

\begin{proof}
    By assumptions, we only need to check conditions (1) and (3). For condition (1), note that since $[\omega]\cdot K_\omega<0$, by Theorem \ref{thm:rationalruled}, $X$ is a rational or ruled manifold. In particular, $b_2^+(X)=1$. By light cone lemma, if $A^2\geq 0$, then $A\cdot F\neq -1$ for the fiber class $F$. By adjunction formula, when $A^2\geq -1$ we have \[I(A)=A^2-K_\omega\cdot A=2A^2-(A^2+K_\omega\cdot A)=2A^2+2\geq 0.\]
    Therefore, $SW(A)\neq 0$ by Corollary \ref{cor:SWnonzero}. For condition (3), assume $E$ is a symplectic exceptional class. When $A^2\geq 0$, we can choose tame almost complex structure $J$ making the symplectic sphere in class $A$ become $J$-holomorphic. By positivity of intersection, $A$ must be $J$-nef. Since $SW(E)\neq 0$, $E$ is $J$-effective so that we have $A\cdot E\geq 0$. If $A^2=-1$ and $A\neq E$, by symmetry in $A$ and $E$, we may further assume $\omega(A)\geq \omega(E)$ without loss of generality. We still take an embedded $J$-holomorphic sphere $C$ in class $A$ and the $J$-holomorphic subvariety $\Theta$ in class $E$. By symplectic area consideration, $C$ can not appear in the component of $\Theta$. Again, by positivity of intersection, we see that $A\cdot E\geq 0$.
\end{proof}

Finally, we need the following crucial result by McDuff-Opshtein which indicates that any $D$-good spherical class or exceptional class has an embedded $J$-holomorphic representative.

\begin{theorem}[\cite{MO15}, Theorem 1.2.7]\label{thm:MO15}
    If $A\in H_2(X;\ZZ)$ is a $D$-good class such that $A^2+K_{\omega}\cdot A=-2$ or $A$ is the class of an embedded symplectic exceptional sphere, then there is a residual set $\mathcal{J}_{\text{emb}}(D,A)\subseteq \mathcal{J}(D)$ such that for any $J\in \mathcal{J}_{\text{emb}}(D,A)$, $A$ can be represented by an embedded $J$-holomorphic sphere. 
\end{theorem}

\subsection{$J$-holomorphic curves in tamed almost complex $4$-manifolds}
In this section, we review several results from Li-Zhang \cite{LiZhangnef} and Zhang \cite{Zhangcurvecone,Zhangmoduli} which will be used later. The first one concerns the configuration of reducible $J$-holomorphic subvariety in a $J$-nef spherical class.

\begin{theorem}[\cite{LiZhangnef}, Theorem 1.5]\label{thm:lizhangnef}
     Let $J$ be any $\omega$-tame almost complex structure on a symplectic $4$-manifold $(X,\omega)$. Suppose $A\in H_2(X;\ZZ)$ is $J$-nef and satisfies $A^2+K_{\omega}\cdot A=-2$. Then any $J$-holomorphic subvariety in class $A$ must have a connected tree configuration whose components are all embedded spheres.
\end{theorem}

The next result, known as the curve cone theorem, is the following almost complex analogue of the celebrated Mori's cone theorem in algebraic geometry.

\begin{theorem}[\cite{Zhangcurvecone}, Theorem 1.1]\label{thm:curvecone}
    Let $(X,\omega)$ be a non-minimal symplectic 4-manifold not diffeomorphic to $\CC\PP^2\#\overline{\CC\PP}^2$ and $J$ an $\omega$-tame almost complex structure. If $A\in H_2(X;\ZZ)$ is $J$-effective, then there exists irreducible $J$-holomorphic subvarieties within the classes $F_1,\cdots,F_m$ satisfying $F_i\cdot K_J\geq 0$ for all $1\leq i\leq m$ and embedded $J$-holomorphic exceptional spheres within the classes $E_1,\cdots,E_m$, and positive real numbers $a_i,b_j$'s, such that 
    \[A=\sum_{i=1}^ma_i F_i+\sum_{j=1}^nb_j E_j.\]
\end{theorem}

We also need results from Section 3 in \cite{Zhangcurvecone} studying curves on $\CC\PP^2\#2\overline{\CC\PP}^2$.

\begin{theorem}[\cite{Zhangcurvecone}]\label{thm:2blowup}
    For any tame almost complex structure $J$ on $\CC\PP^2\#2\overline{\CC\PP}^2$, there exists at least two embedded $J$-holomorphic exceptional spheres $C_1,C_2$. They satisfy either $[C_1]\cdot[C_2]=0$ or $[C_1]\cdot[C_2]=1$.
\end{theorem}

Finally, for irrational ruled manifolds, \cite[Section 3]{Zhangmoduli} obtained the following result concerning curves in the fiber class.

\begin{theorem}[\cite{Zhangmoduli}]\label{thm:fiberclass}
   Let $J$ be any tame almost complex structure on an irrational ruled manifold $X$ and $F$ be the fiber class. Then 
   \begin{itemize}
       \item $F$ is $J$-nef;
       \item there exists an embedded $J$-holomorphic curve $C$ of genus $\frac{1}{2}b_1(X)$ which can be identified with the moduli space of $J$-holomorphic subvarieties in class $F$, with finitely many points representing reducible subvarieties;
       \item any embedded $J$-holomorphic sphere is an irreducible component of an element of $\mathcal{M}_F$.
   \end{itemize}
\end{theorem}

 We remark that all the results reviewed in this section do not require the genericity of the almost complex structure, which makes them particularly useful in practice. Note that when the divisor contains some component with negative index, any divisor-adapted almost complex structure cannot be generic in the Gromov-Taubes sense.

\section{Cuspidal curves and their normal crossing resolutions}\label{section:cuspcurve}

In this section we give an introduction to cuspidal curves in symplectic $4$-manifolds, following the setup in the work by McDuff-Siegel \cite{MSiegeltangency,mcduffsiegelellipsoidalsuperpotentialssingularcurve}.

For a symplectic divisor $D$ in $(X,\omega)$, some $J\in \mathcal{J}(D)$ and a point $x\in D_1$ where $D_1$ is some component of $D$, we can choose a small neighborhood $\mathcal{O}p(x)$ of $x$ such that $J$ is integrable in $\mathcal{O}p(x)$. Let $u:\Sigma\rightarrow X$ be a $J$-holomorphic map from a Riemann surface $\Sigma$ with a marked point $z$ such that $u(z)=x$. We say that $u$ has \textbf{tangency order} $m-1$ (or equivalently \textbf{contact order} $m$) to $(D_1,x)$ at the marked point $z$ if we have 
\[\frac{d^j(g\circ u\circ f)}{d\zeta^j}|_{\zeta=0}=0\,\,\,\,\,\,\text{for }j=1,\cdots,m-1,\]
where $f:\mathbb{C}\supseteq \mathcal{O}p(0)\rightarrow\mathcal{O}p(z)\subseteq \Sigma$ is a choice of local complex coordinates for $\Sigma$ with $f(0)=z$, and $g:X\supseteq \mathcal{O}p(x)\rightarrow \mathbb{C}$ is a holomorphic function such that $D_1\cap \mathcal{O}p(x)=g^{-1}(0)$ and $dg(x)\neq 0$. Assuming $m$ is maximal such that $u$ is tangent to $(D_1,x)$ to order $m-1$, we will denote by
\[\text{ord}(u,D_1;z)=m\]
the \textbf{local contact order} of $u$ to $(D_1,x)$ at $z$. If $x$ is the intersection point of two distinct components $D_1,D_2$ of $D$ and $p,q$ are two positive integers, we denote by \[\mathcal{M}^J_{X,A}\bigl\langle\mathcal{C}_{D_1,D_2}^{p,q}x\bigr\rangle,\] 
a subspace of $\mathcal{M}_{X,A}^J$ ,the moduli space of $J$-holomorphic subvarieties which is the image of a $J$-holomorphic map $u:\mathbb{CP}^1\rightarrow X$ such that $u_*([\CC\PP^1])=A\in H_2(X;\ZZ)$, $u([0:0:1])=x$ and $u$ has local contact order $p$ at $(D_1,x)$ and $q$ at $(D_2,x)$. When $p>q$ and $\text{gcd}(p,q)=1$, this local multidirectional tangency condition is closely related to the $(p,q)$-cuspidal singularity. Let $C\subseteq \CC^2$ be an algebraic curve. A singular point $x\in C$ is a \textbf{$(p,q)$-cusp} if its link is the $(p,q)$-torus knot, i.e. if for $\varepsilon>0$ sufficiently small there is a
diffeomorphism
\[(S_{\varepsilon}^3,C\cap S_{\varepsilon}^3)\cong (S^3,\{z_1^p+z_2^q=0\}\cap S^3),\]
where $S_{\varepsilon}^3\subseteq \CC^2$ is the sphere of radius $\varepsilon$ centered at $x$. If $u$ is an injective $J$-holomorphic curve parametrizing a subvariety in $\mathcal{M}^J_{X,A}\bigl\langle\mathcal{C}_{D_1,D_2}^{p,q}x\bigr\rangle$ with only one singularity at $[0:0:1]$, then we say that $u$ is \textbf{$(p,q)$-unicuspidal}. \cite[Lemma 3.5.3]{mcduffsiegelellipsoidalsuperpotentialssingularcurve} shows that the tangency condition $\bigl\langle\mathcal{C}_{D_1,D_2}^{p,q}x\bigr\rangle$ implies that in the local holomorphic chart of the singularity, the image of $u$ is a $(p,q)$-cusp. Let us remark that when $q=1$, while $u$ might be a smooth curve with $\text{ord}(u,D_1;x)=p$, it will still be called a $(p,1)$-unicuspidal curve for convenience in subsequent statements\footnote{For example, a smooth curve passing through $x$ and transverse to both $D_1,D_2$ will be called $(1,1)$-unicuspidal.}. 


Now we discuss the {\bf normal crossing resolution} for such a $(p,q)$-unicuspidal $J$-holomorphic curve $u$. Denote by $C$ the image of $u$. Since $J$ is integrable near the divisor $D$, we can perform the blowup construction for complex manifolds by removing the point $x$ and replacing it with the space of complex lines in $T_xX$ to get the blowup manifold $X_1\cong X\#\overline{\CC\PP}^2$ equipped with an almost complex structure $J_1$ and a holomorphic map $\pi:(X_1,J_1)\rightarrow (X,J)$. It is well known that the blowup of a K\"ahler manifold is still K\"ahler (see for example \cite[Proposition 3.24]{Voisinbook}), $J_1$ must be tamed with some symplectic structure on $X_1$ since the blowup construction is local in nature. The total transform $\pi^{-1}(D)\subseteq X_1$ is a union of proper transform $\cup \pi^{-1}(D_i)$ and the exceptional curve $e$. Let $C_1\subseteq X_1$ be the proper transforms $\pi^{-1}(C)$. Since we assume $p>q$, $C_1$ is disjoint from $\pi^{-1}(D_2)$ but has contact order $p-q$ with $\pi^{-1}(D_1)$ and contact order $q$ with the exceptional curve $e$. Thus $C_1$ is a $J_1$-holomorphic curve satisfying the constraint $\bigl\langle\mathcal{C}_{\pi^{-1}(D_1),e}^{p-q,q}x_1\bigr\rangle$ where $x_1$ is the intersection between $\pi^{-1}(D_1)$ and $e$. Since $\text{gcd}(p-q,q)=1$, we can continue to blowup at $x_1$ and obtain the proper transform $C_2$ in $X_2=X\#2\overline{\CC\PP}^2$. Therefore, there must be some $k\in\ZZ_+$ such that the proper transform $C_k$ along with the total transform of $D$ gives a normal crossing divisor in $X_k$.
Such a smooth $J_k$-holomorphic curve $C_k$ is called the {\bf normal crossing resolution} of $C$. It intersects exactly one component, the exceptional curve that comes from the $k$-th blowup, of the total transform of $D$.

Given relatively prime numbers $p,q\in\ZZ_+$, there is an associated {\bf weight sequence} $\mathcal{W}(p,q):=(m_1,\cdots,m_k)$ defined as follows. Firstly, let $(p_1,q_1):=(p,q)$. Assuming $(p_n,q_n)$ is defined and $p_n\neq q_n$, then $(p_{n+1},q_{n+1})$ is defined as $(|p_n-q_n|,\min\{p_n,q_n\})$. Since $p,q$ are relatively prime, there exists $k\in\ZZ_+$ such that $(p_k,q_k)=(1,1)$. Then the number $m_i$ in the weight sequence is defined as $\min\{p_i,q_i\}$. For example, we have $\mathcal{W}(5,2)=(2,2,1,1)$.
Another way to introduce the weight sequence is through the box diagram in \cite{mcduffsiegelellipsoidalsuperpotentialssingularcurve}. From the box diagram, it is easy to obtain the following facts.

\begin{fact}\label{fact:boxdiagram}
    $pq=\sum_{i=1}^km_i^2$, $p+q=\sum_{i=1}^km_i+1$.
\end{fact}
\begin{proof}
    Following \cite{mcduffsiegelellipsoidalsuperpotentialssingularcurve}, associate to a $(p,q)$-cusp a $p\times q$ rectangle. The weight sequence corresponds to a decomposition of this rectangle into squares of side lengths $m_1,\dots,m_k$. The identity $pq=\sum_{i=1}^k m_i^2$ is obtained by comparing areas. For the second identity, one considers the half-perimeter contributions of the squares along the boundary of the box diagram, which yields $p+q=\sum_{i=1}^k m_i+1$ (the last square of side length $m_k=1$ contributes twice).
\end{proof} 

The importance of weight sequence is that it gives the homology class of the normal crossing resolution of a $(p,q)$-cuspidal curve. In fact, if $C\subseteq X$ has a $(p,q)$-cusp and $\tilde{C}\subseteq X_k$ is its normal crossing resolution, then we have \[[\tilde{C}]=[C]-\sum_{i=1}^km_iE_i,\]
where $E_i$'s are the exceptional class and we implicitly include $H_2(X;\ZZ)$ into $H_2(X_k;\ZZ)$. See Figure \ref{fig:resolution} for an example of a $(5,2)$-cusp.

\begin{figure}[ht]
		\centering\includegraphics*[height=4cm, width=14cm]{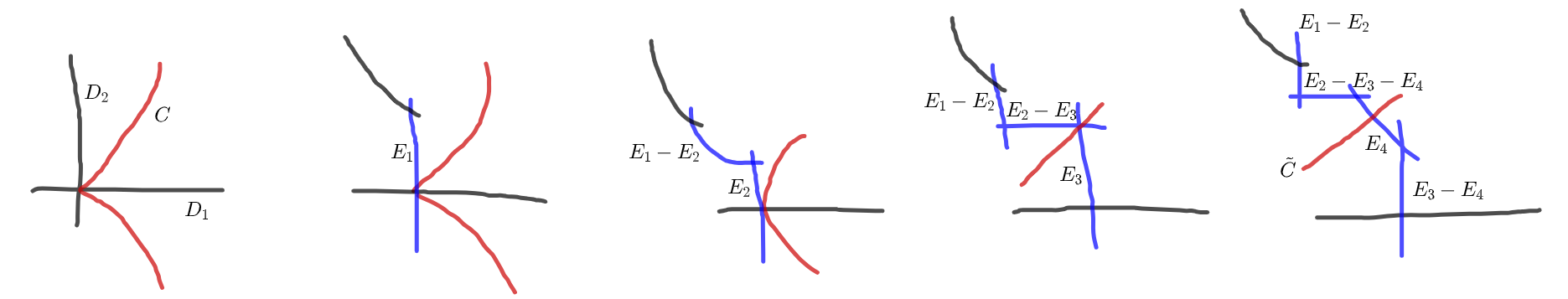}
 
		\caption{The normal crossing resolution for a (5,2)-cusp. The intersection pattern implies the stated homology class relation. \label{fig:resolution}}
	\end{figure}
    
We will also need the following simple lemma.

\begin{lemma}\label{lem:positivecombination}
     Let $p,q\in\ZZ_+$ be relatively prime numbers with weight sequence $(m_1,\cdots,m_k)$ and $\tilde{D}$ the proper transform in $X_k$ of the irreducible smooth divisor $D\subseteq X$ under the normal crossing resolution of a $(p,q)$-cusp with contact order $p$ at $D$. Then we can write \[q([D]-[\tilde{D}])-\sum_{i=1}^km_iE_i\] as a linear combination of the components of the total transform of $D$ with all coefficients being non-negative.
\end{lemma}

\begin{proof}
    We prove by induction on $k$. When $k=1$ we must have $m_1=p=q=1$ and thus $q([D]-[\tilde{D}])-m_1E_1$ is just the zero class. For $k>1$, we can reduce to the case for $k-1$ by considering the proper transform $\tilde{D}_{(1)}$ of $D$ in the one-time blowup $X_1$. If $p>q$, we can split $q([D]-[\tilde{D}])$ as the sum of $q([D]-[\tilde{D}_{(1)}])$ and $q([\tilde{D}_{(1)}]-[\tilde{D}])$. Then on the one hand, by induction $q([\tilde{D}_{(1)}]-[\tilde{D}])-\sum_{i=2}^km_iE_i$ is a non-negative linear combination of the components of the total transform of $\tilde{D}_{(1)}$ which is part of the total transform of $D$, since we have a $(p-q,q)$-cusp with contact order $p-q$ at $\tilde{D}_{(1)}$; on the other hand, $q([D]-[\tilde{D}_{(1)}])=qE_1=m_1E_1$. If $p<q$, then $q([D]-[\tilde{D}])=qE_1=m_1E_1+(q-p)E_1$. Since we have a $(p,q-p)$-cusp with contact order $p$ at the exceptional curve $e_1$ of $X_1$ in class $E_1$, by induction $(q-p)E_1-(q-p)\tilde{E}_1-\sum_{i=2}^km_iE_i$ is a non-negative linear combination of components of the total transform of $e_1$ which is part of the total transform of $D$, where $\tilde{E}_{1}$ denotes the homology class of proper transform $\tilde{e}_1$ of $e_1\subseteq X_1$ in $X_k$. Thus, $q([D]-[\tilde{D}])$ still satisfies our requirement.
\end{proof}

\begin{example}
    Consider $(p,q)=(5,2)$ as shown in Figure \ref{fig:resolution}. The class of $C$ is changed to $[C]-2E_1-2E_2-E_3-E_4$ after resolution. Then $[\tilde{D}_1]=[D_1]-E_1-E_2-E_3$. We can then write \[2([D_1]-[\tilde{D}_1])-2E_1-2E_2-E_3-E_4=E_3-E_4,\] where $E_3-E_4$ appears in the components in the total transform of $D$. Also, we have $[\tilde{D}_2]=[D_2]-E_1$ and similarly one can verify that
\[
5([D_2]-[\tilde{D}_2])-2E_1-2E_2-E_3-E_4 = 3E_1-2E_2-E_3-E_4 = 3(E_1-E_2)+(E_2-E_3-E_4).
\]
\end{example}

The following relationship between the moduli space of curves in $X$ of class $A\in H_2(X;\ZZ)$ with tangency conditions and curves in $X_k$ of class $\tilde{A}:=A-\sum_{i=1}^km_iE_i\in H_2(X_k;\ZZ)$ was pointed out in \cite[Proposition 4.3.1]{mcduffsiegelellipsoidalsuperpotentialssingularcurve}.

\begin{proposition}\label{prop:modulirelation}
    There is a natural bijective correspondence 
    \[\mathcal{M}^J_{X,A}\bigl\langle\mathcal{C}_{D_1,D_2}^{p,q}x\bigr\rangle\cong \mathcal{M}_{X_k,\tilde{A}}^{J_k,\text{emb}},\]
    where $\mathcal{M}_{X_k,\tilde{A}}^{J_k,\text{emb}}\subseteq\mathcal{M}_{X_k,\tilde{A}}^{J_k}$ denotes the moduli space of embedded (unparametrized) $J_k$-holomorphic spheres of class $\tilde{A}$.
\end{proposition}

\section{Irrational ruled manifolds}\label{section:ruled}

Let $(X,\omega)$ a symplectic ruled manifold with genus $g:=\frac{1}{2}b_1(X)$. When $g\geq 1$, we say $(X,\omega)$ is irrational ruled. The diffeomorphism type of $X$ is either a twisted $S^2$-bundle over the genus $g$ Riemann surface $S^2\tilde{\times}\Sigma_g$, or the blowup of the trivial $S^2$-bundle $(S^2\times \Sigma_g)\#n\overline{\CC\PP}^2$. In the former case, we choose the basis
\[\{B_1,F\}\subseteq H_2(X;\ZZ),\]
where $B_1$ is the class of a section with self-intersection one of the twisted $S^2$-bundle and $F$ is the class of the fiber $S^2$. By \cite[Theorem 1]{LiLiucone}, up to a diffeomorphism, the symplectic canonical class can always be assumed to be the standard one \[K_{\omega}=-2B_1+(2g-1)F.\]
In the latter case, we choose the basis 
\[\{B,F,E_1,\cdots,E_n\}\subseteq H_2(X;\mathbb{Z}),\]
where $B$ is the class of a section with self-intersection zero, $F$ is the class of the fiber $S^2$ and $E_i$'s are exceptional classes. Then we may assume the symplectic canonical class is the standard one
\[K_{\omega}=-2B+(2g-2)F+\sum_{i=1}^n E_i.\] 

The following definition will be used both in this section, where $g\geq 1$, and in the next section, where $g=0$.

\begin{definition}\label{def:comblikeruled}
   A symplectic divisor $D$ in a ruled manifold with $g\geq 0$ is called {\bf comb-like} if, after adding finitely many components, it extends to a divisor $\tilde{D}$ obtained from an initial divisor in $\Sigma_g\times S^2$ (or $\Sigma_g\tilde{\times} S^2$) by a sequence of toric/non-toric/half-toric/exterior blowups. Here the initial divisor consists of one component in the section class $B+kF$ (or $B_1+kF$) and several components in the fiber class $F$. 
\end{definition}
  We refer to Figure \ref{fig:comb} for an illustrative example of a comb-like configuration. Now suppose $D=\cup_i D_i\subseteq (X,\omega)$ is a symplectic divisor with classes $[D_i]=a_iF+b_iB-\sum_{j=1}^nc_i^jE_j$ (or $[D_i]=a_iF+b_iB_1$ in the twisted $S^2$-bundle case). It turns out that the configuration of $D$ in an irrational ruled manifold must be comb-like under the assumption $[\omega]\cdot(K_{\omega}+[D])<0$ by the following characterization of the classes of components in $D$. We write $D=D'\cup D''$, where $D'$ consists of spherical components and $D''$ consists of components of higher genus. The notation $\sum E_j$ below denotes $\sum_{j\in \Lambda}E_j$ for some subset $\Lambda\subseteq \{1,2,\cdots,n\}$.

\begin{proposition}\label{prop:ruledconfig}
     Let $D\subseteq (X,\omega)$ be a symplectic divisor with $[\omega]\cdot (K_{\omega}+[D])<0$ where $X$ is an irrational ruled manifold. Then \begin{itemize}
        \item each spherical component $D_i$ of $D'$ has class $F-\sum E_j$ or $E_l-\sum E_j$;
        \item $D''$ is either empty or has exactly one component of genus $g$ within the class $B+kF-\sum E_j$ (or $B_1+kF$ in the twisted $S^2$-bundle case).
    \end{itemize}

\end{proposition}
   
\begin{proof}
  For any component $D_i$ of $D$, the adjunction formula gives 
 \begin{align}\label{ruledadj}
     g(D_i)&=\frac{[D_i]^2+K_{\omega}\cdot [D_i]}{2}+1\notag\\
     &=
     \begin{cases}
         (a_i-1)(b_i-1)+gb_i+\dfrac{b_i(b_i-1)}{2},
         & \text{if } [D_i]=a_iF+b_iB_1,\\[0.4em]
         (a_i-1)(b_i-1)+gb_i+\dfrac{1}{2}\displaystyle\sum_{j=1}^n(c_i^j-(c_i^j)^2),
         & \text{if } [D_i]=a_iF+b_iB-\displaystyle\sum_{j=1}^nc_i^jE_j.
     \end{cases}
 \end{align}
Now, choose any $J\in\mathcal{J}(D)$. By the $J$-nefness of the fiber class $F$ in Theorem \ref{thm:fiberclass}, the $b_i$ coefficient must be non-negative. Note that there is a projection map from the irrational ruled manifold $X$ onto $\Sigma_g$ which also provides a map from $D_i$ to $\Sigma_g$ of mapping degree $b_i$. By Kneser's theorem, when $g(D_i)=0$, the mapping degree $b_i$ must be $0$. Now we apply Theorem \ref{thm:fiberclass} again. The embedded $J$-holomorphic curve $C$ of genus $g$ in Theorem \ref{thm:fiberclass} (see also \cite[Proposition 3.6]{Zhangmoduli}) has class $B+aF-\sum E_j$ (or $B_1+aF$ for twisted $S^2$-bundle case) for some $a\in \mathbb{Z}$. If $b_i=0$, the positivity of intersection between $C$ and $D_i$ implies that $a_i\geq 0$. Therefore, from equation (\ref{ruledadj}) we see that either $a_i=1$ and $c_i^j\in\{0,1\}$ which corresponds to the class $F-\sum E_j$; or $a_i=0$, only one $c_i^l=-1$ and other $c_i^j\in\{0,1\}$ which corresponds to the class $E_l-\sum E_j$. This proves the first bullet.

For the second bullet, let us apply Corollary \ref{cor:SWgenus}. Since $[\omega]\cdot [D_i]>0$, $[\omega]\cdot(K_{\omega}+[D_i])<0$ and $[D_i]\cdot ([D_i]+K_\omega)\geq 0$ for a component $D_i$ with genus $\geq 1$, we must have $b_i=[D_i]\cdot F=1$. From equation (\ref{ruledadj}) we see that $g(D_i)\leq g$. Since we have degree $1$ map from $D_i$ to $\Sigma_g$, by Kneser's theorem, $g(D_i)$ must be equal to $g$ and thus $c_i^j\in\{0,1\}$ which gives the class $B+kF-\sum E_j$. To see there could not be two components $[D_i],[D_j]$ in $D''$, consider the class $A:=[D_i]+[D_j]$. Note that by Lemma \ref{lem:g(D)}, $A\cdot(A+K_\omega)\geq 0$. Then we can just apply Corollary \ref{cor:SWgenus} to the class $A$ to get $A\cdot F=1$. But $([D_i]+[D_j])\cdot F=1+1=2$, which is a contradiction.
\end{proof}

\begin{corollary}\label{cor:no3-Ej}
Under the assumption of Proposition \ref{prop:ruledconfig}, if $\omega(E_n)\leq \omega(E_i)$ for any $1\leq i\leq n$, then there are the following two cases:
\begin{enumerate}
       \item if there is exactly one component of $D$ in class $E_n$, then this component must be a toric/half-toric/exterior exceptional sphere with respect to $D$;
    \item if $D$ has no component in class $E_n$, then the intersection numbers between $E_n$ and the components in $D$ must be either $0$ or $1$, and at most two of them are $1$. In particular, $E_n$ must be $D$-good.
\end{enumerate}
As a consequence, $D$ must be comb-like.
\end{corollary}

\begin{proof}
By Proposition \ref{prop:ruledconfig}, every spherical component of $D$ has class either $F-\sum E_j$ or $E_l-\sum E_j$. In particular, a spherical component contains an $(-E_n)$-term in its class only if it is of the form $F-E_n-\sum_{j\in\Lambda\setminus{\{n\}}}E_j$ or $E_l-E_n-\sum_{j\in\Lambda\setminus{\{n\}}}E_j$. Since $\omega(E_n)\leq \omega(E_i)$ for all $i$, the symplectic area constraints exclude any component in class $E_n-\sum_{j\in\Lambda}E_j$ with nonempty $\Lambda$. Also, Proposition \ref{prop:ruledconfig} implies that there are at most two components of $D$ whose classes involve $-E_n$ since any two componenets of $D$ must have non-negative intersection number.

If $D$ has a component in class $E_n$, then this component is an exceptional sphere and necessarily meets $D$ in at most two points; hence it is exterior, half-toric, or toric. Otherwise, $D$ has no $E_n$-component. Then $E_n\cdot [D_i]\in\{0,1\}$ for every component $D_i$, and at most two of these intersection numbers are equal to $1$. In particular, $E_n$ has nonnegative intersection with every component of $D$, so $E_n$ is $D$-good.

Therefore, after possibly adding an $E_n$-component to $D$, we can always blow down in the class $E_n$ (toric, non-toric, half-toric, or exterior) and obtain a divisor $D_{n-1}$ in $(S^2\times\Sigma_g)\#(n-1)\overline{\CC\PP}^2$.
By construction, the components of $D_{n-1}$ also have homology classes of the type described in Proposition \ref{prop:ruledconfig}.
We may then repeat the same argument for the class $E_{n-1}$ to produce a divisor $D_{n-2}$ in $(S^2\times\Sigma_g)\#(n-2)\overline{\CC\PP}^2$, and so on.
Iterating, we arrive at a divisor $D_0$ in the minimal ruled surface $S^2\times\Sigma_g$.
Its components must represent either the section class $B+kF$ or the fiber class $F$, since their proper transforms must be among the classes listed in Proposition \ref{prop:ruledconfig}.
This proves that $D$ is comb-like.
\end{proof}

\begin{example}\label{example:comb}
    Figure \ref{fig:comb} gives an example of a divisor $D$ in $(S^2\times\Sigma_g)\#11\overline{\CC\PP}^2$ satisfying the conclusion of Proposition \ref{prop:ruledconfig}. The configuration has exactly one component of genus $g\geq 1$, which is in class $B-2F-E_5$ and called the section-component. The other components are either the spherical fibers or contained in some nodal representative of the fiber class, called the fiber-components. Such a curve configuration will be called {\bf comb-like}. We will also see some comb-like configurations in rational manifolds in Section \ref{section:rational}, where the section-component will have genus $0$.
\end{example}

\begin{figure}[ht]
		\centering\includegraphics*[height=5.5cm, width=14cm]{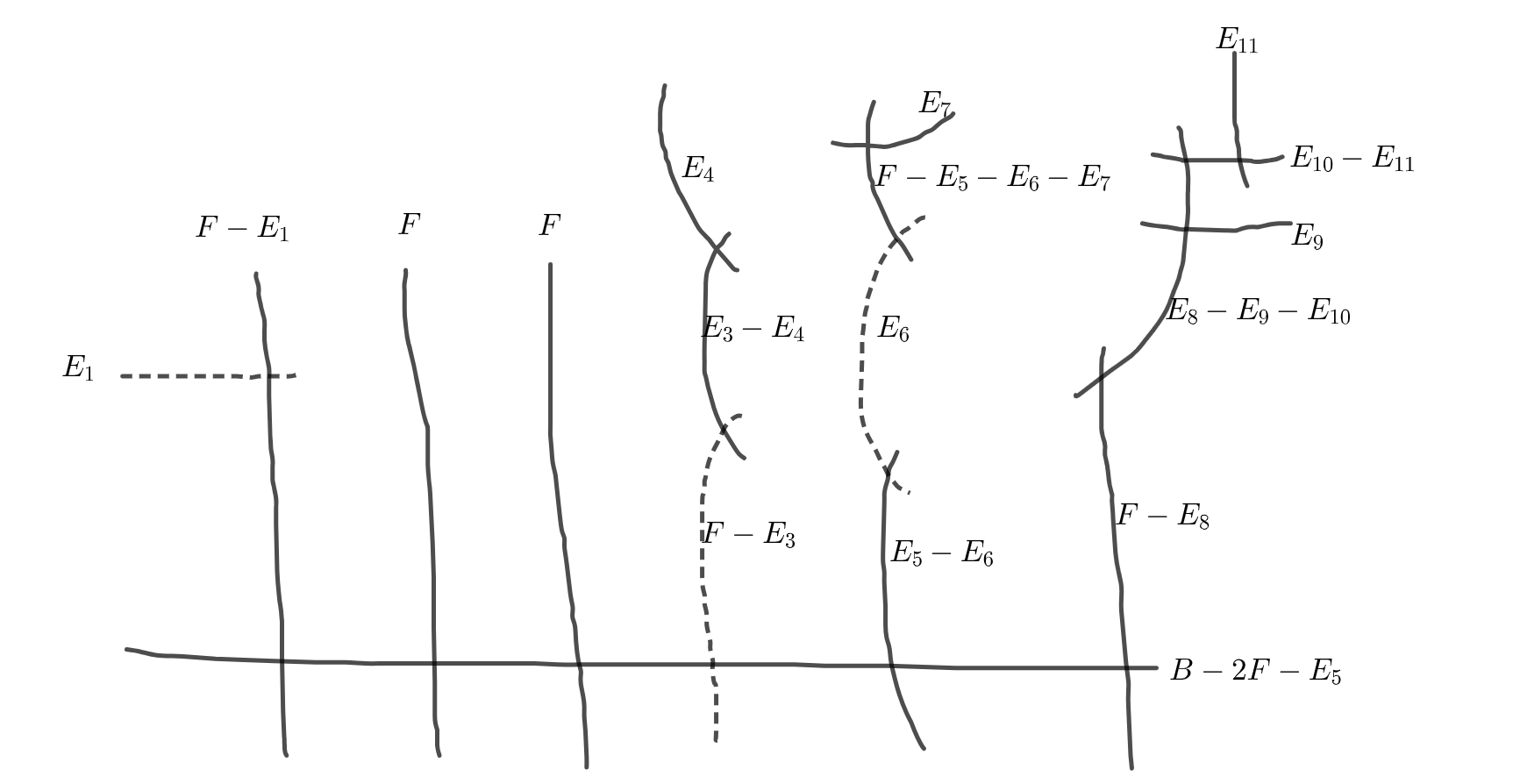}
 
		\caption{A comb-like configuration, where the dashed curves are not included in $D$. After adding components in classes $F-E_3$ and $E_6$, it can be obtained by the blowups from the divisor in $S^2\times \Sigma_g$ with one section-component in class $B-2F$ and six fiber-components.\label{fig:comb}}
	\end{figure}

Now we can prove our main theorem for the irrational ruled case by a modification of the argument as \cite[Proposition 9.4.4]{MSJcurve}. In order to construct the diffeomorphism in Definition \ref{def:sympaffruled}, we also have to consider the moduli space of parametrized curves. Let $A\in H_2(X;\ZZ)$ and $J$ be a tame almost complex structure. Recall that for a smooth map $u:(S^2,j)\to (X,J)$, the Cauchy-Riemann equation is
\[\bar{\partial}_J u:=\frac{1}{2}(du+J\circ du\circ j)=0.\]
Define the relevant moduli spaces
\[\mathcal{M}(A;J):=\{u:S^2\rightarrow X\,|\,\overline{\partial}_J u=0,u_*([S^2])=A\},\]
\[\widetilde{\mathcal{M}}_{0,1}(A;J):=\mathcal{M}(A;J)\times S^2,\]
\[\mathcal{M}_{0,0}(A;J):=\mathcal{M}(A;J)/\text{PSL}(2;\CC),\]
\[\mathcal{M}_{0,1}(A;J):=\widetilde{\mathcal{M}}_{0,1}(A;J)/\text{PSL}(2;\CC).\]
There is the evaluation map $\widetilde{\text{ev}}:\widetilde{\mathcal{M}}_{0,1}(A;J)\rightarrow X$ by sending $(u,z)$ to $u(z)$. The map is invariant under the $\text{PSL}(2;\CC)$-action which descends to the map $\text{ev}:\mathcal{M}_{0,1}(A;J)\rightarrow X$.


\begin{theorem}\label{thm:ruledmain}
    Let $D\subseteq (X,\omega)$ be a symplectic divisor with $[\omega]\cdot (K_{\omega}+[D])<0$ where $X$ is an irrational ruled manifold. Then the pair $(X,\omega,D)$ is symplectic affine-ruled.
\end{theorem}

\begin{proof}
First, let us pick any $J\in \mathcal{J}(D)$. By Proposition \ref{prop:ruledconfig}, there is at most one component in $D$ with genus $g>0$. When no such component exists, we can choose some embedded $J$-holomorphic section in class $B+aF-\sum E_j$ for some $a\in\ZZ$ by Theorem \ref{thm:fiberclass}. Such an embedded surface of genus $g$, whether it is part of $D$ or chosen additionally, will be denoted by $C$. Also, we artificially choose an embedded $J$-holomorphic sphere $C_1$ in the class $F$ not contained as a component of $D$ and denote its intersection point with $C$ by $o$\footnote{When $s=r=0$, this extra choice of $o$ can help trivialize the $(S^2\setminus\{\text{pt}\})$-bundle over the punctured Riemann surface.}. The spherical components of $D$ in the class $F$ will intersect the surface $C$ at finitely many points $p_1,\cdots,p_s$. The spherical components in $D$ not in the class $F$ will be the component of some reducible subvariety in class $F$ by Theorem \ref{thm:fiberclass}. Since $F$ is $J$-nef, any reducible representative of $F$ must be a tree of embedded spheres by Theorem \ref{thm:lizhangnef}. Therefore, we can choose embedded $J$-holomorphic spheres $C_2,\cdots,C_l\subseteq X$ such that the support of all the reducible subvarieties in class $F$ is contained in the union of $\cup_{i=2}^l C_i$ and spherical components $D'$ in $D$. Denote by $q_1,\cdots,q_r$ the intersection points between these reducible subvarieties with $C$ which must be finite by Theorem \ref{thm:fiberclass}. 

Now, consider the moduli space $\mathcal{M}(F;J)$ for the fiber class $F$. By Theorem \ref{thm:fiberclass}, there always exist embedded curves in $\mathcal{M}(F;J)$. By adjunction inequality for $J$-holomorphic curves (\cite[Theorem 2.6.4]{MSJcurve}), all curves in $\mathcal{M}(F;J)$ are indeed embedded. For any $u\in\mathcal{M}(F;J)$, consider the linearization of $\bar{\partial}_J$ at $u$ 
\[D_u:=D\overline{\partial}_J:W^{1,p}(S^2,u^*TX)\rightarrow L^p(S^2,\Lambda^{0,1}T^*S^2\otimes u^*TX).\] Since $F^2=0$, the operator $D_u$ is surjective by automatic transversality (\cite{HLSauttrans} or \cite[Corollary 3.3.4]{MSJcurve}). Thus, $\mathcal{M}(F;J)$, $\widetilde{\mathcal{M}}_{0,1}(F;J)$ and $\mathcal{M}_{0,1}(F;J)$ are smooth manifolds of real dimensions $8$, $10$ and $4$ respectively, though not compact in general. Note that the points $p_1,\cdots,p_s$ correspond to $s$ copies of $\text{PSL}(2;\CC)$ in $\mathcal{M}(F;J)$. Let $\mathring{\mathcal{M}}(F;J)$ be the subspace of $\mathcal{M}(F;J)$ obtained by removing these elements (and similarly for other moduli spaces). The restriction of the evaluation map then gives
\[\widetilde{\text{ev}}:\mathring{\mathcal{M}}(F;J)\times S^2\rightarrow X\setminus (D'\cup C_1\cup\cdots\cup C_l)\]
which descends to 
\[\text{ev}:\mathring{\mathcal{M}}_{0,1}(F;J)\rightarrow X\setminus (D'\cup C_1\cup\cdots\cup C_l)\]
modulo the $\text{PSL}(2;\CC)$-action. $\text{ev}$ is a bijection between two manifolds of the same dimension by our choice of these $C_i$ spheres. To further prove  $\text{ev}$ is a diffeomorphism, let us consider the kernel of $d\widetilde{\text{ev}}$. For $(u,z)\in \mathring{\mathcal{M}}(F;J)\times S^2$, the tangent space
\[T_{(u,z)}(\mathring{\mathcal{M}}(F;J)\times S^2)=\text{ker}D_u\oplus T_zS^2.\]
Suppose $(\xi,\zeta)\in\text{ker }d\widetilde{\text{ev}}_{(u,z)}$ where $\xi\in\text{ker}D_u$ and $\zeta\in T_zS^2$. The differential of the evaluation map is given by
\[d\widetilde{\text{ev}}_{(u,z)}(\xi,\zeta)=\xi(z)+du_z\zeta.\]
Note that $u^*TX\cong\text{im }du\oplus(\text{im }du)^{\perp}\cong\mathcal{O}(2)\oplus \mathcal{O}$ so that we can write $\xi=\xi_0+\xi_1$ where $\xi_0\in W^{1,p}(S^2,\mathcal{O}(2))$ and $\xi_1\in W^{1,p}(S^2,\CC)$. Consider the real linear Cauchy-Riemann operator \[\tilde{D}_u:W^{1,p}(S^2,\CC)\rightarrow L^{p}(S^2,\Lambda^{0,1}T^*S^2)\]
obtained by restricting $D_u$ to $W^{1,p}(S^2,\CC)$ and then project to $L^p(S^2,\Lambda^{0,1}T^*S^2)$.
Since $D_u(\xi_0)\in L^p(S^2,\Lambda^{0,1}T^*S^2\otimes\mathcal{O}(2))$ by definition of $D_u$,  $\xi_1\in \text{ker}\tilde{D}_u$. $\tilde{D}_u$ can be further written as 
\[\overline{\partial}+a, \text{\,\,\,\,\,\,\,\,\, } a\in L^p(S^2,\Lambda^{0,1}T^*S^2\otimes \text{End}_\RR(\CC)).\]
We can define $b\in L^p(S^2,\Lambda^{0,1}T^*S^2)$ by setting 
\[b(z,\hat{z}):=\begin{cases}
    (a(z,\hat{z})\xi_1(z))/\xi_1(z),&\text{if }\xi_1(z)\neq 0,\\
    0,&\text{otherwise,}
\end{cases}\]
where $\hat{z}\in T_zS^2$. Then $\xi_1$ is in the kernel of a complex linear Cauchy-Riemman operator $\overline{\partial}+b$, which corresponds to a holomorphic section with respect to some holomorphic structure on $\mathcal{O}$. As a result, $\xi_1$ must be zero and $(\xi,\zeta)=(\xi_0,(du_z)^{-1}(\xi_0(z)))$ where $\xi_0$ is the vector field on $S^2$ induced by $\text{PSL}(2;\CC)$-action. Therefore, after modulo the $\text{PSL}(2;\CC)$-action, $d\text{ev}$ will be bijective everywhere by dimensional consideration and thus $\text{ev}$ must be a diffeomorphism.

Finally, note that the smooth, free and proper action of $\text{PSL}(2;\CC)$ provides $\mathring{\mathcal{M}}(F;J)$ with a principle $\text{PSL}(2;\CC)$-bundle structure. Then there is the associated $S^2$-bundle 
\[\begin{tikzcd}
	S^2 & \mathring{\mathcal{M}}_{0,1}(F;J) & \text{ev}^{-1}(C\setminus\{o,p_1,\cdots,p_s,q_1,\cdots,q_r\})\\
	& \mathring{\mathcal{M}}_{0,0}(F;J)
	\arrow[hook, from=1-1, to=1-2]
    \arrow[hook', from=1-3, to=1-2]
	\arrow[from=1-2, to=2-2]
\end{tikzcd}\]
where $\text{ev}^{-1}(C\setminus\{o,p_1,\cdots,p_s,q_1,\cdots,q_r\})$ is a submanifold. Since each curve in $\mathcal{M}(F;J)$ intersects $C$ transversely, this submanifold must be a section of the $S^2$-bundle transverse to all the $S^2$-fibers and positivity of intersection. Therefore $\mathring{\mathcal{M}}_{0,0}(F;J)$ is diffeomorphic to the punctured Riemman surface $C\setminus\{o,p_1,\cdots,p_s,q_1,\cdots,q_r\}$ over which any $S^2$-bundle must be trivial. It follows that $\text{ev}^{-1}(C\setminus\{o,p_1,\cdots,p_s,q_1,\cdots,q_r\})$ can be viewed as the graph of a smooth function from $\mathring{\mathcal{M}}_{0,0}(F;J)$ to $S^2$, whose complement in $\mathring{\mathcal{M}}_{0,1}(F;J)$ is diffeomorphic to $\mathring{\mathcal{M}}_{0,0}(F;J)\times(S^2\setminus\{\text{pt}\})$. Since $\text{ev}$ is a diffeomorphism, $X\setminus(D'\cup C_1\cup\cdots\cup C_l\cup C)$ will then be diffeomorphic to 
\[(C\setminus\{o,p_1,\cdots,p_s,q_1,\cdots,q_r\})\times (S^2\setminus\{\text{pt}\})\]
where all the fibers $S^2\setminus\{\text{pt}\}$ are symplectic since they are $J$-holomorphic. So the pair $(X,\omega,D)$ is symplectic affine-ruled.  

\end{proof}

\section{Rational manifolds}\label{section:rational}

In this section, we focus on connected symplectic divisors in rational manifolds with $[\omega]\cdot (K_\omega+[D])<0$. Our first observation is that the configurations of such divisors must form a tree of spheres, as established by the following proposition, which will be utilized throughout this section.

\begin{proposition}\label{prop:treeofsphere}
    Let $D=\cup D_i $ be a connected symplectic divisor in a symplectic rational manifold $(X,\omega)$ with $[\omega]\cdot(K_{\omega}+[D])<0$. Then each component $D_i$ must be a sphere and the graph of $D$ has no loop. In particular, $[D]^2+K_{\omega}\cdot [D]=-2$.
\end{proposition}

\begin{proof}
    Taking $A$ to be $[D]$ and each $[D_i]$ in Corollary \ref{cor:SWgenus}, we see that $[D]^2+K_{\omega}\cdot [D]<0$ and each $D_i$ must be a sphere by adjunction formula. Then we can smooth all the intersection points of the components to get a connected embedded symplectic surface representing the class $[D]$. Applying adjunction formula to this embedded symplectic surface we see that the negative value $[D]^2+K_{\omega}\cdot [D]$ must be $-2$ and the graph of $D$ has no loop since the genus is $0$ by Lemma \ref{lem:g(D)}.
\end{proof}

Before moving on, we first comment on the additional complexity in the rational case and give a brief outline of Section \ref{section:rational}. The main difference from the irrational ruled case is that, to equip the divisor complement with a one-punctured sphere fibration structure, one must look for foliations by singular curves rather than by smooth curves alone. There is a result in \cite[Theorem 1.7]{LLWnef} on $\omega$-nef smooth spherical classes in rational manifolds, analogous to Theorem \ref{thm:fiberclass}, which was used in Section \ref{section:ruled} for irrational ruled manifolds. However, there may exist pairs $(X,\omega,D)$ satisfying $[\omega]\cdot (K_\omega+[D])<0$ for which the spherical classes listed in \cite[Theorem 1.7]{LLWnef} cannot serve as candidates for affine rulings. For example, let $X=\mathbb{CP}^2\#4\overline{\mathbb{CP}}^2$ with standard basis $\{H,E_1,\cdots,E_4\}$ of $H_2(X;\mathbb{Z})$, and consider the string-like configuration $D$ in $X$ with homological sequence
    \[(H_{1234},H,H_1,E_1-E_3,E_3-E_4),\]
    where $H_{1234}$ and $H_1$ denote $H-E_1-E_2-E_3-E_4$ and $H-E_1$, respectively. The total homology class is $[D]=3H-E_1-E_2-E_3-2E_4$, so $[\omega]\cdot (K_\omega+[D])=-\omega(E_4)<0$ holds automatically. If the $\omega$-nef spherical class $A$ is taken from the classes listed in \cite[Theorem 1.7]{LLWnef}, then the following intersections occur:
    \begin{itemize}
        \item $H-E_1$: the fiber intersects the $H$- and $(E_1-E_3)$-components of $D$.
        \item $lH-(l-1)E_1$: the fiber intersects the $H_{1234}$-, $H$-, $H_1$-, and $(E_1-E_3)$-components of $D$.
        \item $lH-(l-1)E_1-E_2$: the fiber intersects the $H$-, $H_1$-, and $(E_1-E_3)$-components of $D$.
        \item $2H$: the fiber intersects the $H_{1234}$-, $H$-, and $H_1$-components of $D$.
    \end{itemize}
In each case, the fiber intersects $D$ in more than two points. Thus, on the divisor complement $X\setminus D$, one obtains at best a $\mathbb{C}^*$-fibration, or something worse, whereas the desired structure is a $\mathbb{C}$-fibration.
This motivates the search for unicuspidal curves whose cusps are located at intersection points of components of $D$. In this example, one may take either
\[A=(E_3-E_4)+2(E_1-E_3)+3H_1=3H-E_1-E_2-E_3,\]
which has a $(2,3)$-cusp at the intersection of the $H_1$- and $H$-components, or take
\[A=H_{1234}+3H=4H-E_1-E_2-E_3-E_4,\]
which has a $(3,4)$-cusp at the intersection of the $H_1$- and $H$-components.

More generally, the divisor may be even more complicated than in the example above, so we must first reduce its complexity by successive blowdowns. In Section \ref{section:Emin}, we collect useful properties of exceptional spheres with minimal symplectic area. In Section \ref{section:defofquasiminimal}, we introduce the key notion of \emph{quasi-minimal} pairs in Definition \ref{def:quasiminimal}. After reviewing the possible divisor configurations in manifolds with $b_2(X)\leq 2$ in Section \ref{section:divisorsinminimal}, we prove the structural result for quasi-minimal pairs in Section \ref{section:quasiminimal}. Section \ref{section:findcusp} is then devoted to finding unicuspidal curves that foliate the divisor complement, based on the disussions in Section \ref{section:cuspcurve}. Finally, in Section \ref{secton:proofofaffineruled}, we use the preceding reduction schemes to prove the affine-ruledness result.

\subsection{Exceptional classes with minimal area}\label{section:Emin}

As outlined above, the complexity of divisors can be simplified through blowdowns. Compared to the holomorphic setting, a notable feature working in the symplectic setting is the ability to discuss the symplectic area of an exceptional class. Let us introduce the following notions for a symplectic $4$-manifold $(X,\omega)$:
\[\mathcal{E}(X):=\{E\in H_2(X;\ZZ)\,|\,E\text{ can be represented by an embedded smooth $(-1)$-sphere}\},\]
\[\mathcal{E}_{\omega}(X):=\{E\in H_2(X;\ZZ)\,|\,E\text{ can be represented by an embedded symplectic $(-1)$-sphere}\},\]
\[e_{\text{min}}(X,\omega):=\inf\{\omega(E)\,|\,E\in \mathcal{E}_{\omega}(X)\},\]
\[\mathcal{E}_{\text{min}}(X,\omega):=\{E\in\mathcal{E}_{\omega}(X)\,|\,\omega(E)=e_{\text{min}}(X,\omega)\}.\]
 By Gromov's compactness, whenever $\mathcal{E}_{\omega}(X)$ is non-empty, $e_{\text{min}}(X,\omega)$ is a positive real number and $\mathcal{E}_{\text{min}}(X,\omega)$ is a finite and non-empty set. A useful fact (see \cite[Theorem~A]{LiLiuruled}) is that, when $b_2^+(X)=1$,
 \[
   \mathcal{E}_{\omega}(X)=\{E\in\mathcal{E}(X)\mid E\cdot K_{\omega}=-1\}.
 \]
 Combined with the case $b_2^+(X)>1$ in \cite{TaubesSWandGr}, it follows that $\mathcal{E}(X)$ is empty if and only if $\mathcal{E}_\omega(X)$ is empty. We call $(X,\omega)$ \textbf{minimal} if $\mathcal{E}(X)$ (equivalently, $\mathcal{E}_\omega(X)$) is empty. We will need the following lemma.

\begin{lemma}\label{lem:minimalarea}
Let $(X,\omega)$ be a non-minimal symplectic $4$-manifold not diffeomorphic to $\CC\PP^2\#\overline{\CC\PP}^2$. If $A\in H_2(X;\ZZ)$ is a class with $SW(A)\neq 0$ and $A\cdot K_{\omega}\leq -1$, then $\omega(A)\geq e_{\text{min}}(X,\omega)$. 
\end{lemma}

\begin{proof}
    We choose any $\omega$-tame almost complex structure $J$. Since $SW(A)\neq 0$ we know $A$ is $J$-effective. By Zhang's curve cone Theorem \ref{thm:curvecone}, there is a decomposition \[A=\sum_{i=1}^m a_iF_i+\sum_{j=1}^n b_jE_j,\] where $F_i$ is a $J$-effective class with $F_i\cdot K_{\omega}\geq 0$, $E_j\in \mathcal{E}_{\omega}(X)$ and $a_i,b_j$'s are positive real numbers. By the assumption that $A\cdot K_{\omega}\leq -1$, we have 
    \[1\leq -K_{\omega}\cdot A=-\sum_{i=1}^m a_i(K_{\omega}\cdot F_i)-\sum_{j=1}^n b_j(K_{\omega}\cdot E_j) \leq \sum_{j=1}^nb_j.\] Therefore
    \[\omega(A)=\sum_{i=1}^m a_i\omega(F_i)+\sum_{j=1}^nb_j\omega(E_j)\geq \sum_{j=1}^nb_j\omega(E_j)\geq\sum_{j=1}^nb_j e_{\text{min}}(X,\omega)\geq e_{\text{min}}(X,\omega).\]
\end{proof}

The following useful result was first observed by Pinsonnault in \cite{Pinsonnaultmaxtori}\footnote{The original proof in \cite{Pinsonnaultmaxtori} uses Mori's cone theorem in algebraic geometry, which only works for smooth projective varieties. Since many symplectic manifolds are even non-K\"ahler, a complete proof needs Zhang's curve cone theorem in the almost complex setting.}. Here we give a proof as a corollary of Zhang's curve cone theorem. This proof is informed to us by Weiyi Zhang.

\begin{corollary}\label{cor:minimalarea}
    Let $(X,\omega)$ be a symplectic $4$-manifold not diffeomorphic to $\CC\PP^2\#\overline{\CC\PP}^2$ and $J$ is any $\omega$-tame almost complex structure. Then any class $E\in \mathcal{E}_{\text{min}}(X,\omega)$ can be represented by an embedded $J$-holomorphic sphere.
\end{corollary}
\begin{proof}
    By Gromov's compactness, $E$ can be represented by a possibly reducible $J$-holomorphic subvariety $\Theta_E=\{(C_i,m_i)\}_{i=1}^n$. Then there will be some component $C_i$ with $[C_i]\cdot K\leq -1$ since $-1=E\cdot K_\omega=\sum_{i=1}^n m_i[C_i]\cdot K_\omega$. Taking $A=[C_i]$ in the proof of Lemma \ref{lem:minimalarea}, we then must have $\omega([C_i])=e_{\text{min}}(X,\omega)$. This implies that there is only one component in $\Theta_E$. Moreover, it is embedded by adjunction formula.
\end{proof}

If some $E_{\text{min}}\in\mathcal{E}_{\text{min}}(X,\omega)$ is chosen, we can further define
\[e_{\text{min}}(X,\omega;E_{\text{min}}):=\inf\{\omega (E)\,|\,E\in \mathcal{E}_{\omega}(X),E\cdot E_{\text{min}}=0\},\]
\[\mathcal{E}_{\text{min}}(X,\omega;E_{\text{min}}):=\{E\in\mathcal{E}_{\omega}(X)\,|\,\omega(E)=e_{\text{min}}(X,\omega;E_{\text{min}}),E\cdot E_{\text{min}}=0\}.\]
A class $E'\in\mathcal{E}_{\text{min}}(X,\omega;E_{\text{min}})$ will be called a {\bf secondary minimal area exceptional class} of $E_{\text{min}}$. If $(X',\omega')$ is the symplectic blowdown of $C_{E_\text{min}}\subseteq (X,\omega)$, the embedded symplectic exceptional sphere of minimal area, $E'$ can be naturally viewed as a class in $\mathcal{E}_{\text{min}}(X',\omega')$. If $X'\neq \CC\PP^2\#\overline{\CC\PP}^2$, we can then further consider the secondary minimal area exceptional class of $E'$ and blow down the embedded $J$-holomorphic exceptional sphere representing $E'$ for any $\omega'$-tame $J$ on $X'$. This process can be repeated until the manifold becomes minimal or $\CC\PP^2\#\overline{\CC\PP}^2$ which will produce a sequence of exceptional classes. In summary, we have the following corollary.

\begin{corollary}\label{cor:secondminexceptional}
    Let $(X,\omega)$ be a symplectic rational manifold and suppose $n=b_2^-(X)\geq 2$. Then there is a choice of $E_2,\cdots,E_n\subseteq \mathcal{E}_{\omega}(X)$ such that $E_n\in \mathcal{E}_{\text{min}}(X,\omega)$ and $E_{i}$ is a secondary minimal area exceptional class of $E_{i+1}$ for all $2\leq i\leq n-1$. Furthermore, let $X'$ be the blowdown of the exceptional spheres in these classes $E_2,\cdots,E_n$. Then, there exist $E_1,E_1'\in\mathcal{E}_{\omega}(X)$ satisfying $E_1\cdot E_j=E_1'\cdot E_j=0$ for all $3\leq j\leq n$ and
    \begin{itemize}
        \item either $X'=S^2\times S^2$ and $E_1\cdot E_1'=0,E_1\cdot E_2=E_1'\cdot E_2=1$;
        \item or $X'=\CC\PP^2\#\overline{\CC\PP}^2$ and $E_1\cdot E_1'=E_1'\cdot E_2=1,E_1\cdot E_2=0$.
    \end{itemize}
\end{corollary}
\begin{proof}
    The two cases reflect the well-known non-uniqueness of minimal models of rational surfaces. Choose a standard basis $\{h,e_1,\cdots,e_n\}$ of $H_2(X;\ZZ)$ such that $e_i=E_i$ for $3\leq i\leq n$. After blowing down a collection of exceptional spheres in classes $e_n,e_{n-1},\cdots,e_3$, we arrive at $\mathbb{CP}^2\#2\overline{\mathbb{CP}}^2$. The resulting model is determined by whether one next blows down the exceptional sphere in class $h-e_1-e_2$ or $e_2$, yielding respectively $S^2\times S^2$ or $\mathbb{CP}^2\#\overline{\mathbb{CP}}^2$. If $X'=S^2\times S^2$, then $E_2=h-e_1-e_2$; set $E_1:=e_1$ and $E_1':=e_2$. If $X'=\CC\PP^2\#\overline{\CC\PP}^2$, then $E_2$ is either $e_1$ or $e_2$; accordingly, set $(E_1,E_1')=(e_2,h-e_1-e_2)$ or $(e_1,h-e_1-e_2)$. In all cases, the required intersection properties follow immediately from the standard intersection form.
\end{proof}

\subsection{Minimal reductions for pairs}\label{section:defofquasiminimal}
We now introduce several key notions for studying the structure of connected symplectic divisors in rational manifolds satisfying $[\omega]\cdot (K_{\omega}+[D])<0$. These notions can be viewed as a relative version of the minimality introduced in the previous section for the absolute setting.
\begin{definition}\label{def:quasiminimal}
    The connected pair $(X,\omega,D)$ is said to be \textbf{quasi-minimal} if $b_2(X)\geq 3$ and there exists some $E_{\text{min}}\in\mathcal{E}_{\text{min}}(X,\omega)$ so that $[D]\cdot E_{\text{min}}\geq 2$.
\end{definition}

The following lemma allows us to reduce any connected pair to a quasi-minimal pair via blowdowns.

\begin{lemma}\label{lem:quasiminred}
    Any connected pair $(X,\omega,D)$, where $X$ is a rational manifold, can be obtained from a sequence of blowups (toric, non-toric, half-toric or exterior) from either a quasi-minimal pair or a divisor in $\CC\PP^2, S^2\times S^2$ or $\CC\PP^2\#\overline{\CC\PP}^2$.
\end{lemma}

\begin{proof}
    Assume $[D]\cdot E_{\text{min}}\leq 1$. Then by Corollary \ref{cor:minimalarea}, for any $J\in\mathcal{J}(D)$, $E_{\text{min}}$ has an embedded representative $C_{E_{\text{min}}}$.  By positivity of intersection, if $C_{E_{\text{min}}}$ is a component of $D$, then the condition $[D]\cdot E_{\text{min}}\leq 1$ implies that there are at most two other components of $D$ intersecting $C_{E_{\text{min}}}$ which we can blowdown (toric or half-toric) to get another symplectic divisor. Otherwise, $C_{E_{\text{min}}}$ would intersect exactly one component of $D$ or be disjoint from $D$. And we can perform non-toric or exterior blowdown to obtain another symplectic divisor. Repeating this process finitely many times (since $b_2^-$ is finite) we will arrive at either a quasi-minimal pair or a pair with $b_2\leq 2$.
\end{proof}

\begin{rmk}
    The prefix `quasi' is used for the obvious reason: we only require the non-existence of one specific exterior/toric/non-toric/half-toric exceptional sphere with minimal symplectic area. A quasi-minimal pair could contain many other exceptional spheres with symplectic area $\geq e_{\text{min}}$ which allow further exterior/toric/non-toric/half-toric blowdowns.
\end{rmk}

For quasi-minimal pairs, we introduce a further minimality notion below.

\begin{definition}\label{def:partialmin}
    A quasi-minimal pair $(X,\omega,D)$ is called {\bf partially minimal} if 
    \begin{enumerate}
        \item no class $E\in \mathcal{E}_{\omega}(X)$ with $E\cdot E_{\text{min}}=0$ has non-negative intersections with all components of $D$;
        \item no component $D_i$ of $D$ with $[D_i]^2=-1$ satisfies $[D_i]\cdot([D]-[D_i])=2$.
    \end{enumerate}
\end{definition}

Note that the first condition ensures that there are no non-toric exceptional spheres, while the second condition implies the absence of toric exceptional spheres. Using an argument similar to that of Lemma \ref{lem:quasiminred}, we obtain the following partially minimal reduction in the category of quasi-minimal pairs; its proof is deferred to Section \ref{section:quasiminimal}.

\begin{lemma}\label{lem:partialminimalred}
    Any connected quasi-minimal pair $(X,\omega,D)$, where $X$ is a rational manifold, is obtained from a sequence of toric or non-toric blowups from either another quasi-minimal pair which is partially minimal or a divisor in $\CC\PP^2, S^2\times S^2$ or $\CC\PP^2\#\overline{\CC\PP}^2$.
\end{lemma}

\subsection{Divisors in manifolds with $b_2(X)\leq 2$}\label{section:divisorsinminimal}

By Lemma \ref{lem:quasiminred} and \ref{lem:partialminimalred}, in order to establish our ultimate result, we must first consider the cases with $b_2(X)\leq 2$, which are not covered by the quasi-minimal pairs which are also partially minimal. It turns out that all possible configurations are either comb-like, as seen in irrational ruled surfaces (Example \ref{example:comb}), or sub-configurations of minimal models of symplectic log Calabi-Yau divisors which we now introduce.

A pair $(X,\omega,D)$ is called {\bf symplectic log Calabi-Yau} if $[D]=-K_{\omega}$. In a series of work \cite{LiMakLCY,LiMakICCM,LiMinMak,Enumerate}, such a notion was introduced as the symplectic analogue of the anticanonical pairs in algebraic geometry and their deformation, contact and enumerative aspects were studied and related to symplectic fillings, toric actions and almost toric fibrations. One basic fact is that their configurations must be a cycle of spheres when $l(D)\geq 2$ and $X$ is a rational manifold (\cite[Lemma 3.1]{LiMakLCY}).

It was shown in \cite[Lemma 3.4]{LiMakLCY} that any symplectic log Calabi-Yau divisor in a rational manifold with $l(D)\geq 2$ can be obtained from a sequence of toric blowups followed by a sequence of non-toric blowups from {\bf minimal models} shown below ($k$ is some integer):
		
		$\bullet$  Case $(A)$: $X=\mathbb{CP}^2$,
		$-K_{\omega}=3h$.

		$(A1)$ $D$ consists of a $h-$sphere and a $2h-$sphere, or
		
		$(A2)$ $D$ consists of three $h-$spheres.
		
		$\bullet$ Case $(B)$: $X=S^2 \times S^2$, $-K_{\omega}=2f_1+2f_2$, where $f_1$ and $f_2$ are the homology classes
		of the two factors.
		
		$(B1)$ $l(D)=2$ and  $[D_1]=kf_1+f_2, [D_2]=(2-k)f_1+f_2$, $k\geq 1$.
		
		$(B2)$ $l(D)=3$ and  $[D_1]=kf_1+f_2, [D_2]=f_1, [D_3]=(1-k)f_1+f_2$, $k\geq 1$.
		
		$(B3)$ $l(D)=4$ and $[D_1]=kf_1+f_2,  [D_2]=f_1, [D_3]=-kf_1+f_2, [D_4]=f_1$, $k\geq 0$.

        $\bullet$ Case $(C)$: $X=\mathbb{CP} ^2 \# \overline{\mathbb{CP}}^2$, $-K_{\omega}=f+2s$, where $f$ and $s$ are the fiber class and section class
		with $f\cdot f=0$, $f\cdot s=1$ and $s\cdot s=1$.

		$(C1)$ $l(D)=2$,  and either
		$([D_1],[D_2])=(kf+s,(1-k)f+s)$, $k\geq 1$ or $([D_1],[D_2])=(2s, f)$.
		
		$(C2)$  $l(D)=3$ and  $[D_1]=kf+s, [D_2]=f, [D_3]=-kf+s$, $k\geq 0$.
		
		$(C3)$ $l(D)=4$ and  $[D_1]=kf+s, [D_2]=f,  [D_3]=-(k+1)f+s, [D_4]=f$, $k\geq 0$.

The dual graphs \footnote{Labelled points denote symplectic surfaces with a prescribed self-intersection number, edges denote intersections between two symplectic surfaces.} in $(A1)$, $(A2)$ are given respectively  by
        \[
		\begin{tikzpicture}
		\node at (-1,0) {$\CC\PP^2$};
		
		\node (x1) at (1,0) [circle,fill,outer sep=5pt, scale=0.5] [label=above:$ 1 $]{};
		\node (y1) at (2.5,0) [circle,fill,outer sep=5pt, scale=0.5] [label=above:$ 4 $]{};
		\draw (x1) to[bend right] (y1);\draw (y1) to[bend right] (x1);
		\node at (1.75,-1) {($A_1$)};
		
		\node (x2) at (4,0) [circle,fill,outer sep=5pt, scale=0.5] [label=above:$ 1 $]{};
		\node (y2) at (5.5,0) [circle,fill,outer sep=5pt, scale=0.5] [label=above:$ 1 $]{};
		\node (z2) at (4,-1) [circle,fill,outer sep=5pt, scale=0.5] [label=below:$ 1 $]{};
		\draw (x2) -- (y2);\draw (y2) -- (z2);\draw (z2) -- (x2);
		\node at (5,-1) {($A_2$)};
		
		\end{tikzpicture}
		\]	
		
		The dual graphs in $(B1)$, $(B2)$, $(B3)$ are given respectively  by
		\[
		\begin{tikzpicture}
		
		\node at (-1,0) {$S^2\times S^2$};
		
		\node (x1) at (1,0) [circle,fill,outer sep=5pt, scale=0.5] [label=above:$ 2k $]{};
		\node (y1) at (2.5,0) [circle,fill,outer sep=5pt, scale=0.5] [label=above:$ 4-2k $]{};
		\draw (x1) to[bend right] (y1);\draw (y1) to[bend right] (x1);
		\node at (1.75,-1) {($B_1$)};
		
		\node (x2) at (4,0) [circle,fill,outer sep=5pt, scale=0.5] [label=above:$ 2k $]{};
		\node (y2) at (5.5,0) [circle,fill,outer sep=5pt, scale=0.5] [label=above:$ 0 $]{};
		\node (z2) at (4,-1) [circle,fill,outer sep=5pt, scale=0.5] [label=below:$ 2-2k $]{};
		\draw (x2) -- (y2);\draw (y2) -- (z2);\draw (z2) -- (x2);
		\node at (5,-1.1) {($B_2$)};
		
		\node (x3) at (7,0) [circle,fill,outer sep=5pt, scale=0.5] [label=above:$ 2k $]{};
		\node (y3) at (8.5,0) [circle,fill,outer sep=5pt, scale=0.5] [label=above:$ 0 $]{};
		\node (z3) at (8.5,-1) [circle,fill,outer sep=5pt, scale=0.5] [label=below:$ -2k $]{};
		\node (w3) at (7,-1) [circle,fill,outer sep=5pt, scale=0.5] [label=below:$ 0 $]{};
		\draw (x3) -- (y3);\draw (y3) -- (z3);\draw (z3) -- (w3);\draw (w3) to (x3);
		\node at (7.75,-1.7) {($B_3$)};
		\end{tikzpicture}
		\]	
		
		The dual graphs in $(C1)$, $(C2)$ and $(C3)$ are given respectively by
		\[
		\begin{tikzpicture}
		
		\node at (-4,0) {$\CC\PP^2\#\overline{\CC\PP}^2$};
		
		\node (x) at (-2,0) [circle,fill,outer sep=5pt, scale=0.5] [label=above:$ 2k+1 $]{};
		\node (y) at (-0.5,0) [circle,fill,outer sep=5pt, scale=0.5] [label=above:$ 3-2k $]{};
		\draw (x) to [bend right] (y);\draw (y) to [bend right] (x);
		
		\node (x1) at (1,0) [circle,fill,outer sep=5pt, scale=0.5] [label=above:$ 4 $]{};
		\node (y1) at (2.5,0) [circle,fill,outer sep=5pt, scale=0.5] [label=above:$ 0 $]{};
		\draw (x1) to[bend right] (y1);\draw (y1) to[bend right] (x1);
		\node at (0.25,-1.2) {($C_1$)};
		
		\node (x2) at (4,0) [circle,fill,outer sep=5pt, scale=0.5] [label=above:$ 2k+1 $]{};
		\node (y2) at (5.5,0) [circle,fill,outer sep=5pt, scale=0.5] [label=above:$ 0 $]{};
		\node (z2) at (4,-1) [circle,fill,outer sep=5pt, scale=0.5] [label=below:$ 1-2k $]{};
		\draw (x2) -- (y2);\draw (y2) -- (z2);\draw (z2) -- (x2);
		\node at (5,-1.2) {($C_2$)};
		
		\node (x3) at (7,0) [circle,fill,outer sep=5pt, scale=0.5] [label=above:$ 2k+1 $]{};
		\node (y3) at (8.5,0) [circle,fill,outer sep=5pt, scale=0.5] [label=above:$ 0 $]{};
		\node (z3) at (8.5,-1) [circle,fill,outer sep=5pt, scale=0.5] [label=below:$ -2k-1 $]{};
		\node (w3) at (7,-1) [circle,fill,outer sep=5pt, scale=0.5] [label=below:$ 0 $]{};
		\draw (x3) -- (y3);\draw (y3) -- (z3);\draw (z3) -- (w3);\draw (w3) to (x3);
		\node at (7.75,-1.8) {($C_3$)};
		\end{tikzpicture}
		\]	

We now enumerate all possible configurations of divisors satisfying $[\omega]\cdot (K_\omega+[D])<0$, analogous to the list above for which $[\omega]\cdot (K_\omega+[D])=0$.

When $X$ is $\CC\PP^2$, let $h$ be the line class in $H_2(X;\ZZ)$. Then $K_\omega=-3h$. In order to make $[\omega]\cdot (K_\omega+[D])<0$, $[D]$ must be either $h$ or $2h$ and we have the following cases:

    $(A1)'$  $D$ is a single sphere in class $h$. 
    
    $(A2)'$ $D=D_1\cup D_2$ where $D_1,D_2$ are spheres in class $h$. 
    
    $(A3)'$ $D$ is a single sphere in class $2h$.

When $X$ is $S^2\times S^2$, let $f_1,f_2\in H_2(X;\ZZ)$ be the classes of $[S^2\times \{\text{pt}\}]$ and $[\{\text{pt}\}\times S^2]$. Then $K_\omega=-2f_1-2f_2$. By Proposition \ref{prop:treeofsphere}, each component of $D$ is a sphere. It is easy to see from adjunction formula that the homology class of each component must be either $f_1+kf_2$ or $f_2+kf_1$ for some $k\in \ZZ$. After modulo the symmetry in $f_1$ and $f_2$, the following cases remain when $[\omega]\cdot (K_\omega+[D])<0$:

     $(B1)'$ $D$ is comb-like: $D=D_1\cup \cdots \cup D_n$ where $[D_1]=f_1+kf_2$ and each $[D_i]=f_2$ for $2\leq i\leq n$. 
     
     $(B2)'$ $D$ has three components with $[D_1]=f_1+kf_2$, $[D_2]=f_2$, $[D_3]=f_1-kf_2$ for some $k\in\ZZ_{\geq0}$.
     
     $(B3)'$ $D$ has two components with $[D_1]=f_1+(k+1)f_2$, $[D_2]=f_1-kf_2$ for some $k\in \ZZ_{\geq0}$.

When $X$ is $\CC\PP^2\#\overline{\CC\PP}^2$, let $f,s\in H_2(X;\ZZ)$ be the fiber class and section class respectively with $f\cdot f=0,f\cdot s=1$ and $s\cdot s=1$. Then $K_\omega=-f-2s$. By adjunction formula, the homology classes of embedded symplectic spheres must be $s+kf$ for some $k\in\ZZ$. There are the following cases when $[\omega]\cdot (K_\omega+[D])<0$:

     $(C1)'$ $D$ is comb-like: $D=D_1\cup \cdots \cup D_n$ where $[D_1]=s+kf$ and each $[D_i]=f$ for $2\leq i\leq n$. 
     
     $(C2)'$ $D$ has three components with $[D_1]=s+kf$, $[D_2]=f$, $[D_3]=s+(-1-k)f$ for some $k\in\ZZ_{\geq0}$.
     
     $(C3)'$ $D$ has two components with $[D_1]=s+kf$, $[D_2]=s-kf$ for some $k\in \ZZ_{\geq0}$. 

The dual graphs in $(A1)'$, $(A2)'$, $(A3)'$ are given respectively  by
        \[
		\begin{tikzpicture}
		\node at (-3,0) {$\CC\PP^2$};
		
        \node (x1) at (-1,0) [circle,fill,outer sep=5pt, scale=0.5] [label=above:$ 1 $]{};
		\node at (-1,-0.8) {$(A1)'$};
      
		\node (x1) at (1,0) [circle,fill,outer sep=5pt, scale=0.5] [label=above:$ 1 $]{};
		\node (y1) at (2.5,0) [circle,fill,outer sep=5pt, scale=0.5] [label=above:$ 1 $]{};
		\draw (x1) to (y1);
		\node at (1.75,-0.8) {$(A2)'$};
		
		\node (x2) at (4.5,0) [circle,fill,outer sep=5pt, scale=0.5] [label=above:$ 4 $]{};
		\node at (4.5,-0.8) {$(A3)'$};
		
		\end{tikzpicture}
		\]	
		
		The dual graphs in $(B1)'$, $(B2)'$, $(B3)'$ are given respectively  by
		\[
		\begin{tikzpicture}
		
		\node at (-3,0.5) {$S^2\times S^2$};
		
		\node (x1) at (1,0) [circle,fill,outer sep=5pt, scale=0.5] [label=below:$ 2k $]{};
		\node (y1) at (-0.5,1) [circle,fill,outer sep=5pt, scale=0.5] [label=above:$ 0 $]{};
        \node (y2) at (0.2,1) [circle,fill,outer sep=5pt, scale=0.5] [label=above:$ 0 $][label=right: $\cdots\cdots$]{};
        \node (y3) at (1.8,1) [circle,fill,outer sep=5pt, scale=0.5] [label=above:$ 0 $]{};
        \node (y4) at (2.5,1) [circle,fill,outer sep=5pt, scale=0.5] [label=above:$ 0 $]{};
		\draw (x1) to (y1);\draw (x1) to (y2);\draw (x1) to (y3);\draw (x1) to (y4);
		\node at (1.0,-0.8) {$(B1)'$};
		
		\node (x2) at (4,1) [circle,fill,outer sep=5pt, scale=0.5] [label=above:$ 0 $]{};
		\node (y2) at (5.5,1) [circle,fill,outer sep=5pt, scale=0.5] [label=above:$ 2k $]{};
		\node (z2) at (4,0) [circle,fill,outer sep=5pt, scale=0.5] [label=below:$ -2k $]{};
		\draw (x2) -- (y2);\draw (z2) -- (x2);
		\node at (4.75,-0.8) {$(B2)'$};
		
		\node (x3) at (7,1) [circle,fill,outer sep=5pt, scale=0.5] [label=above:$ 2k+2 $]{};
		\node (y3) at (8.5,1) [circle,fill,outer sep=5pt, scale=0.5] [label=above:$ -2k $]{};
		\draw (x3) -- (y3);
		\node at (7.75,-0.8) {$(B3)'$};
		\end{tikzpicture}
		\]	
		
		The dual graphs in $(C1)'$, $(C2)'$, $(C3)'$  are given respectively by
		\[
		\begin{tikzpicture}
		
		\node at (-4,0.5) {$\CC\PP^2\#\overline{\CC\PP}^2$};
		
		\node (x1) at (1,0) [circle,fill,outer sep=5pt, scale=0.5] [label=below:$ 2k+1 $]{};
		\node (y1) at (-0.5,1) [circle,fill,outer sep=5pt, scale=0.5] [label=above:$ 0 $]{};
        \node (y2) at (0.2,1) [circle,fill,outer sep=5pt, scale=0.5] [label=above:$ 0 $][label=right: $\cdots\cdots$]{};
        \node (y3) at (1.8,1) [circle,fill,outer sep=5pt, scale=0.5] [label=above:$ 0 $]{};
        \node (y4) at (2.5,1) [circle,fill,outer sep=5pt, scale=0.5] [label=above:$ 0 $]{};
		\draw (x1) to (y1);\draw (x1) to (y2);\draw (x1) to (y3);\draw (x1) to (y4);
		\node at (1.0,-0.8) {$(C1)'$};
		
		\node (x2) at (4,1) [circle,fill,outer sep=5pt, scale=0.5] [label=above:$ 0 $]{};
		\node (y2) at (5.5,1) [circle,fill,outer sep=5pt, scale=0.5] [label=above:$ 2k+1 $]{};
		\node (z2) at (4,0) [circle,fill,outer sep=5pt, scale=0.5] [label=below:$ -2k-1 $]{};
		\draw (x2) -- (y2);\draw (z2) -- (x2);
		\node at (4.75,-0.8) {$(C2)'$};
		
		\node (x3) at (7,1) [circle,fill,outer sep=5pt, scale=0.5] [label=above:$ 2k+1 $]{};
		\node (y3) at (8.5,1) [circle,fill,outer sep=5pt, scale=0.5] [label=above:$ -2k+1 $]{};
		\draw (x3) -- (y3);
		\node at (7.75,-0.8) {$(C3)'$};
		\end{tikzpicture}
		\]	

Based on the above enumerations, we will give the proof of the following result at the end of Section \ref{section:findcusp}.

\begin{lemma}\label{lem:threecases}
    Let $(X,\omega,D)$ be a connected pair where $X$ is a rational manifold with $[\omega]\cdot(K_\omega+[D])<0$ and $b_2(X)\leq 2$, then all blowups (toric, non-toric, half-toric, exterior) of $(X,\omega,D)$ are symplectic affine-ruled.
\end{lemma}

\subsection{Structural results for quasi-minimal pairs}\label{section:quasiminimal}
The goal of this section is to prove Lemma \ref{lem:partialminimalred}. The proof will be divided into two parts, given in Sections \ref{section:LCY} and \ref{section:psedolcy}, which is based on the analysis of the structure of quasi-minimal pairs given below.

\begin{lemma}\label{lem:E+D+K}
  Suppose $(X,\omega,D)$ is a quasi-minimal connected pair with $[\omega]\cdot(K_\omega+[D])<0$, where $X$ is a rational manifold. If $E_{\text{min}}\cdot[D]\geq 2$, then $SW(E_{\text{min}}+[D]+K_{\omega})\neq 0$.
\end{lemma}

\begin{proof}
   By Proposition \ref{prop:treeofsphere} we have $[D]^2+K_{\omega}\cdot [D]=-2$. Using the assumption that $E_{\text{min}}\cdot[D]\geq 2$, we then compute the index 
   \begin{align*}
       I(E_{\text{min}}+[D]+K_{\omega}) & =(E_{\text{min}}+[D]+K_{\omega})^2-K_{\omega}\cdot (E_{\text{min}}+[D]+K_{\omega})\\
       &=(E_{\text{min}}+[D])\cdot(E_{\text{min}}+[D]+K_{\omega})\\
       & = E_{\text{min}}^2+2E_{\text{min}}\cdot [D]+[D]^2+E_{\text{min}}\cdot K+[D]\cdot K_{\omega}\\
       & = 2E_{\text{min}}\cdot [D]-4\\
       & \geq 0.
   \end{align*}
    Now since \[[\omega]\cdot (K_{\omega}-(E_{\text{min}}+[D]+K_{\omega}))=-[\omega]\cdot (E_{\text{min}}+[D])<0,\] by Corollary \ref{cor:SWnonzero} we see that $SW(E_{\text{min}}+[D]+K_{\omega})\neq 0$.
\end{proof}

\begin{lemma}\label{lem:(K+D)^2>-1}
    If $(X,\omega, D)$ is a quasi-minimal connected pair with $[\omega]\cdot(K_{\omega}+[D])<0$, where $X$ is a rational manifold, then we have $(K_{\omega}+[D])^2\geq -1$.
\end{lemma}
\begin{proof}
    By assumptions, we can choose some class $E_{\text{min}}\in \mathcal{E}_{\text{min}}(X,\omega)$ with $E_{\text{min}}\cdot [D]\geq 2$. By Lemma \ref{lem:E+D+K} we have $SW(E_{\text{min}}+[D]+K_{\omega})\neq 0$. If $(K_{\omega}+[D])^2\leq -2$, by Proposition \ref{prop:treeofsphere}, we see that 
    \begin{align*}
        (E_{\text{min}}+[D]+K_{\omega})\cdot K_{\omega}&=E_{\text{min}}\cdot K_{\omega}+([D]+K_{\omega})^2-[D]\cdot([D]+K_{\omega})\\
       & \leq -1-2+2=-1.
    \end{align*}
    Thus, Lemma \ref{lem:minimalarea} implies that $[\omega]\cdot(E_{\text{min}}+[D]+K_{\omega})\geq \omega(E_{\text{min}})$, which is a contradiction with the assumption $[\omega]\cdot (K_{\omega}+[D])<0$.
\end{proof}

\begin{proposition}\label{prop:-K-D=Emin}
    If $(X,\omega,D)$ is a quasi-minimal connected pair with $[\omega]\cdot(K_{\omega}+[D])<0$, where $X$ is a rational manifold, then $-[D]-K_{\omega}\in\mathcal{E}_{\text{min}}(X,\omega)$. Moreover, $-[D]-K_{\omega}$ is the unique element in $\mathcal{E}_{\text{min}}(X,\omega)$ whose pairing with $[D]$ is at least $2$.
\end{proposition}
\begin{proof}
    By Lemma \ref{lem:(K+D)^2>-1} we may assume $(K_{\omega}+[D])^2\geq -1$. By Proposition \ref{prop:treeofsphere}, $[D]^2+K_\omega\cdot[D]=-2$. We then compute the index 
    \begin{align*}
        I(-[D]-K_{\omega})&=(-[D]-K_{\omega})^2-K_{\omega}\cdot(-[D]-K_{\omega})\\
        &=[D]^2+2[D]\cdot K_\omega+K_\omega^2+K_\omega\cdot [D]+K_\omega^2\\
        &=2([D]+K_{\omega})^2-[D]^2-[D]\cdot K_{\omega}\\
        &=2([D]+K_{\omega})^2+2\\
        &\geq 0.
    \end{align*} Note that \[[\omega]\cdot(K_\omega-(-[D]-K_\omega))=[\omega]\cdot([D]+2K_{\omega})=2[\omega]\cdot([D]+K_{\omega})-[\omega]\cdot [D]<0.\] We thus have $SW(-[D]-K_{\omega})\neq 0$ by Corollary \ref{cor:SWnonzero}. Suppose $E_{\text{min}}\in \mathcal{E}_{\text{min}}(X,\omega)$ is any class satisfying $E_{\text{min}}\cdot [D]\geq 2$. By Lemma \ref{lem:E+D+K}, we also know $SW(E_{\text{min}}+[D]+K_{\omega})\neq 0$. For any $J\in \mathcal{J}(D)$, by the non-vanishing of Seiberg-Witten invariants, there will be two $J$-holomorphic subvarieties $\Theta_{-[D]-K_\omega}=\{(C_i,m_i)\}$ and $\Theta_{E_{\text{min}}+[D]+K_\omega}=\{(C_i',m_i')\}$ in classes $-[D]-K_{\omega}$ and $E_{\text{min}}+[D]+K_{\omega}$ respectively. Since $E_{\text{min}}$ can be represented by an embedded $J$-holomorphic sphere $C_{E_\text{min}}$ by Corollary \ref{cor:minimalarea}, any $J$-holomorphic subvariety representing the class $E_{\text{min}}$ must contain the component $C_{E_\text{min}}$ by positivity of intersection. Thus, $\mathcal{M}_{E_\text{min}}$ must be a single point by symplectic area consideration. This implies that the union of $\{(C_i,m_i)\}$ and $\{(C_i',m_i')\}$ is a single smooth exceptional sphere. Since $[\omega]\cdot(-[D]-K_{\omega})>0$, the class $-[D]-K_{\omega}$ is non-zero. We thus have $\Theta_{E_{\text{min}}+[D]+K_{\omega}}=\emptyset$ and $-[D]-K_{\omega}=E_{\text{min}}$ is represented by an embedded symplectic exceptional sphere.
\end{proof}


 Consequently, if the quasi-minimal pair $(X,\omega,D)$ is connected, Propositions \ref{prop:treeofsphere} and \ref{prop:-K-D=Emin} show that the quasi-minimal inequality is in fact an equality:
 \begin{align}\label{equation:=2}
    E_{\text{min}}\cdot [D]=(-[D]-K_\omega)\cdot [D]=2.
 \end{align}
 This gives rise to two possibilities: either the exceptional sphere $C_{E_\text{min}}$ is not a component of $D$ and intersects $D$ in two points, or $C_{E_\text{min}}$ is a component of $D$ and intersects exactly three other components of $D$. We refer to these as {\bf quasi-minimal pairs of the first/second kind}, and treat the two cases separately below.

\begin{example}\label{example:quasiminimal}
 Figure \ref{fig:quasiminimal} provides the example of a quasi-minimal pair of first kind in $\CC\PP^2\#8\overline{\CC\PP}^2$ and a quasi-minimal pair of second kind in $\CC\PP^2\#4\overline{\CC\PP}^2$. Note that both of them satisfy $-[D]-K_\omega=E_{\text{min}}$.
\end{example}

\begin{figure}[ht]
		\centering\includegraphics*[height=5cm, width=14cm]{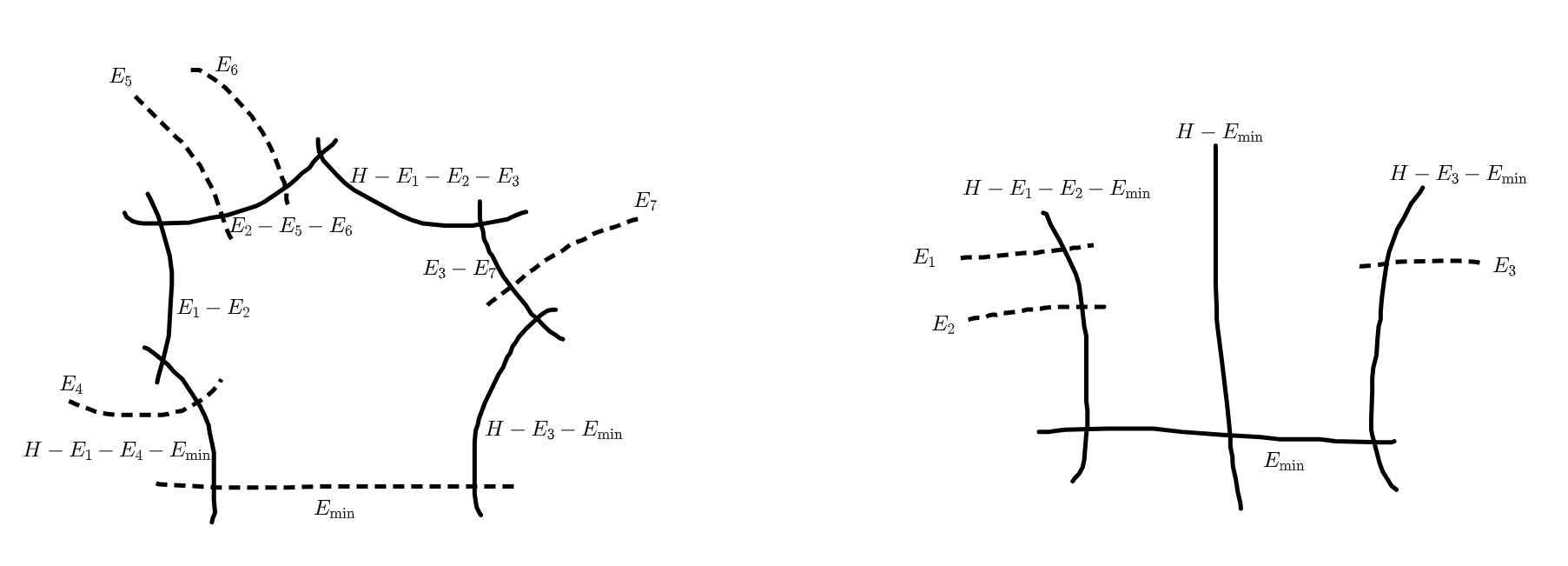}
 
		\caption{Two possible curve configurations for quasi-minimal pairs of the first/second kind. The dashed curves are not part of $D$. Neither configuration is partially minimal, due to the non-toric exceptional spheres indicated by the dashed curves. \label{fig:quasiminimal}}
	\end{figure}

\subsubsection{Quasi-minimal pairs of first kind}\label{section:LCY}

If $D\subseteq (X,\omega)$ is connected and quasi-minimal with $[\omega]\cdot([D]+K_{\omega})<0$, from Equation (\ref{equation:=2}) we see that $E_{\text{min}}=-K_{\omega}-[D]$ and $E_{\text{min}}$ is represented by a $J$-holomorphic sphere $C_{E_{\text{min}}}$ such that $[D]\cdot E_{\text{min}}=2$. Now,  let us further assume that $D$ is of first kind so that $C_{E_\text{min}}$ is not a component of $D$. It is possible for $C_{E_\text{min}}$ to intersect only one component of $D$ with contact order $2$, or to intersect two components of $D$ but at the same point. In such cases, we can always perform a $C^2$-small perturbation of $C_{E_{\text{min}}}$ to get another symplectic sphere $C_{E_{\text{min}}}'$ intersecting with $D$ at two distinct points transversely, as being symplectic is an open condition. This could even be achieved by making the perturbed sphere $J'$-holomorphic under a $C^1$-small perturbation of $J$ (see \cite[Proposition 3.3]{LiUsher}). After further perturbing  $C_{E_{\text{min}}}'$ into $C_{E_{\text{min}}}''$ to ensure the transverse intersections are $\omega$-orthogonally (\cite[Lemma 2.3]{Gompfnew}), $(\cup D_i)\cup C_{E_{\text{min}}}''$ will form a symplectic log Calabi-Yau divisor. In summary, we have the following result.
\begin{corollary}
    A quasi-minimal pair of first kind with $[\omega]\cdot([D]+K_{\omega})<0$ can be completed into a symplectic log Calabi-Yau pair by adding a symplectic exceptional sphere in class $E_{\text{min}}$. In particular, $D$ must be a chain of spheres.
\end{corollary}

For a quasi-minimal pair of the first kind $(X,\omega,D)$, viewed as part of a symplectic log Calabi-Yau pair $(X,\omega, D\cup C_{E_{\text{min}}})$, we may first carry out all non-toric blowdowns in the minimal reduction procedure for $(X,\omega, D\cup C_{E_{\text{min}}})$, and then perform only those toric blowdowns along $(-1)$-components that are not adjacent to $C_{E_{\text{min}}}$. This motivates Definition \ref{def:partialmin} and explains the term ``partially'' in this setting. In particular, when $D$ is quasi-minimal of first kind which is also partially minimal, the $(-1)$-component in $D$ can only appear at the first or the last one in the chain of spheres. We emphasize that, although the notion is motivated by quasi-minimal pairs of the first kind, it is also defined for quasi-minimal pairs of the second kind; we will rely on this in Section~\ref{section:psedolcy}.

\begin{example}\label{example:partiallyminimal}
    We continue Example \ref{example:quasiminimal} by performing partially minimal reductions for its first configuration. As shown in Figure \ref{fig:partiallyminimal}, we first blowdown the non-toric exceptional spheres in classes $E_4,E_5,E_6,E_7$ and then the toric exceptional spheres in classes $E_3,E_2,E_1$ (in order). It will be reduced to a divisor in $\CC\PP^2\#\overline{\CC\PP}^2$.
\end{example}

\begin{figure}[ht]
		\centering\includegraphics*[width=14cm,trim=0cm 13cm 0cm 0cm,clip]{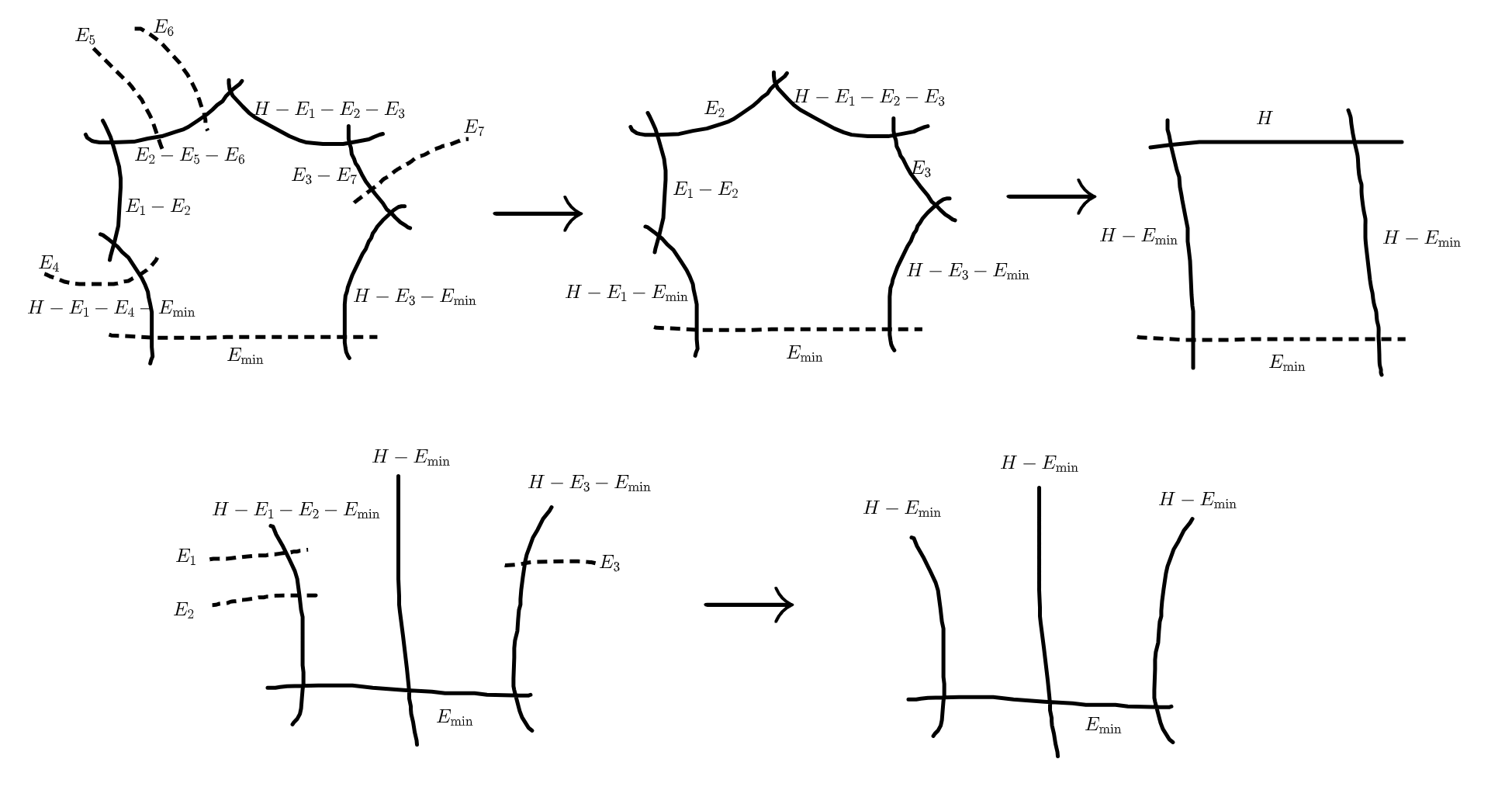}
		\caption{Partially minimal reduction of a quasi-minimal pair of first kind. \label{fig:partiallyminimal}}
	\end{figure}

\begin{proof}[Proof of Lemma \ref{lem:partialminimalred} for quasi-minimal pairs of first kind]
If there is an $E\in \mathcal{E}_{\omega}$ with $E\cdot E_{\text{min}}=0$ that has non-negative intersections with all components of $D$, then $E$ must be $D$-good by Lemma \ref{lem:specialclassDgood}. By Theorem \ref{thm:MO15} we can choose $J\in\mathcal{J}_{\text{emb}}(D,E)$ to realize $E$ as an embedded $J$-holomorphic sphere $C_E$. Note that 
\[E\cdot [D]=E\cdot(-K_{\omega}-E_{\text{min}})=1-E\cdot E_{\text{min}}=1\]
which would imply that $C_E$ is a non-toric exceptional sphere. If we perform non-toric blowdown on $C_E$ to get $(X',\omega',D')$, since $E\cdot E_{\text{min}}=0$, $E_{\text{min}}$ will persist to be a class in $\mathcal{E}_{\omega'}$ with minimal $\omega'$-symplectic area and \[E_{\text{min}}\cdot [D']=E_{\text{min}}\cdot([D]+E)=E_{\text{min}}\cdot [D]\geq 2.\]This means that $(X',\omega',D')$ is also quasi-minimal when $b_2(X')\geq 3$.

If there is a component $D_i$ of $D$ with $[D_i]^2=-1$ that satisfies $[D_i]\cdot([D]-[D_i])=2$, we can see that 
\[[D_i]\cdot E_{\text{min}}=[D_i]\cdot(-K_{\omega}-[D])=0.\]
This implies $D_i$ is a toric exceptional sphere which does not intersect $C_{E_{\text{min}}}$. As above, if we perform toric blowdown on $D_i$, $C_{E_{min}}$ will persist and we then obtain another pair which is quasi-minimal when $b_2\geq 3$.

Now we just repeat the above process until we reach either a partially minimal pair or a pair with $b_2\leq 2$.
\end{proof}

\subsubsection{Quasi-minimal pairs of second kind}\label{section:psedolcy}
Now let us deal with the pair $(X,\omega,D)$ with $[\omega]\cdot (K_{\omega}+[D])<0$ which is quasi-minimal of second kind. This means that there is a component $D_0$ of $D$ with $[D_0]=E_{\text{min}}$ and $[D_0]\cdot ([D]-[D_0])=3$. By Proposition \ref{prop:treeofsphere}, the dual graph $\Gamma(D)$ has no loop so that if we delete the component $D_0$ there will be three connected components. We denote them by $U_1\cup\cdots\cup U_a$, $V_1\cup\cdots\cup V_b$ and $W_1\cup\cdots\cup W_c$ where $[U_1]\cdot [D_0]=[V_1]\cdot [D_0]=[W_1]\cdot[D_0]=1$ and all the other $U,V,W$-components are disjoint from $D_0$. 

\begin{lemma}\label{lem:quasisecondb2geq5}
   When $X$ is a rational manifold with $b_2(X)\geq 5$, there is no quasi-minimal pair $(X,\omega,D)$ of second kind with $[\omega]\cdot (K_{\omega}+[D])<0$ which is also partially minimal. 
\end{lemma}

\begin{proof}
    First, by Corollary \ref{cor:secondminexceptional}, let us choose $E_2,\cdots,E_n\in \mathcal{E}_{\omega}(X)$ where $n=b_2(X)-1\geq 4$, $E_n=E_{\text{min}}$ and $E_{j-1}$ is the secondary minimal area exceptional class of $E_j$ for $3\leq j\leq n$. If $(X,\omega,D)$ is partially minimal, $E_{n-1}$ will have negative intersection with some component of $D$, say $U_j$ without loss of generality. By half-toric blowdown of $D_0$ with respect to the divisor $D_0\cup U_1\cup\cdots\cup U_a$, there will be $(X',\omega',U')$ where $U'=U_1'\cup\cdots\cup U_a'$. For any $J_U\in\mathcal{J}(U')$, $E_{n-1}$ has embedded $J_U$-holomorphic representative by Corollary \ref{cor:minimalarea}. Since $E_{n-1}\cdot [U_j']=E_{n-1}\cdot[U_j]<0$, by positivity of intersection, $[U_j']=E_{n-1}$. Then we must have $j=1$ since otherwise, \[[U_j]\cdot ([D]-[U_j])=[U_j]\cdot(-K_{\omega}-E_{\text{min}}-[U_j])=1-0-(-1)=2,\]
which contradicts with the partially minimal assumption. Also note that 
\[[U_1]\cdot ([D]-[U_1])=[U_1]\cdot(-K_{\omega}-E_{\text{min}}-[U_1])=[U_1]\cdot(-K_\omega-[U_1])-[U_1]\cdot E_{\text{rm}}=2-1=1,\]
which implies $a=1$ and $[U_1]=E_{n-1}-E_n$. 

Next, let us take $E_{n-2}$ and assume it has negative intersection with some $V$-component. Consider the divisor $U_1\cup D_0\cup V_1\cup\cdots\cup V_b$. By toric blowdown of $D_0$, there will be the half-toric exceptional sphere $U_1'$ by the previous discussion. By half-toric blowdown $U_1'$ we then get $(X'',\omega'',V'')$ with $V''=V_1''\cup\cdots\cup V_b''$. Again, by choosing any $J_V\in\mathcal{J}(V'')$ and represent $E_{n-2}$ by an embedded $J_V$-holomorphic sphere, the same argument as above shows that $b=1$ and $[V_1]=E_{n-2}-E_{n-1}-E_{n}$.

Then, we take the divisor $U_1\cup D_0\cup W_1\cup\cdots\cup W_c$ and consider the divisor $W'':=W_1''\cup \cdots\cup W_c''$ obtained by toric blowdown of $D_0$ followed by half-toric blowdown of $U_1'$ as the previous paragraph. Since $[U_1]\cdot[W_j]=[V_1]\cdot[W_j]=0$ for all $1\leq j\leq c$, we see that 
\begin{align*}
    E_{n-2}\cdot [W_1'']&=E_{n-2}\cdot([W_1]+E_{n-1}+E_n)=E_{n-2}\cdot[W_1]\\
    &=(E_{n-2}-[V_1]-[U_1])\cdot[W_1]=2E_{n}\cdot[W_1]\\
    &=2[D_0]\cdot [W_1]=2,
\end{align*}
and for $j>1$,
\begin{align*}
    E_{n-2}\cdot [W_j'']&=E_{n-2}\cdot[W_j]=(E_{n-2}-[V_1]-[U_1])\cdot[W_j]\\
    &=2E_{n}\cdot[W_1]=2[D_0]\cdot [W_j]=0.
\end{align*}
As a result, $E_{n-2}$ is $W''$-good so that we can represent $E_{n-2}$ by embedded $J_W$-holomorphic sphere $C_{E_{n-2}}$ for generic $J_W\in\mathcal{J}(W'')$. As Section \ref{section:LCY}, after small isotopy, we may assume $C_{E_{n-2}}\cup W''$ is a symplectic divisor. Note that 
\begin{align*}
    E_{n-2}+[W_1'']+\cdots+[W_c'']&=E_{n-2}+E_{n-1}+E_n+[W_1]+\cdots+[W_c]\\
    &=E_{n-1}+2E_n+[U_1]+[V_1]+[D_0]+[W_1]+\cdots+[W_c]\\
    &=E_{n-1}+2E_n-K_{\omega}-E_n=-K_{\omega}+E_{n}+E_{n-1}.
\end{align*}
Therefore, $C_{E_{n-2}}\cup W''$ is a symplectic log Calabi-Yau divisor, which must be a cycle of spheres. This implies $c=1$ and $[W_1]=-K_{\omega}-E_{n-2}$.

Finally, by Corollary \ref{cor:secondminexceptional}, there exists $E_1'\in \mathcal{E}_{\omega}(X)$ with $E_1'\cdot E_n=E_1'\cdot E_{n-1}=0$ and $E_1'\cdot E_{n-2}\in\{0,1\}$. Since we now know $D$ only has four components $U_1,V_1,W_1,D_0$, it is easy to check they all have non-negative intersections with $E_1'$, which contradicts with the partially minimal assumption.
\end{proof}

Now let us discuss the cases when $b_2(X)\leq 4$ using Zhang's theorem on $J$-holomorphic exceptional spheres on $\CC\PP^2\#2\overline{\CC\PP}^2$.

\begin{lemma}\label{lem:quasisecondb2geq4}
    When $X$ is a rational manifold with $b_2(X)=4$, there is no quasi-minimal pair $(X,\omega,D)$ of second kind with $[\omega]\cdot (K_{\omega}+[D])<0$ which is also partially minimal. 
\end{lemma}

\begin{proof}
    By Corollary \ref{cor:secondminexceptional}, there are $E_1,E_1',E_2,E_3\in\mathcal{E}_{\omega}(X)$ such that $E_3=E_{\text{min}}$, $E_2$ is the secondary minimal area exceptional class of $E_3$ and 
    \begin{enumerate}
        \item either $E_1\cdot E_1'=0,E_1\cdot E_2=E_1'\cdot E_2=1$; 
        \item or $E_1\cdot E_1'=E_1'\cdot E_2=1,E_1\cdot E_2=0$.
    \end{enumerate}
    Note that the same argument in the proof of Lemma \ref{lem:quasisecondb2geq5} implies that we can assume $a=1$ and $[U_1]=E_2-E_3$. Consider the divisor $V_1'\cup\cdots\cup V_b'\cup W_1'\cup\cdots\cup W_c'$ obtained from the toric blowdown of $D_0$ from the divisor $V_1\cup\cdots\cup V_b\cup D_0\cup W_1\cup\cdots\cup W_c$. We can choose any divisor-adapted almost complex structure $J$ and at least one of $E_1,E_1'$ will have an embedded $J$-holomorphic representative by Theorem \ref{thm:2blowup}. By the partially minimal condition, this embedded $J$-holomorphic representative must be the same as some component in $V_1'\cup\cdots\cup V_b'\cup W_1'\cup\cdots\cup W_c'$.

   Let us first suppose that case (1) happens. Without loss of generality, we may assume $E_1$ has an embedded $J$-holomorphic representative. Then the same argument in Lemma \ref{lem:quasisecondb2geq5} shows that we can further assume $b=1$ and $[V_1]=E_1-E_3$. Observe that when $1<j\leq c$, $[W_j]\cdot [D_0]=[W_j]\cdot [U_1]=[W_j]\cdot [V_1]=0$, which implies that $[W_j]\cdot E_3=[W_j]\cdot E_2=[W_j]\cdot E_1=0$. Since $b_2^-(X)=3$, we must have $[W_j]^2\geq 0$. Note that $[W_j]\cdot (E_2+E_1)=0$ and $(E_2+E_1)^2=0$. By light cone lemma, $[W_j]$ is a positive multiple of $E_2+E_1$. Thus, we can assume $[W_2]+\cdots+[W_c]=k(E_2+E_1)$ and $[W_1]=-K_\omega-E_1-E_2-k(E_2+E_1)$. By adjunction formula,
\begin{align*}
    -2=&[W_1]^2+K_{\omega}\cdot [W_1]=(-K_\omega-E_1-E_2-k(E_2+E_1))\cdot(-E_1-E_2-k(E_2+E_1))\\
    =&(-K_\omega-(k+1)(E_1+E_2))\cdot((-k-1)(E_1+E_2))=-2k-2.
\end{align*}
We see $k=0$ so that $c=1$ and $[W_c]=-K_{\omega}-E_1-E_2$. It then follows that $E_1'$ will have non-negative intersections with all components of $D$, which contradicts with the partially minimal assumption.

   Now suppose case (2) happens. If $E_1$ has an embedded $J$-holomorphic representative, then by the same reasoning as above, $[V_1]=E_1-E_3$. But this is impossible since $[V_1]\cdot[U_1]=-1$. So we may assume $E_1'$ has an embedded $J$-holomorphic representative and $[V_1]=E_1'-E_3$. One can then see that by replacing $E_1,E_1',E_2$ by $E_2,E_1,E_1'$ respectively in the previous paragraph, $E_1$ will be an exceptional class having non-negative intersections with all components of $D$. Again, this contradicts with the partially minimal assumption.
    
\end{proof}

\begin{corollary}\label{cor:quasiminmalsecondkindreduction}
  Any connected quasi-minimal pair $(X,\omega,D)$ of second kind with $[\omega]\cdot (K_{\omega}+[D])<0$, where $X$ is a rational manifold, can be obtained from a sequence of blowups from a divisor in $\CC\PP^2, S^2\times S^2$ or $\CC\PP^2\#\overline{\CC\PP}^2$.
\end{corollary}

\begin{proof}
By Lemma \ref{lem:partialminimalred}, \ref{lem:quasisecondb2geq4} and \ref{lem:quasisecondb2geq5}, we only need to show that such a divisor $D$ in $X=\CC\PP^2\#2\overline{\CC\PP}^2$ can be reduced by blowdown. Choosing any $J\in\mathcal{J}(D)$, by Theorem \ref{thm:2blowup}, there will be a $J$-holomorphic exceptional curve $C$ with $[C]\cdot [D_0]\in \{0,1\}$. If $C$ is a component of $D$, then \[[C]\cdot ([D]-[C])=[C]\cdot (-K_\omega-[D_0]-[C])\in\{1,2\}\] so that we can either toric or half-toric blowdown the exceptional component $C$. If $C$ is not a component of $D$, then \[[C]\cdot [D]=[C]\cdot (-K_\omega-[D_0])\in\{0,1\}.\] By positivity of intersection, $C$ must be a non-toric exceptional sphere (when $[C]\cdot [D]=0$, $C$ will intersect $D_0$; when $[C]\cdot [D]=1$, $C$ will intersect other component of $D$) and we can perform the non-toric blowdown.
\end{proof}

\begin{rmk}
    Quasi-minimal pairs of second kind can be understood as the blowup of `abnormal' log Calabi-Yau divisors. Recall that our definition of symplectic divisors requires that no three components share a common intersection point. Under this assumption, the configuration of log Calabi-Yau divisors is constrained to form a cycle of spheres.  The simple normal crossing condition excludes configurations like three concurrent lines in $\CC\PP^2$. For instance, the partially minimal model for the divisor of second kind in Example \ref{example:partiallyminimal} can actually be obtained from the blowup at the point of concurrency of three lines in $\CC\PP^2$.
\end{rmk}

\begin{example}\label{example:partiallyminimal2}
    We continue Example \ref{example:quasiminimal} by performing partially minimal reductions for its second configuration. As shown in Figure \ref{fig:partiallyminimal2}, we blowdown the non-toric exceptional spheres in classes $E_1,E_2,E_3$. It will be reduced to a divisor in $\CC\PP^2\#\overline{\CC\PP}^2$.
\end{example}

\begin{figure}[ht]
		\centering\includegraphics*[width=14cm,trim=0cm 0cm 0cm 13cm,clip]{partiallyminimal}
		\caption{Partially minimal reduction of a quasi-minimal pair of second kind. \label{fig:partiallyminimal2}}
	\end{figure}

\begin{proof}[Proof of Lemma \ref{lem:partialminimalred} for quasi-minimal pairs of second kind]
    By Corollary \ref{cor:quasiminmalsecondkindreduction}, we can reduce any quasi-minimal pair of second kind to a pair with $b_2\leq 2$ via blowdowns.
\end{proof}

\subsection{Affine rulings from unicuspidal curves}\label{section:findcusp}
From now on, $(X,\omega,D)$ always denotes a connected pair, where $X$ is a rational manifold. We will prove the main Theorem \ref{thm:firstmain} for the rational case in the next section. By Lemma \ref{lem:partialminimalred}, we divide the argument into two parts: quasi-minimal pairs of the first kind and pairs obtained by blowing up pairs with $b_2\leq 2$ (i.e. Lemma \ref{lem:threecases}).
As outlined in the introduction, our strategy for proving symplectic affine-ruledness is to find a suitable homology class represented by unicuspidal rational curves (or smooth rational curves for specific cases in Lemma \ref{lem:threecases}) that foliate $X$ and meet the divisor $D$ precisely at their cusps. This motivates the following definition.

\begin{definition}\label{def:homaffruled}
A pair $(X,\omega,D)$ is called {\bf $(p,q)$-homologically affine-ruled} if there are
\begin{itemize}
    \item relatively prime integers $p,q\in \ZZ_{\geq 0}$;
    \item two intersecting components $D_a,D_b$ of $D$ with $[D_a]\cdot[D_b]=1$;
    \item a class $A\in H_2(X;\ZZ)$ satisfying $A\cdot[D_a]=p,A\cdot[D_b]=q$ and has trivial intersection numbers with any other components of $D$;
    \item $A^2=pq$ and $A\cdot K_{\omega}=-p-q-1$.
\end{itemize}

\end{definition}

\begin{rmk}
    In the above definition, we allow $(p,q)=(1,0)$ or $(0,1)$ and adopt the convention that the weight sequence for such pairs is empty. In such a situation, we also think of $(X,\omega,D)$ itself as the outcome of the normal crossing resolution of this $(1,0)$ or $(0,1)$-cusp. 
\end{rmk}

When such a class $A$ in Definition \ref{def:homaffruled} exists, let us first pretend that there is a $(p,q)$-cusp at $D_a\cap D_b$. Assume the weight sequence $\mathcal{W}(p,q)=(m_1,\cdots,m_k)$. Then we can topologically perform toric blowups of $X$ for $k$ times according to the pattern of normal crossing resolution. Let $\tilde{X}$ be the manifold after blowups and $\tilde{D}$ be the total transform of $D$. Take the resolution class $\tilde{A}=A-\sum_{i=1}^km_iE_i\in H_2(\tilde{X};\ZZ)$.  By Fact \ref{fact:boxdiagram}, it satisfies 
\begin{equation}\label{equation:index}
\begin{gathered}
\tilde{A}^2=A^2-\sum_{i=1}^km_i^2=pq-pq=0,\\
\tilde{A}\cdot (K_{\omega}+\sum_{i=1}^kE_i)=A\cdot K_\omega+\sum_{i=1}^km_i=(-p-q-1)+(p+q-1)=-2.
\end{gathered} 
\end{equation}
Moreover, the class $\tilde{A}$ will have intersection number $1$ with the exceptional component $C_{E_k}$ of $\tilde{D}$ occurred in the last blowup, and $0$ with all other components of $\tilde{D}$. The following lemma is an analogue of Theorem \ref{thm:fiberclass} for rational manifolds, which describes the properties of a fiber class.

\begin{lemma}\label{lem:tildeAmoduli}
    Let $\tilde{J}$ be a tame almost complex structure on $\tilde{X}$ with canonical class $K_{\tilde{J}}=K_\omega+\sum_{i=1}^kE_i$. Assume all components in $\tilde{D}$ are $\tilde{J}$-holomorphic. If $SW(\tilde{A})\neq 0$ and $\tilde{A}$ is represented by an embedded $\tilde{J}$-holomorphic sphere, then the moduli space of $\tilde{J}$-holomorphic subvarieties $\mathcal{M}_{\tilde{A}}$ can be naturally identified with $C_{E_k}$, with finitely many points representing reducible subvarieties. Moreover, any component of $\tilde{D}$ other than $C_{E_k}$ is an irreducible component of an element of $\mathcal{M}_{\tilde{A}}$.
\end{lemma}
\begin{proof}
    First, since $SW(\tilde{A})\neq 0$ and the index $I(\tilde{A})=2$, we know there is a $\tilde{J}$-holomorphic subvariety passing through each point on $C_{E_k}$. The uniqueness follows from the observation that two distinct subvarieties in $\mathcal{M}_{\tilde{A}}$ can not have any common irreducible component. To see this, note that after allowing some coefficients to be zero, we may write them as $\{(C_i,k_i)\}$ and $\{(C_i,k_i')\}$. Let $\{(C_i,\min\{k_i,k_i'\})\}$ be the common part of these two subvarieties and $B$ be the class of this common part ($B\neq \tilde{A}$). Then $(\tilde{A}-B)^2$ can be interpreted as the intersection number between two subvarieties with no common irreducible components, which must be $\geq 0$ by positivity of intersection. Since $\tilde{A}$ has an embedded representative, positivity of intersection also implies that $\tilde{A}\cdot(\tilde{A}-B)= 0$. By light cone lemma (\cite[Lemma 3.7]{McDuffGT}), $\tilde{A}$ is a multiple of $\tilde{A}-B$, which is impossible unless $B=0$. Therefore, we have a bijective correspondence between $\mathcal{M}_{\tilde{A}}$ and $C_{E_k}$. 
    
    Since $\tilde{A}^2=0$, each reducible representative of $\tilde{A}$ must contain a component of negative self-intersection. By the observation above, there are only finitely many reducible elements in $\mathcal{M}_{\tilde{A}}$ since $b_2^-(\tilde{X})$ is finite.

    Finally, for any component $\tilde{C}$ of $\tilde{D}$ other than $C_{E_k}$, we have $[\tilde{C}]\cdot \tilde{A}=0$. However, there is a $J$-holomorphic subvariety in class $\tilde{A}$ passing through the point on $\tilde{C}$ since $SW(\tilde{A})\neq 0$ and $I(A)$=2. Again, positivity of intersection implies that $\tilde{C}$ is contained as an irreducible component of that subvariety.
\end{proof}

\begin{rmk}
    Zhang \cite[Theorem 1.3]{Zhangmoduli} actually shows a much more general result that if $J$ is {\bf any} tame almost complex structure on a rational manifold and $A\in H_2(X;\ZZ)$ is a primitive class represented by an embedded $J$-holomorphic sphere, then the moduli space of $J$-holomorphic subvarieties in class $A$ must be homeomorphic to $\CC\PP^l$ where $l=\max\{0,A^2+1\}$. Zhang's proof involves the choice of a $J$-nef spherical class $B$ with $B\cdot A=1$ and interprets $\CC\PP^l$ as the symmetric product $\text{Sym}^l(S)$ where $S$ is an embedded $J$-holomorphic sphere in class $A+B$. In the proof of Lemma \ref{lem:homologicallyaffineruled} we have a natural candidate $C_{E_k}$ which plays the role of this sphere $S$.
\end{rmk}

The following lemma relates the homologically and geometrically affine-ruledness.

\begin{lemma}\label{lem:homologicallyaffineruled}
    Let $(X,\omega,D)$ be a homologically affine-ruled pair and $A\in H_2(X;\ZZ)$ be the class in Definition \ref{def:homaffruled}. Let $\tilde{X}$ be the topological blowup of $X$ according to the pattern of normal crossing resolution of a $(p,q)$-cusp at the intersection point $D_a\cap D_b$ and $\tilde{D}\subseteq \tilde{X}$ be the total transform of $D$. Assume that the resolution class $\tilde{A}$ is $\tilde{D}$-good. Then there is an open dense subset in $X\setminus D$ foliated by $(p,q)$-unicuspidal rational curves. In particular, $(X,\omega,D)$ is symplectic affine-ruled.
\end{lemma}

\begin{proof}
   We can make the topological pair $(\tilde{X},\tilde{D})$ into a symplectic divisor $(\tilde{X},\tilde{\omega},\tilde{D})$ by choosing some symplectic form $\tilde{\omega}$ with $K_{\tilde{\omega}}=K_\omega+\sum_{i=1}^kE_i$. Then by McDuff-Opshtein's criterion Theorem \ref{thm:MO15}, for generic $\tilde{J}\in \mathcal{J}(\tilde{D})$, $\tilde{A}$ can be represented by an embedded $\tilde{J}$-holomorphic sphere since $\tilde{A}$ is assumed to be $\tilde{D}$-good. The condition $\tilde{A}^2=0$ further implies that $\tilde{A}$ is $\tilde{J}$-nef so that any $\tilde{J}$-holomorphic subvariety in class $\tilde{A}$ must be a tree of embedded spheres by Theorem \ref{thm:lizhangnef}. By Lemma \ref{lem:tildeAmoduli}, there will be finitely many embedded $\tilde{J}$-holomorphic embedded spheres $C_1,\cdots,C_l$ such that $(\cup C_i)\cup(D\setminus C_{E_k})$ is the union of configurations of all reducible elements in $\mathcal{M}_{\tilde{A}}$. Now consider the evaluation map
\[\text{ev}:\mathring{\mathcal{M}}_{0,1}(\tilde{A};\tilde{J})\rightarrow \tilde{X}\setminus ((\cup C_i)\cup(\tilde{D}\setminus C_{E_k}))\]
where $\mathring{\mathcal{M}}_{0,1}(\tilde{A};\tilde{J})$ denotes the moduli space $\mathcal{M}_{0,1}(\tilde{A};\tilde{J})$ with finitely many elements represented by components in $\tilde{D}$ removed.\footnote{As in the proof of Theorem \ref{thm:ruledmain}, to make the $(S^2\setminus\{\text{pt}\})$-bundle trivial, we may need to choose an additional sphere and delete the corresponding elements from $\mathcal{M}_{0,1}(\tilde{A};\tilde{J})$ when $\tilde{D}\setminus C_{E_k}$ is empty.}
 The map $\text{ev}$ must be bijective by Lemma \ref{lem:tildeAmoduli}. The same trick in the proof of Theorem \ref{thm:ruledmain} would then imply that $\text{ev}$ is indeed a diffeomorphism. After further removing the section sphere $C_{E_k}$, we see that $\tilde{X}\setminus{((\cup C_i)\cup \tilde{D})}$ must be diffeomorphic to a trivial $(S^2\setminus\{\text{pt}\})$-bundle whose fibers are $\tilde{J}$-holomorphic.
   
   Since $\tilde{J}$ is integrable near $\tilde{D}$, it naturally corresponds to some $J\in\mathcal{J}(D)$ by contracting holomorphic exceptional spheres. By Proposition \ref{prop:modulirelation}, embedded $\tilde{J}$-holomorphic spheres in class $\tilde{A}$ will descend to $(p,q)$-unicuspidal rational curves in class $A$. By the identification between $(\tilde{X}\setminus{\tilde{D}},\tilde{J}|_{\tilde{X}\setminus\tilde{D}})$ and $(X\setminus D,J_{X\setminus D})$, we can also view $C_i$ as $J$-holomorphic spheres in $X$. Then $X\setminus((\cup C_i)\cup D)$ will be foliated by $(p,q)$-unicuspidal rational curves and diffeomorphic to the trivial $(S^2\setminus\{\text{pt}\})$-bundle whose fibers are $J$-holomorphic. In particular, all the fibers are $\omega$-symplectic with the same symplectic area $\omega(A)$. Therefore $(X,\omega,D)$ is symplectic affine-ruled.
\end{proof}

\begin{rmk}
    Let us emphasize that in the above proof, we do not treat $(\tilde{X},\tilde{\omega},\tilde{D})$ as the symplectic blowup of $(X,\omega,D)$ as Section \ref{sec:divisoroperation}. Instead, it should be viewed as the topological blowup $(\tilde{X},\tilde{D})$, using the construction of complex blowups, equipped with some symplectic form $\tilde{\omega}$.
\end{rmk}

\begin{corollary}\label{cor:blowupaffineruled}
    Under the same settings and assumptions as Lemma \ref{lem:homologicallyaffineruled}, the non-toric, toric, half-toric or exterior blowup $(X',\omega',D')$ of $(X,\omega,D)$ is still symplectic affine-ruled. 
\end{corollary}
    
\begin{proof}
    Write $H_2(X';\ZZ)=H_2(X;\ZZ)\oplus \ZZ e$. If the blowup is a toric blowup at $D_a\cap D_b$, then we can just take the pair of relatively prime numbers $(\text{min}\{p,q\},|p-q|)$ and  $A':=A-m_1e\in H_2(X';\ZZ)$ and choose either the proper transform $D_a'$ (if $p>q$) or $D'_b$ (if $p<q$) along with the exceptional curve to be the intersecting components for $A'$. For such choices, the normal crossing resolution can be identified with $(\tilde{X},\tilde{D})$ and the resolution class of $A'$ is $\tilde{A}$. So the $\tilde{D}$-good condition is naturally satisfied by the assumption in Lemma \ref{lem:homologicallyaffineruled} and we know $(X',\omega',D')$ is symplectic affine-ruled by Lemma \ref{lem:homologicallyaffineruled}.
    
    If the blowup is not a toric blowup at $D_a\cap D_b$, then we just take the same relatively prime numbers $(p,q)$, class $A$ viewed as in $H_2(X';\ZZ)$ and the proper transforms of $D_a,D_b$ as the intersecting components. Note that the normal crossing resolution $(\tilde{X}',\tilde{D}')$ can also be viewed as the blowup of $(\tilde{X},\tilde{D})$ and the resolution class $\tilde{A'}$ is the same as $\tilde{A}$ viewed as class in $H_2(\tilde{X'};\ZZ)=H_2(\tilde{X};\ZZ)\oplus \ZZ e$. $\tilde{A}'$ is then $\tilde{D}'$-good by Lemma \ref{lem:criterionforDgood} and we have the symplectic affine-ruledness.
\end{proof}

\begin{rmk}
    The above Corollary \ref{cor:blowupaffineruled} provides some clues for the birational invariance of symplectic affine-ruledness, generalizing the well-known birational invariance of symplectic ruledness in the absolute settings. In general, it is an interesting question whether this birational invariance still holds for non-connected divisors or outside the category of log Kodaira dimension $-\infty$. 
\end{rmk}

 Let us now deal with quasi-minimal pairs of the first kind. Our goal is to find such a class $A$ in Lemma \ref{lem:homologicallyaffineruled} for all such pairs. The main idea to to select $A$ as a positive linear combination of some components in $D$, since we wish them to be the configuration of some subvariety representing the class $A$. Let us introduce some notions first. If $\{a_i\}_{1\leq i\leq k}$ is a sequence of integers, we define another sequence $\{c_i\}_{1\leq i\leq k}$ called its {\bf associated sequence} as follows: let 
\begin{align*}
c_1&=1,\\
c_2&=a_1,\\
c_{i}&=a_{i-1}c_{i-1}-c_{i-2} \text{ for } i\geq 3.
\end{align*}

For a quasi-minimal pair $(X,\omega,D)$ of first kind, since $D$ is a chain of spheres, we can label its components as $D_1,\cdots,D_n$ such that $[D_i]\cdot [D_{i+1}]=1$ for $1\leq i\leq n-1$ and $[D_i]\cdot[D_j]=0$ if $|i-j|>1$. By reversing the indices, there are two ways to label the components. We say $(X,\omega,D)$ contains an {\bf admissible subchain} if there is a labeling of $D$ and $1\leq k<n$ such that the associated sequence $\{c_i\}_{1\leq i\leq k}$ of $\{a_i=-[D_i]^2\}_{1\leq i\leq k}$ satisfies (we make the convention that $c_0=0$. )
\begin{itemize}
    \item $c_i\geq 0$ for all $1\leq i\leq k$;
    \item $c_{k-1}-c_ka_k>0$.
\end{itemize}

\begin{lemma}\label{lem:chainimplyhaffruled}
If the quasi-minimal pair of first kind $(X,\omega,D)$  contains an admissible subchain, then it must be homologically affine-ruled.    
\end{lemma}

\begin{proof}
    Let us make the choice
    \begin{align}\label{choice}\tag{$*$}
    \begin{split}
        (p,q) &=(c_{k-1}-c_ka_k,c_k);\\
        D_a &=D_k,D_b=D_{k+1};\\
        A &=\sum_{i=1}^kc_i[D_i].
    \end{split}
    \end{align}
        
    Thus, we can verify that
    \begin{itemize}
        \item $p,q\in \ZZ_+$ and are relatively prime by induction;
        \item $[D_{k}]\cdot[D_{k+1}]=1$;
        \item the intersection number between $A$ and components in $D$: 
        \[A\cdot[D_{k+1}]=(\sum_{i=1}^kc_i[D_i])\cdot[D_{k+1}]=c_k=q;\]
         \[A\cdot[D_k]=(\sum_{i=1}^kc_i[D_i])\cdot[D_{k}]=c_{k-1}-c_ka_k=p;\]
         \[A\cdot[D_j]=(\sum_{i=1}^kc_i[D_i])\cdot[D_{j}]=\begin{cases}c_{j-1}-c_ja_j+c_{j+1}=0,\text{for } j<k;\\
0,\text{for } j>k+1
         \end{cases}\]
        \item by the definition of associated sequence we have 
        \begin{align*}
            A^2&=(\sum_{i=1}^kc_i[D_i])^2=-\sum_{i=1}^kc_i^2a_i+2\sum_{i=1}^{k-1}c_{i}c_{i+1}\\
            &=-a_1-c_k^2a_k-\sum_{i=2}^{k-1}c_i(c_{i+1}+c_{i-1})+2\sum_{i=1}^{k-1}c_{i}c_{i+1}\\
            &=-a_1-c_k^2a_k+c_1c_2+c_{k-1}c_k\\
            &=c_k(c_{k-1}-c_ka_k)=pq;
        \end{align*}
        \item since $D_i$'s are spheres, by adjunction formula we have 
        \begin{align*}
            A\cdot K_{\omega}&=(\sum_{i=1}^kc_i[D_i])\cdot K_{\omega}=\sum_{i=1}^kc_i(-2+a_i)\\
            &=-2\sum_{i=1}^kc_i+\sum_{i=2}^{k-1}(c_{i-1}+c_{i+1})+c_1a_1+c_ka_k\\
            &=-c_1-c_2-c_{k-1}-c_k+c_1a_1+c_ka_k=-p-q-1.
        \end{align*}
    \end{itemize}
    Therefore, our choice meets all the requirements in Definition \ref{def:homaffruled}.
\end{proof}

\begin{lemma}\label{lem:resolutionDgood}
    Let $(X,\omega,D)$ be a quasi-minimal pair of first kind with $[\omega]\cdot(K_\omega+[D])<0$. With the choice (\ref{choice}) made in Lemma \ref{lem:chainimplyhaffruled}, let $\tilde{X}$ be the blowups of $X$ according to the pattern of the normal crossing resolution of a $(p,q)$-cusp at $D_a\cap D_b$ and $\tilde{D}\subseteq \tilde{X}$ be the total transform  of $D\subseteq X$. Then the resolution class $\tilde{A}:=A-\sum_{i=1}^lm_iE_i\in H_2(\tilde{X};\ZZ)$ is $\tilde{D}$-good, where $(m_1,\cdots,m_l)$ is the weight sequence $\mathcal{W}(p,q)$.
\end{lemma}

\begin{proof}
    Let us equip $\tilde{X}$ with a symplectic form $\tilde{\omega}$ with $K_{\tilde{\omega}}=K_\omega+\sum_{i=1}^l E_i$ making $\tilde{D}$ a symplectic divisor. By the symplectic blowup construction, we may assume that all exceptional classes $E_i$'s have sufficiently small $\tilde{\omega}$-symplectic area so that $[\tilde{\omega}]\cdot K_{\tilde{\omega}}<0$. Then we have the following observations.
    \begin{itemize}
        \item $I(\tilde{A})=\tilde{A^2}-\tilde{A}\cdot K_{\tilde{\omega}}=\tilde{A}^2-\tilde{A}\cdot(K_{\omega}+E_1+\cdots+E_l)=2$ by Equation (\ref{equation:index}), and $[\tilde{\omega}]\cdot(K_{\tilde{\omega}}-\tilde{A})<0$ since $[\tilde{\omega}]\cdot K_{\tilde{\omega}}<0$. So $SW(\tilde{A})\neq 0$ by Corollary \ref{cor:SWnonzero}.
        \item $\tilde{A}$ must be primitive since $\tilde{A}\cdot E_l=1$.
        \item By the property of normal crossing resolution and the fact that $A$ only has non-trivial intersection with $D_a$ and $D_b$, $\tilde{A}$ has intersection pairing $1$ with $E_l$ and $0$ with all the classes of other components in $\tilde{D}$.
        \item Let $\tilde{D}_j$ be the proper transform of $D_j$. Then we can write \begin{align*}
            \tilde{A}&=A-\sum_{i=1}^lm_iE_i=\sum_{i=1}^kc_i[D_i]-\sum_{i=1}^lm_iE_i=\sum_{i=1}^kc_i[\tilde{D}_i]+c_k([D_k]-[\tilde{D}_k])-\sum_{i=1}^lm_iE_i.
        \end{align*} By Lemma \ref{lem:positivecombination}, $c_k([D_k]-[\tilde{D}_k])-\sum_{i=1}^lm_iE_i$ is a non-negative linear combination of the components of $\tilde{D}$ and so is $\tilde{A}$ since all $c_i$'s are non-negative. Choose any $J\in\mathcal{J}(\tilde{D})$. If $C$ is an irreducible $J$-holomorphic curve which is not a component of $\tilde{D}$, we have $\tilde{A}\cdot [C]\geq 0$ by positivity of intersection between $C$ and components in $\tilde{D}$. On the other hand, the previous bullet implies that $\tilde{A}\cdot [\tilde{D_i}]\geq 0$ for any component $\tilde{D_i}$ of $\tilde{D}$. This shows that $\tilde{A}$ is $\tilde{J}$-nef. As a consequence, for any $E\in \mathcal{E}_{\omega}$, since $SW(E)\neq 0$, we must have $\tilde{A}\cdot E\geq 0$.
        
    \end{itemize}
    Therefore, we see that $\tilde{A}$ satisfies all the conditions for $\tilde{D}$-goodness in Definition \ref{def:Dgood}.
\end{proof}

We now need the following lemma which helps us to find admissble subchains. It is a consequence from the observation of the minimal models for log Calabi-Yau pairs in Section \ref{section:divisorsinminimal}.

\begin{lemma}\label{lem:partialminimalgoodchain}
    For a quasi-minimal pair $(X,\omega,D)$ of first kind which is also partially minimal, we can label the components in $D$ as $D_1\cup\cdots\cup D_l$ with $[D_i]\cdot [D_{i+1}]=1$ for all $1\leq i\leq l-1$ and find some $1\leq k\leq l-1$ such that 
    
    \begin{itemize}
        \item either $[D_s]^2\leq -2$ for all $1\leq s\leq k-1$\footnote{If $k=1$, then there is no requirement for $[D_s]^2\leq -2$.} and $[D_k]^2\geq 0$;
        \item or $k=2$ and $[D_1]^2=-1,[D_2]^2\geq 0$. 
    \end{itemize}
\end{lemma}
\begin{proof}
    If we complete $D$ into the log Calabi-Yau divisor $D':=D\cup C_{E_{\text{min}}}$, then by the partially minimal condition, we know the minimal reduction procedure for $D'$ does not involve non-toric blowdowns. Moreover, all toric blowdowns are performed at the components which come from the total transform of an intersection point between two components in a minimal model. In other words, there will be a self-intersection sequence $(a_1,\cdots,a_m)$ of some minimal model (so $2\leq m\leq 4$) and a toric blowup sequence $(t_1,\cdots,t_n)$ in the sense of Definition \ref{def:toricblowupseq} such that the self-intersection sequence of $D'$ is $(a_1+t_n,\cdots,a_m+t_1,t_2,\cdots,t_{n-1})$. The partially minimal condition requires that there is exactly one $2\leq i\leq n-1$ such that $t_i=-1$ corresponding to $C_{E_{\text{min}}}$ and all the other $t_j$'s are $\leq-2$ for $j\neq i$ and $2\leq j\leq n-1$. By Fact \ref{fact:blowupseq}, either $t_1$ or $t_n$ must be equal to $-1$. Without loss of generality, assume $t_1=-1$, and split into cases according to the value of $a_m+t_1$.
    
     When $a_m+t_1\geq 0$, we obtain the first bullet by taking $D_k,D_{k-1},\dots, D_1$ to be the components of $D$ corresponding to the self-intersection sequence $(a_m+t_1,t_2,\dots,t_{i-1})$.

     When $a_m+t_1=-1$, the second partially minimal condition forces $i=2$. Fact \ref{fact:blowupseq} then implies $(t_1,\dots,t_n)=(-1,-1,-1)$. If the first bullet still fails, then the minimal model for $D$ must be of type (B2) or (B3), with $(a_1,\dots,a_m)=(0,2,0)$ or $(0,0,0,0)$; in either case, the second bullet follows.
     
      Suppose now that $a_m+t_1\leq -2$. If the minimal model is not of type (B1) or (C1), then one checks case-by-case (using $a_m\leq -1$) that the first bullet holds.

      It remains to consider types (B1) and (C1). We claim that, in these cases, the first condition in the partially minimal assumption forces $(t_1,\cdots,t_n)=(-1,-1,-1)$. Let $e_1$ (resp. $C_{e_1}$) denote the exceptional class (resp. exceptional divisor) introduced by the first toric blowup from the minimal model to $D'$. If $n\geq 4$, then the proper transform $\tilde{C}_{e_1}\subset D'$ of $C_{e_1}$ has self-intersection number $<-1$ and hence must be a component of $D$. In this situation, the exceptional class $f_1-e_1$ (resp. $f-e_1$) has non-negative intersection with every component of $D'$ and satisfies $(f_1-e_1)\cdot E_{\text{min}}=0$, contradicting partial minimality. Therefore $n\leq 3$, and the claim follows.

      When $(t_1,\cdots,t_n)=(-1,-1,-1)$, we may simply take $l=1$ and choose $D_1$ to be the component with $[D_1]^2\geq 0$, which yields the first bullet.
\end{proof}

\begin{lemma}\label{lem:admsubchain}
    A quasi-minimal pair of first kind which is also partially minimal contains an admissible subchain.
\end{lemma}

\begin{proof}
    Let us choose the subchain $D_1\cup\cdots\cup D_k$ satisfying the requirements in Lemma \ref{lem:partialminimalgoodchain}. Let $a_i:=-[D_i]^2$. We can verify the admissible condition by discussing the following two cases.
    \begin{itemize}
        \item When $a_j\geq 2$ for all $1\leq j\leq k-1$ and $a_k\leq 0$, then we have $c_2>c_1>0$. Assume $c_{l}>c_{l-1}>0$ holds, then $c_{l+1}=a_lc_l-c_{l-1}\geq 2c_l-c_{l-1}>c_l>0$. By inductions, we see that the associated sequence $c_i$'s will be a positive increasing sequence. Also, since $a_k\leq 0$, $c_{k-1}-c_ka_k>0$ holds.
        \item When $k=2$ and $a_1=1,a_2\leq 0$, the associated sequence $(c_1,c_2)=(1,1)$ which is positive and $c_1-c_2a_2\geq c_1>0$.
    \end{itemize}
\end{proof}

We now deal with the additional cases arising from blowups of pairs with $b_2\le 2$, where the affine rulings can be chosen to be smooth.
\begin{proof}[Proof of Lemma \ref{lem:threecases}]
    When $D$ is comb-like (cases $(B1)'$ and $(C1)'$), we can still use the fiber class $f_2$ or $f$ to foliate the divisor complement as what we did in the irrational ruled cases. Note that $D$ can be viewed as a $(1,0)$-homologically affine-ruled divisor by choosing the $D$-good class $A$ to be $f$ or $f_2$. For all the other cases except for $(A3)'$, first notice that case $(A1)'$ can be subsumed by case $(A2)'$ since the definition of symplectic affine-ruledness allows us to remove some symplectic submanifolds. Then we can choose the class $A:=[D_1]$ in each case to make the pair $(1,[D_1]^2)$-homologically affine-ruled. The resolution class $\tilde{A}$ would just be $[\tilde{D}_1]$, the class of the proper transform of $D_1$. Since $[\tilde{D}_1]$ is represented by the embedded $J$-holomorphic sphere as the component in the total transform $\tilde{D}$ for any $J\in\mathcal{J}(\tilde{D})$, it must be $\tilde{D}$-good. Thus, by Lemma \ref{lem:homologicallyaffineruled} and Corollary \ref{cor:blowupaffineruled}, all blowups of these cases are symplectic affine-ruled.

    Case $(A3)'$ requires additional care since we can not find a class $A$ which makes it homologically affine-ruled (see also Example \ref{example:A3} below). Assume $(X,\omega,D)$ is obtained from some non-toric, half-toric or exterior blowups of a sphere in class $2h$ in $\CC\PP^2$. We can further topologically perform a half-toric blowup at a point $x\in D$ (with exceptional class $e_1$) and three toric blowups (with exceptional class $e_1,e_2,e_3$) to get the configuration $D':=\tilde{D}\cup C_{e_4}\cup C_{e_3-e_4}\cup C_{e_2-e_3}\cup C_{e_1-e_2}$ where $\tilde{D}$ is the proper transform of $D$. Then we take the class $A:=2h-e_1-e_2-e_3-e_4$ which can be easily checked to be $D'$-good. Then by Theorem \ref{thm:MO15} there is some $J'\in\mathcal{J}_{\text{emb}}(D',A)$ which corresponds to a $J\in\mathcal{J}(D)$. We can identify the moduli space of $J'$-holomorphic subvarieties in class $A$ with the sphere $C_{e_4}$ by the same argument as Lemma \ref{lem:tildeAmoduli}. Any smooth representative of $A$ will descend to a smooth rational $J$-holomorphic curve $u$ in class $2h$ with $\text{ord}(u,D;x)=4$ by Proposition \ref{prop:modulirelation}. After throwing out all the descendants of the reducible representatives of $A$, we see that the remaining part is foliated by $J$-holomorphic punctured spheres. Therefore, we can verify the symplectic affine-ruledness for $(X,\omega,D)$ by the same argument as Lemma \ref{lem:homologicallyaffineruled}.
\end{proof}

\begin{figure}[ht]
		\centering\includegraphics*[height=4cm, width=12cm]{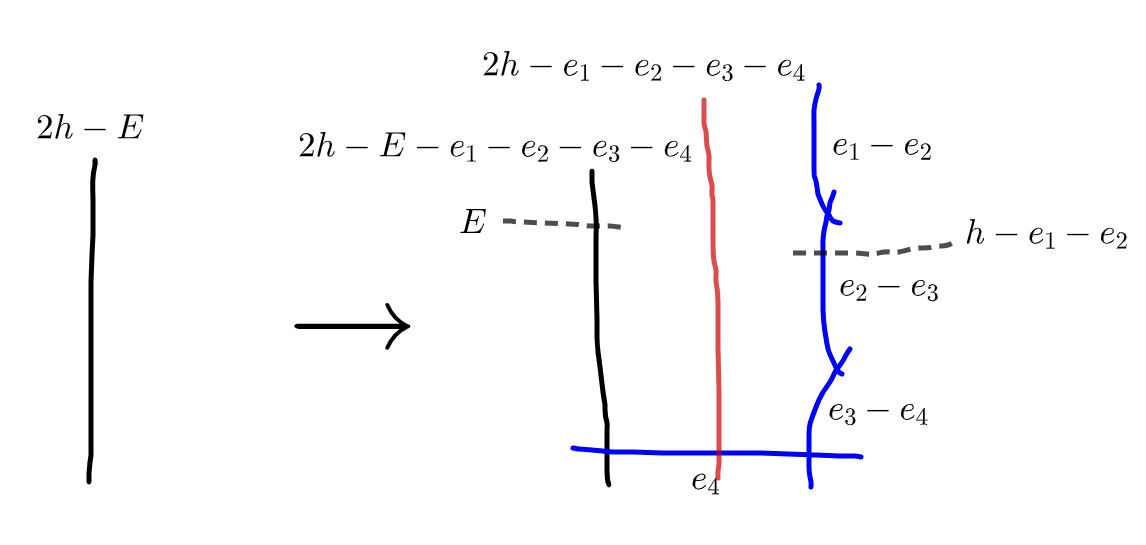}
 
		\caption{The blowup for case $(A3)'$ in the proof of Lemma \ref{lem:threecases}. \label{fig:2h}}
	\end{figure}

\begin{example}\label{example:A3}
   Consider a single embedded sphere of class $2h-E$ in $\CC\PP^2\#\overline{\CC\PP}^2$, which is the non-toric blowup of case $(A3)'$. Figure \ref{fig:2h} shows the configuration after $4$ blowups. The class $A:=2h-e_1-e_2-e_3-e_4$ can be used to foliate the complement. Note that we can visualize two reducible representatives of $A$ in the figures: $A=(2h-E-e_1-e_2-e_3-e_4)+E$ and $A=2(h-e_1-e_2)+(e_1-e_2)+2(e_2-e_3)+(e_3-e_4)$.
\end{example}

\subsection{Proof of symplectic affine-ruledness}\label{secton:proofofaffineruled}
With all the necessary lemmas in place, we are now ready to prove the main result of this section.

\begin{theorem}\label{thm:rationalmain}
    Any connected pair $(X,\omega,D)$ where $X$ is a rational manifold with $[\omega]\cdot([D]+K_{\omega})<0$ is symplectic affine-ruled.
\end{theorem}
 
\begin{proof}
By Lemma \ref{lem:quasiminred}, any such pair is from the blowups of a pair with $b_2\leq 2$ or a quasi-minimal pair. In the former case, we can just use Lemma \ref{lem:threecases} to see the symplectic affine-ruledness. In the latter case, the quasi-minimal pair is either of first kind or second kind. For the case of second kind, Corollary \ref{cor:quasiminmalsecondkindreduction} will lead us to the case of $b_2\leq 2$ and the symplectic affine-ruledness is again guaranteed by Lemma \ref{lem:threecases}. For the case of first kind, by Lemma \ref{lem:partialminimalred}, we can further reduce them into partially minimal pairs or $b_2\leq 2$ cases by blowdowns. If $b_2\leq 2$ we apply Lemma \ref{lem:threecases} again to see the symplectic affine-ruledness. Otherwise, note that Lemma \ref{lem:admsubchain} says that its partially minimal reduction must contain an admissible subchain and Lemma \ref{lem:chainimplyhaffruled} combined with Lemma \ref{lem:resolutionDgood} shows that the existence of admissible subchain will imply the homologically affine-ruledness with resolution class being $\tilde{D}$-good. Next, we can apply Lemma \ref{lem:homologicallyaffineruled} to see the symplectic affine-ruledness of all partially minimal pairs of first kind and then apply Corollary \ref{cor:blowupaffineruled} to include all quasi-minimal pairs of first kind. Finally we use Corollary \ref{cor:blowupaffineruled} again to include all connected pair $(X,\omega,D)$ where $X$ is a rational manifold with $[\omega]\cdot([D]+K_{\omega})<0$. See the diagram below for a visualization of the logic behind our argument.
\end{proof}


\adjustbox{scale=0.7,center}
    {
		\begin{tikzpicture}
		\node (a) at (0,1) [][]{$[\omega]\cdot(K_\omega+[D])<0$};
		\node (b) at (-6,-1) [][]{quasi-minimal of first kind};
        \node (c) at (0,-4) [][]{$b_2\leq 2$};
        \node (d) at (6,-1) [][]{quasi-minimal of second kind};
        \node (e) at (-6,-2.5) [][]{partially minimal};
        \node (f) at (-6,-4) [][]{admissible subchian};
        \node (g) at (-9,-5.5) [][]{homologically affine-ruled};
        \node (h) at (-3,-5.5) [][]{resolution class $\tilde{A}$ is $\tilde{D}$-good};
        \node (i) at (-6,-7) [][]{symplectic affine-ruled};
        \node (j) at (0,-8.5) [][]{Theorem \ref{thm:rationalmain}};
        \draw[->] (a) -- (b) node[midway,above left] {Lemma \ref{lem:quasiminred}};
        \draw[->] (a) -- (c) node[midway,right] {Lemma \ref{lem:quasiminred}};
        \draw[->] (a) -- (d) node[midway,above right] {Lemma \ref{lem:quasiminred}};
        \draw[->] (b) -- (e) node[midway,left] {Lemma \ref{lem:partialminimalred}};
        \draw[->] (e) -- (f) node[midway,left] {Lemma \ref{lem:admsubchain}};
         \draw[->] (f) -- (g) node[midway,above left] {Lemma \ref{lem:chainimplyhaffruled}};
          \draw[->] (f) -- (h) node[midway,above right] {Lemma \ref{lem:resolutionDgood}};
          \draw[-] (g) -- (-6,-6) node[midway,above] {};
          \draw[-] (h) -- (-6,-6) node[midway,above] {};
          \draw[->] (-6,-6) -- (i) node[midway,left] {Lemma \ref{lem:homologicallyaffineruled}};
           \draw[->] (i) -- (j) node[midway,above] {Corollary \ref{cor:blowupaffineruled}};
			 \draw[->] (b) -- (c) node[midway,above right] {Lemma \ref{lem:partialminimalred}};
              \draw[->] (d) -- (c) node[midway,below right] {Corollary \ref{cor:quasiminmalsecondkindreduction}};
              \draw[->] (c) -- (j) node[midway,right] {Lemma \ref{lem:threecases}};
		\end{tikzpicture}
	}	

\begin{example}
    We provide an example that illustrates the reduction procedure. The first configuration in Figure \ref{fig:example} is a divisor $D$ in $\CC\PP^2\#13\overline{\CC\PP}^2$ where \[[D]=3H-E_1-E_2-E_3-E_4-E_5-E_6-2E_7-E_8-E_{12}.\]Thus, if we assume $\omega(E_i)\gg \omega(E_{i+1})$ for all $i$'s, $K_{\omega}+[D]=-E_7+E_9+E_{10}+E_{11}+E_{13}$ will have negative pairing with $\omega$. The second configuration in Figure \ref{fig:example} denotes its quasi-minimal reduction, obtained by exterior blowdown $E_{13}$, toric blowdown $E_{12}$, half-toric blowdown $E_{11},E_{10},E_9$ and non-toric blowdown $E_8$. This quasi-minimal pair is of first kind. By further toric blowdown $E_6,E_5$ and non-toric blowdown $E_4$, we will get the third configuration in Figure \ref{fig:example}, which is the partially minimal reduction. Then we have an admissible subchain $([D_1],[D_2],[D_3])=(E_3-E_7,E_2-E_3,2H-E_1-E_2)$ with negative self-intersection sequence $(a_1,a_2,a_3)=(2,2,-2)$. The associated sequence $(c_1,c_2,c_3)$ is then $(1,2,3)$. So, the class for the unicuspidal rational curve is \begin{align*}
        A&=c_1[D_1]+c_2[D_2]+c_3[D_3]=(E_3-E_7)+2(E_2-E_3)+3(2H-E_1-E_2)\\
        &=6H-3E_1-E_2-E_3-E_7.
    \end{align*}Note that $A\cdot [D_3]=8$ and $A\cdot (H-E_1)=3$. This implies that the cusp is of type $(3,8)$.
\end{example}

\begin{figure}[ht]
		\centering\includegraphics*[height=7cm, width=15cm]{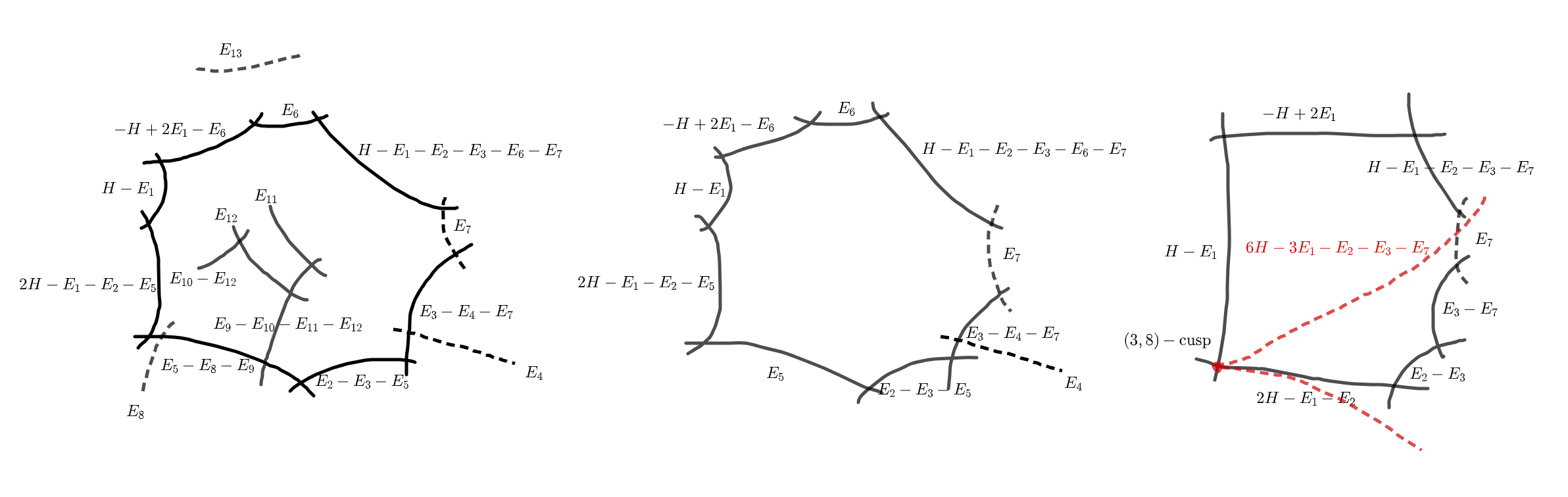}
 
		\caption{The reduction procedure for finding a $(3,8)$-unicuspidal rational curve. Dashed curves are not included as part of the symplectic divisors in each graph. \label{fig:example}}
	\end{figure}

\section{Symplectic deformation aspect and relation with K\"ahler pairs}\label{section:deformation}
In this section, we extend the Li-Mak's deformation result (\cite{LiMakLCY,LiMakICCM}) for symplectic log Calabi-Yau pairs to the divisors satisfying $[\omega]\cdot (K_{\omega}+[D])<0$. Let us start by introducing the notions for deformation equivalence between symplectic divisors.

\begin{definition} 
A {\bf symplectic homotopy} (resp. {\bf isotopy}) of the pair $(X,\omega,D)$ is a smooth one-parameter family of symplectic divisors $(X,\omega_t,D_t)$ with $(X,\omega_0,D_0)=(X,\omega,D)$ (resp. such that in addition $\omega_t=\omega$ for all $t$). A {\bf $D$-symplectic homotopy} (resp. {\bf isotopy}) of the pair $(X,\omega,D)$ is a smooth one-parameter family of symplectic forms $\omega_t$ with $\omega_0=\omega$, keeping components in $D$ as symplectic submanifolds throughout (resp. such that in addition $[\omega_t]$ is constant). Two pairs $(X,\omega,D)$ and $(X',\omega',D')$ are called {\bf symplectic deformation equivalent} if there is a diffeomorphism $f:X\rightarrow X'$ such that there is a symplectic homotopy interpolating between $(X,\omega,D)$ and $(X,f^*\omega',f^{-1}(D'))$. The symplectic deformation equivalence is called {\bf strict} if the symplectic homotopy is a symplectic isotopy.
\end{definition}

As observed in \cite[Lemma 2.2]{LiMakLCY}, it is a consequence of the smooth isotopy extension theorem that any symplectic homotopy can actually be viewed as a $D$-symplectic homotopy up to a one-parameter family of symplectomorphisms. By the main result in \cite{LiMakLCY}, there is a Torelli theorem for symplectic log Calabi-Yau divisors which states that their symplectic deformation classes are fully determined by the homology classes of the components. In \cite[Proposition 2.10]{Enumerate}, the Torelli theorem was further refined to show that two symplectic log Calabi-Yau pairs are symplectomorphic if and only if there is an integral isometry between their cohomology lattice which preserves their symplectic classes and components in the divisors. \cite{LiMakICCM} later explored the connections with divisors in the K\"ahler category. We say that a symplectic divisor $(X,\omega,D)$ is a {\bf K\"ahler pair} if there exists an integrable complex structure $J$ compatible with $\omega$ making all components $J$-holomorphic. \cite{LiMakICCM} proved that any symplectic log Calabi-Yau divisor is symplectic deformation equivalent to a K\"ahler pair. The proof relies on the fact that blowup\footnote{The `blowup' in this section refers to any type (exterior, toric, non-toric, half-toric).} and blowdown operations can be performed within the K\"ahler category, as well as several key results in the proof of the Torelli theorem. In this section, we will generalize the deformation K\"ahlerness results in \cite{LiMakICCM} from log Calabi-Yau divisors to divisors with $[\omega]\cdot(K_{\omega}+[D])<0$.

\begin{theorem}\label{thm:deformation}
    Let $D\subseteq (X,\omega)$ be a symplectic divisor with $[\omega]\cdot(K_{\omega}+[D])<0$. If $X$ is a rational manifold, further assume that $D$ is connected. Then $(X,\omega,D)$ is symplectic deformation equivalent to a K\"ahler pair with $\overline{\kappa}(X\setminus D)=-\infty$.
\end{theorem}
 We start with a preparation lemma which is an adaptation of \cite[Lemma 5.5 B]{MP94} where the statement is originally formulated for tamed symplectic forms.

\begin{lemma}\label{lem:flatKahler}
    Let $\omega$ be a K\"ahler form on $B(1)$ with respect to the standard complex structure. Denote by $\omega_{\text{std}}$ the standard K\"ahler form. Then there is a smooth family of K\"ahler forms $\omega_t$ such that $\omega_0=\omega$, and $\omega_1$ agrees with $\omega$ near the boundary of the ball and is a constant multiple of $\omega_{\text{std}}$ near $0$ (the associated metric is flat).
\end{lemma}
\begin{proof}
   By the local $\partial\overline{\partial}$ lemma, after rescaling the radius of the ball, we may assume $\omega=\partial\overline{\partial}\phi$ for some $\phi\in C^{\infty}(B(1),\CC)$. Choose two parameters $K$ and $\varepsilon$ such that $K$ is large and $\varepsilon$ is close to $0$. Take a strictly increasing function $f$ with $f(r)=\frac{K}{\varepsilon}r$ for $r\leq \frac{\varepsilon}{2K}$ and $f(r)=r$ for $r$ close to $1$ and the diffeomorphism $h:B(1)\rightarrow B(1)$ defined by $h(r,p)=(f(r),p)$ in the spherical coordinate $(r,p)\in \RR_{\geq0}\times S^3$. Then there will be another K\"ahler form $\tau_K:=h^*(\varepsilon^2\omega_{\text{std}})$ which is equal to $K^2\omega_0$ in $B(\frac{\varepsilon}{2K})$ and $\varepsilon^2\omega_{\text{std}}$ near the boundary. We also fix a bump function $\rho$ on $\CC^2$ supported in $B(1)$ which is equal to $1$ near the origin and define $\rho_K(z):=\rho(\frac{2K}{\varepsilon}z)$. Consider the deformation of $(1,1)$-forms
   \[\omega_t:=\omega+t(\tau_K-\varepsilon^2\omega_{\text{std}}-\partial\overline{\partial}(\rho_K\phi)).\]
   Note that outside the ball $B(\frac{\varepsilon}{2K})$, $\omega_t=\omega+t(\tau_K-\varepsilon^2\omega_{\text{std}})$. When $\varepsilon$ is sufficiently small such that $\omega-\varepsilon^2\omega_{\text{std}}$ is positive on $B(1)$, it follows that $\omega_t$ will also be positive outside the ball $B(\frac{\varepsilon}{2K})$. Inside the ball $B(\frac{\varepsilon}{2K})$, we have 
     \[\tau_K-\varepsilon^2\omega_{\text{std}}-\partial\overline{\partial}(\rho_K\phi)=(K^2-\varepsilon^2)\omega_{\text{std}}-\rho_K\omega-\frac{2K}{\varepsilon}\alpha,\]
     where $\alpha=(\partial\overline{\partial}\rho)\wedge \phi-\overline{\partial}\rho\wedge\partial\phi+\partial\rho\wedge\overline{\partial}\phi$ is a two-form independent of $K$ and $\varepsilon$. From the quadratic growth in $K$ we can see that when $K$ is sufficiently large, $\omega_t$ is positive inside $B(\frac{\varepsilon}{2K})$. Therefore $\omega_t$ is a deformation of K\"ahler forms and it is easy to see that $\omega_1$ satisfies the requirement.
\end{proof}

We address the following result concerning the minimal models of log Calabi-Yau pairs which is essentially proved in \cite[Proposition 3.6]{LiMakLCY}.

\begin{proposition}\label{prop:limakminKahler}
    Let $(X,\omega)$ be a symplectic rational manifold with $b_2(X)\leq 2$ and $D\subseteq (X,\omega)$ be a symplectic log Calabi-Yau divisor. Then $(X,\omega,D)$ is a K\"ahler pair.
\end{proposition}
\begin{proof}
     Observe that the homological configuration of any pair listed in Section \ref{section:LCY} can be realized by a K\"ahler pair $(X,\omega',D')$ with $[\omega']=[\omega]$: this is clear for $\CC\PP^2$ since the K\"ahler cone of the standard complex structure on $\CC\PP^2$ contains $[\omega]$; for $S^2\times S^2$ or $\CC\PP^2\#\overline{\CC\PP^2}$, we choose the complex structure coming from an appropriate Hirzebruch surface so that there is a section whose self-intersection equals the minimum self-intersection among the components of $D$, whose K\"ahler cone also contains $[\omega]$ by Nakai-Moishezon criterion. By Lemma \ref{lem:flatKahler}, we can modify the K\"ahler form in a small neighborhood of intersection points to guarantee the $\omega'$-orthogonality within the same K\"ahler class. Since cohomologous symplectic forms on rational ruled manifolds are diffeomorphic (\cite{spaceofsymp}), we may assume $\omega'=\omega$. Then $(X,\omega,D)$ will be symplectic isotopic to $(X,\omega,D')$ by \cite[Proposition 3.6]{LiMakLCY}. They are furthermore symplectomorphic by \cite[Proposition 2.10]{Enumerate}. Since $(X,\omega,D')$ is a K\"ahler pair, so is $(X,\omega,D)$. 
\end{proof}

Now, we extend the above result to divisors with $[\omega]\cdot(K_{\omega}+[D])<0$.

\begin{proposition}\label{prop:defminimal}
    Let $D\subseteq (X,\omega)$ be a symplectic divisor in a rational or ruled manifold $X$ with $b_2(X)\leq 2$. If $X$ is rational manifold, assume that $D$ is connected and $[\omega]\cdot(K_{\omega}+[D])<0$; if $X$ is irrational ruled, assume $D$ is comb-like. Then $(X,\omega,D)$ is a K\"ahler pair.
\end{proposition}
\begin{proof}
   First, assume $X$ is rational and the configuration is not comb-like (case $(B1)'$ or $(C1)'$ in Section \ref{secton:proofofaffineruled}). Note that the cases $(A2)'/(A2)/(A1)$ can be obtained from $(A1)'/(A2)'/(A3)'$ by adding one sphere in the divisor configuration, in the class $h$; the cases $(B2)/(B3)/(C2)/(C3)$ can be obtained from $(B2)'/(B3)'/(C2)'/(C3)'$ by adding one sphere in the divisor configuration, in the fiber class: the class $h$ and the fiber class are known to have non-trivial SW invariants and satisfy the homological conditions in the $D$-good definition so that we can apply McDuff-Opshtein's criterion (Theorem \ref{thm:MO15}) to obtain their embedded $J$-holomorphic representatives $C$. After perturbing $C$ to satisfy the $\omega$-orthogonal condition, we can complete $D$ into a log Calabi-Yau divisor $D\cup C$ in $(X,\omega)$. Then $(X,\omega,D)$ must be a K\"ahler pair by the previous Proposition \ref{prop:limakminKahler}.
   
   Now, assume the configuration is comb-like ($X$ is irrational ruled or case $(B1)'$ or $(C1)'$ when $X$ is rational). By the result of Coffey \cite{Coffey} (see also \cite[Section 9]{HindIvrii}), the group $Symp_h(X,\omega)$ of homologically trivial symplectomorphisms acts transitively on the set of symplectic surfaces in a section class. Thus, by using the model of the projectivization of a rank $2$ holomorphic vector bundle over the Riemann surface (see Section \ref{section:appendix}), we can construct a K\"ahler pair $(X,\omega,D')$ (again, by the fact that cohomologous symplectic forms are diffeomorphic) whose homological configuration, the component in the section class and all the intersection points are the same as $(X,\omega,D)$. To find the symplectic isotopy between $D$ and $D'$, it suffices to isotope the spherical components in the fiber class and then apply \cite[Proposition 2.10]{Enumerate}. This can be achieved by the standard pseudo-holomorphic curve argument by choosing a regular path of almost complex structures $\{J_t\}$, where $J_0,J_1$ make the spherical component in $D,D'$ holomorphic, and analyzing the $J_t$-holomorphic curves in the fiber class passing through a fixed point. We only have to rule out the possible degenerations. When $X$ is irrational ruled, this is clear from Theorem \ref{thm:fiberclass}. When $X=\CC\PP^2\#\overline{\CC\PP^2}$, this can be seen from \cite[Lemma 1.6, Theorem 1.7]{LLWnef} where it is proved that the fiber class is $\omega$-nef and thus has an embedded $J$-holomorphic representative for any $\omega$-tame $J$.  When $X=S^2\times S^2$, any degeneration of a sphere in class $f_2$ would produce a sphere in class $f_2-af_1$ with $a>0$.
   Thus it suffices to have $\omega(f_2)\le \omega(f_1)$ to rule out degenerations.
   In case $(B1)'$, since $[D_1]=f_1+kf_2$, this inequality holds automatically when $k<0$.
   When $k\ge 0$, the assumption $[\omega]\cdot(K_{\omega}+[D])<0$ implies (in case $(B1)'$) that $\omega((k+n-3)f_2)<\omega(f_1)$.
   Unless $k+n\le 3$ (in which case $D$ can be completed to a log Calabi--Yau divisor and the K\"ahlerness follows from Proposition \ref{prop:limakminKahler}), we still obtain $\omega(f_2)\le \omega(f_1)$, and hence degenerations are excluded.
\end{proof}

Next, we need several lemmas for the symplectic deformation equivalence under blowups.

\begin{lemma}\label{lem:shrinkingsizeblowup}
   Let $(X,\omega,D)$ be a pair with symplectic embedding $I:B(\delta)\rightarrow X$ relative to $D$. Then the symplectic blowups of $(X,\omega,D)$ by using embeddings $I|_{B(\delta')}$ are symplectic deformation equivalent for any $\delta'\leq \delta$.
\end{lemma}

\begin{proof}
We can choose compatible almost complex structure $J$ such that $I$ is also a holomorphic embedding. Then we can identify the symplectic blowup using $I$ with the model of almost complex blowup at $I(0)$.  By \cite[Theorem 7.1.23]{MSintro}, under this identification, there will be a deformation of symplectic forms on the blowup manifold. This deformation will keep the total transform of $D$ as symplectic divisors by its construction.
\end{proof}

\begin{lemma}\label{lem:strongD}
    Let $(X,\omega_t,D)$ be a $D$-symplectic homotopy and $I_i:B(\delta_i)\rightarrow X$ be $\omega_i$-symplectic embeddings relative to $D$ for $i=0,1$ with $I_0(0)=I_1(0)$. Then there exists another $D$-symplectic homotopy $(X,\omega_t',D)$ such that $\omega_i'=\omega_i$ for $i=0,1$ and a smooth family $I_t:B(\varepsilon)\rightarrow X$ of $\omega_t'$-symplectic embeddings for some $\varepsilon\ll \min\{\delta_0,\delta_1\}$.
\end{lemma}

\begin{proof}
    This is essentially proved in \cite[Lemma 2.9]{LiMakLCY}. First, by the one-parameter family version of Moser lemma, there is a family of Darboux balls centered at $I_0(0)$. Then choose a two-parameter family of perturbations $D_{s,t}$ of the divisor $D$ such that $D_{s,t}$ is an $\omega_t$-symplectic divisors, $D_{s,0}=D_{s,1}=D$ and the preimages of $D_{s,0},D_{s,1},D_{1,t}$ in the Darboux balls (after shrinking the size) are coordinate planes. By \cite[Lemma 2.2]{LiMakLCY}, there will be a two-parameter family of diffeomorphisms $\Delta_{s,t}:X\rightarrow X$ such that $\Delta_{s,t}(D_{0,t})=D_{s,t}$ and $\Delta_{s,1}=Id_X$. If we take $\omega_t':=\Delta_{1,t}^*\omega_t$, then $\omega_0'=\omega_0 ,\omega_1'=\omega_1$ and the composition of the Darboux embeddings with $\Delta_{1,t}$ gives the desired embeddings $I_t$. 
\end{proof}

\begin{lemma}\label{lem:blowupdeformation}
Let $(X,\omega_t,D)$ be a $D$-symplectic homotopy and $I_t:B(\varepsilon)\rightarrow X$ be a smooth family of $\omega_t$-symplectic embeddings relative to $D$ with constant $I_t(0)$. Then the blowup pairs of $(X,\omega_1,D)$ and $(X,\omega_2,D)$ by using $I_1$ and $I_2$ respectively are symplectic deformation equivalent.
\end{lemma}

\begin{proof}
    By the smooth isotopy extension theorem (\cite[Theorem 3.1, Chapter 8]{Hirsch}), there is a family of diffeomorphisms $f_t:X\rightarrow X$ such that for all $t$, $f_t\circ I_t$ is a constant map from $B(\varepsilon)$ to $X$ and is a $(f_t^{-1})^*\omega_t$-symplectic embedding relative to $f_t(D)$. Then, the symplectic blowup construction reviewed in Section \ref{sec:divisoroperation} will yield 
    \begin{itemize}
        \item a fixed underlying smooth manifold $\tilde{X}$ by removing the ball $f_t\circ I_t(B(\varepsilon))$ and collapsing its boundary by Hopf fibration;
        \item a smooth family of symplectic forms $\tilde{\omega}_t$ on $\tilde{X}$ corresponding to $(f_t^{-1})^*\omega_t$;
        \item a smooth family of $(f_t^{-1})^*\omega_t$-symplectic divisors $\tilde{D}_t\subseteq \tilde{X}$ coming from the total transforms of $f_t(D)$.
    \end{itemize} 
    The above data provide the desired symplectic deformation equivalence for the blowup pairs.
\end{proof}

\begin{proposition}\label{prop:smallKahler}
Assume that the symplectic divisor $(X,\omega,D)$ is symplectic deformation equivalent to a K\"ahler pair. Let $I:B(\delta)\rightarrow X$ be a symplectic embedding relative to $D$. Then the symplectic blowup by using $I$ is also symplectic deformation equivalent to a K\"ahler pair.
\end{proposition}
\begin{proof}
By the assumption and Lemma \ref{lem:flatKahler}, we may suppose there is a $D$-symplectic homotopy of $(X,\omega,D)$ towards a K\"ahler pair $(X,\omega',D)$ such that for some $\varepsilon \ll \delta$, there is an $\omega'$-symplectic embedding $I':B(\varepsilon)\rightarrow X$ relative to $D$  with $I'(0)=I(0)$ which is also holomorphic. Namely, $I'$ is a normalization in the sense of \cite[Section 7.1]{MSintro}. By Lemma \ref{lem:strongD}, we can further assume $I|_{B(\varepsilon')}$ and $I'|_{B(\varepsilon')}$ can be joint by a smooth family of symplectic embeddings for some $\varepsilon'\ll\varepsilon$ by possibly modifying the deformation. Then, Lemma \ref{lem:blowupdeformation} and Lemma \ref{lem:shrinkingsizeblowup} would imply that the blowup of $(X,\omega,D)$ by using $I$ is symplectic deformation equivalent to the blowup of $(X,\omega',D)$ by using $I'|_{B(\varepsilon')}$. Finally, note that for some $\varepsilon''<\varepsilon'$, the symplectic blowup of $(X,\omega',D)$ using $I'|_{B(\varepsilon'')}$ will be a K\"ahler pair by Corollary \ref{cor:divisorblowupcompare}. Putting all things together and applying Lemma \ref{lem:shrinkingsizeblowup} again, we see that the blowup of $(X,\omega,D)$ is symplectic deformation equivalent to a K\"ahler pair.
\end{proof}

\begin{proof}[Proof of Theorem \ref{thm:deformation}]
When $X$ is a rational manifold, the proof in Section \ref{secton:proofofaffineruled} shows that $(X,\omega,D)$ can be obtained from the blowups of either a divisor in $\CC\PP^2, S^2\times S^2$ or $\CC\PP^2\#\overline{\CC\PP}^2$ or a quasi-minimal pair of first kind. For the former case, the pair must be K\"ahler by Proposition \ref{prop:defminimal}. For the latter case, note that the pair can be completed into a symplectic log Calabi-Yau pair by adding an exceptional sphere with minimal symplectic area by Proposition \ref{prop:-K-D=Emin}, which must be deformation equivalent to a K\"ahler pair by \cite[Theorem 1.3]{LiMakICCM}. Therefore, we can apply Proposition \ref{prop:smallKahler} to show that $(X,\omega,D)$ is also deformation equivalent to a K\"ahler pair.

When $X$ is an irrational ruled manifold, by Corollary \ref{cor:no3-Ej}, we know $D$ can be completed into a divisor which can be obtained from the blowups of a comb-like configuration in $S^2\times \Sigma_g$. The conclusion of being a K\"ahler pair again follows from Propositions \ref{prop:defminimal} and \ref{prop:smallKahler}.


Finally, suppose that we have a K\"ahler pair $(X,\omega_1,D)$ with respect to the complex structure $J$ after the symplectic deformation. To see $\overline{\kappa}(X\setminus D)=-\infty$ for this K\"ahler pair, remember that if we forget about the K\"ahler form $\omega_1$, the holomorphic pair $(X,J,D)$ is constructed from a sequence of complex blowups of another K\"ahler pair $(X',\omega',D')$, where either $b_2(X')\leq 2$ or $(X',\omega',D')$ is $D'$-symplectic homotopy to a quasi-minimal pair. In the former case, one naturally has $[\omega']\cdot(K_{\omega'}+[D'])<0$ since blowdown will preserve this condition, as mentioned in Section \ref{sec:divisoroperation}; In the latter case, $[\omega']\cdot(K_{\omega'}+[D'])=\omega'(-E_{\text{min}})$ which also must be negative since $E_{\text{min}}$ has non-trivial SW invariant. Therefore, $\overline{\kappa}(X'\setminus D')=-\infty$ since the condition $[\omega']\cdot(K_{\omega'}+[D'])<0$ implies that the linear system $|n(K_{\omega'}+[D'])|$ must be empty for any $n\geq 1$. By the birational invariance of log Kodaira dimension, $\overline{\kappa}(X\setminus D)=-\infty$ also holds true.
\end{proof}

\section{Appendix: proof of Proposition \ref{prop:kahlermain}}\label{section:appendix}

In this appendix, we provide the proof of Proposition \ref{prop:kahlermain} for irrational ruled manifolds, as outlined in the introduction. Only the non-minimal case needs to be proved, since the minimal case is well-known (See the references in \cite{spaceofsymp}).

\begin{proposition}\label{prop:c1positiveKahler}
    Let $(X,\omega)$ be a symplectic irrational ruled manifold with $[\omega]\cdot K_{\omega}<0$. Then $\omega$ is a K\"ahler form.
\end{proposition}

Let us introduce some notions first. Suppose $\Sigma_g$ is a Riemann surface of genus $g\geq 1$. Let $X_{g,n}=(\Sigma_g\times S^2)\#n\overline{\CC\PP}^2$, $B,F,E_1,\cdots,E_n\in H_2(X_{g,n};\ZZ)$ be the standard basis such that $B^2=0,F^2=0,B\cdot F=1$, and $\tilde{X}_{g,0}=S^2 \tilde{\times} \Sigma_g$ with $B_1,F\in H_2(\tilde{X}_{g,0};\ZZ)$ such that $B_1^2=1,F^2=0,B_1\cdot F=1$ as Section \ref{section:ruled}. Consider the following regions in \[\RR_{>0}^{n+1}=\{(\dd_B,\dd_1,\cdots,\dd_n)\,|\,\dd_B,\dd_i>0\}\] which characterize the moduli of symplectic forms on irrational ruled manifolds:
\[\mathcal{P}^n\supseteq \mathcal{P}_1^n\supseteq \mathcal{P}_2^n\supseteq\cdots\text{ for $n\geq 0$}.\]

\begin{itemize}
    \item When $n=0$, they are intervals $\mathcal{P}^0=(0,\infty)$, $\mathcal{P}^0_g=(g,\infty)$ for $g\geq 1$.
    \item When $n=1$, $\mathcal{P}^1=\{2\dd_B-\dd_1^2>0,\dd_1<1\}$ and $\mathcal{P}_g^1=\{2-2g+2\dd_B-\dd_1>0,\dd_1<1\}$.
    \item When $n\geq 2$, $\mathcal{P}^n=\{2\dd_B-\sum_{i=1}^n\dd_i^2>0,\dd_1+\dd_2<1,\dd_1\geq \dd_2\geq \cdots\geq \dd_n\}$ and $\mathcal{P}^n_g=\{2-2g+2\dd_B-\sum_{i=1}^n\dd_i>0, \dd_1+\dd_2<1,\dd_1\geq \dd_2\geq \cdots\geq \dd_n\}$.
\end{itemize}

\begin{figure}[ht]
\centering
\begin{tikzpicture} [yscale=0.5]
    \begin{axis}[
  axis x line=center,
  axis y line=center,
  xtick={0,1,2,3},
  ytick={0,1},
  xlabel={$\dd_B$},
  ylabel={$\dd_1$},
  xlabel style={below right},
  ylabel style={above left},
  xmin=0,
  xmax=4.1,
  ymin=0,
  ymax=1.1]
    \addplot[name path=f,samples=300, domain=0:4.1,dashed] {min(sqrt(2*x),1)};
   \addplot[name path=g,samples=300, domain=0:4.1,dashed] {min(2*x,1)};
   \addplot[name path=h,samples=300, domain=1:4.1,dashed] {min(2*x-2,1)};
   \addplot[name path=k,samples=300, domain=2:4.1,dashed] {min(2*x-4,1)};
   \addplot[name path=l,samples=300, domain=3:4.1,dashed] {min(2*x-6,1)};
   
   \path[name path=axis] (axis cs:0,0) -- (axis cs:4.1,0);
   
    \addplot [
        thick,
        color=black,
        fill=black, 
        fill opacity=0.2
    ]
    fill between[
        of=f and axis,
    ];

\addplot [
        thick,
        color=black,
        fill=black, 
        fill opacity=0.2
    ]
    fill between[
        of=g and axis,
    ];

\addplot [
        thick,
        color=black,
        fill=black, 
        fill opacity=0.2
    ]
    fill between[
        of=h and axis,
    ];

\addplot [
        thick,
        color=black,
        fill=black, 
        fill opacity=0.2
    ]
    fill between[
        of=k and axis,
    ];

\addplot [
        thick,
        color=black,
        fill=black, 
        fill opacity=0.2
    ]
    fill between[
        of=l and axis,
    ];

    \node at (axis cs:  .2,  .5) {$\mathcal{P}^1$};
    \node at (axis cs:  .8,  .5) {$\mathcal{P}_1^1$};
    \node at (axis cs:  1.8,  .5) {$\mathcal{P}_2^1$};
    \node at (axis cs:  2.8,  .5) {$\mathcal{P}_3^1$};
    \node at (axis cs:  3.8,  .5) {$\mathcal{P}_4^1$};
    \end{axis}
\end{tikzpicture}  
\caption{Symplectic regions for $n=1$}\label{fig:regionforn=1}
\end{figure}

A vector in $\mathcal{P}^n$ is called a {\bf normalized reduced vector}, which encodes the symplectic area $\omega(B),\omega(E_1),\cdots,\omega(E_n)$ for some symplectic form $\omega$ with $\omega(F)=1$. Its subset $\mathcal{P}^n_g$ consists of vectors encoding the symplectic area for symplectic forms with $[\omega]\cdot K_\omega<0$ (see Figure \ref{fig:regionforn=1} for the case $n=1$). Holm-Kessler \cite{HolmKessler} study the Cremona transformation on irrational ruled manifolds and \cite[Theorem 2.14]{HolmKessler} can be reformulated as the following result which we will rely on.

\begin{theorem}[\cite{HolmKessler}]\label{thm:ruledreducedcone}
  There is a bijection between symplectic forms (resp. symplectic forms with $[\omega]\cdot K_\omega<0$) on $X_{g,n}$ modulo symplectomorphism and constant rescaling with $\mathcal{P}^n$ (resp. $\mathcal{P}^n_g$) by taking the symplectic area on the classes $B,E_1,\cdots,E_n$.   
\end{theorem}
   
Therefore, it suffices to realize any vector in $\mathcal{P}^n_g$ by some K\"ahler class to prove Proposition \ref{prop:c1positiveKahler}. The technique we will apply is the so-called $b_2^+=1$ $J$-compatible inflation theorem highlighted by Anjos-Li-Li-Pinsonnault in \cite[Lemma 1.2]{ALLP} and Buse-Li in \cite[Theorem 3.7]{BuseLi}.

\begin{theorem}[{\bf $b_2^+=1$ $J$-compatible inflation}]\label{thm:inflation}
    Let $(X,\omega)$ be a symplectic $4$-manifold with $b_2^+(X)=1$ and $J$ be a compatible almost complex structure. If $Z$ is an embedded $J$-holomorphic submanifold, then there exists a symplectic form $\omega'$ compatible with $J$ such that $[\omega']=[\omega]+t\text{PD}([Z])$, $t\in[0,\lambda)$ whenever 
$\omega'([Z])>0$. Or equivalently, $\lambda=\infty$ if $[Z]^2\geq 0$ and $\lambda=\frac{\omega([Z])}{-[Z]^2}$ if $[Z]^2<0$.
\end{theorem}

By the above theorem, if $J$ is a complex structure and $\omega$ is some K\"ahler form, we can deform the K\"ahler class $[\omega]$ along the directions in the classes of embedded $J$-holomorphic submanifolds. Our strategy is to construct a sequence of complex structures $J^1,J^2,\cdots$ on $X_{g,1},X_{g,2},\cdots$ such that for each $k$, $(X_{g,k+1},J^{k+1})$ is obtained by the blowup of $(X_{g,k},J^k)$ and they contain enough smooth holomorphic divisors to inflate.

Now, we review some facts in the theory of ruled surfaces which can be found in \cite[Section \RomanNumeralCaps{5}.2]{Hartshorne}. A rank $2$ holomorphic vector bundle $\mathcal{E}$ over the genus $g$ Riemann surface $\Sigma_g$ is called {\bf normalized} if $H^0(\mathcal{E})\neq 0$ but $H^0(\mathcal{E}\otimes\mathcal{L})=0$ for all line bundles $\mathcal{L}$ on $\Sigma_g$ with $\text{deg}\mathcal{L}<0$. If $\text{deg}\mathcal{E}$ is odd (resp. even), then the projectivization $\PP(\mathcal{E})$ is a complex manifold $(X,J)$ diffeomorphic to $\tilde{X}_{g,0}$ (resp. $X_{g,0}$). When $\mathcal{E}$ is normalized, the integer $e:=-\text{deg}\mathcal{E}$ is an invariant of $(X,J)$. The importance of this $e$-invariant is that any holomorphic section of $(X,J)$ will have self-intersection $\geq -e$, and the equality holds for a canonical section. By \cite[Exercise \RomanNumeralCaps{5}.2.5 (c)]{Hartshorne}, there always exists a complex structure $J$ on $\tilde{X}_{g,0}$ with $e$-invariant equal to $-1$ for any $g\geq 1$. This implies there is a $J$-holomorphic section in class $B_1$. We blowup a point on this section to get our desired complex structure $J^1$ on $X_{g,1}$.

 We first show how our method works for $X_{g,1}$. Note that by our construction of $J^1$, there will be an embedded $J^1$-holomorphic submanifold in class $B$ which is the proper transform of the section of self-intersection $1$ in $\tilde{X}_{g,0}$. The effect of inflation along $B$ will turn the normalized symplectic area vector $(\delta_B,\dd_1)$ into $(\frac{\dd_B}{1+t},\frac{\dd_1}{1+t})$ for any $t\in[0,\infty)$. By \cite[Proposition 3.24]{Voisinbook}, if $[\omega]$ is a K\"ahler class on $(\tilde{X}_{g,0},J)$, then $[\omega]+\varepsilon\text{PD}(F-E_1)$ will be a K\"ahler class on $(X_{g,1},J^1)$ for sufficiently small $\varepsilon$. Hence, by \cite[Proposition \RomanNumeralCaps{5}.2.21]{Hartshorne}, there will be an open neighborhood of the ray $\{(x,1)\,|\,x>\frac{{1}}{2}\}$ in $\mathcal{P}_1^1$ which can realized as the K\"ahler class of $J^1$. The inflation along $B$ initiating at the vectors in this small open neighborhood can then fulfill the entire $\mathcal{P}_1^1$. See Figure \ref{fig:inflationalongB}.

\begin{figure}[ht]
\centering
\begin{tikzpicture} [xscale=0.7,yscale=0.7]
    \begin{axis}[
  axis x line=center,
  axis y line=center,
  xtick={0,1},
  ytick={0,1},
  xlabel={$\dd_B$},
  ylabel={$\dd_1$},
  xlabel style={below right},
  ylabel style={above left},
  xmin=0,
  xmax=1.1,
  ymin=0,
  ymax=1.1]
    \addplot[name path=f,samples=300, domain=0:4.1,dashed] {min(sqrt(2*x),1)};

   \addplot[name path=g,samples=300, domain=0.5:4.1,white] {1-0.06*sqrt(x-0.5)};
   \addplot[name path=h,samples=300, domain=0.5:4.1,white] {1}; 
      
   \path[name path=axis] (axis cs:0,0) -- (axis cs:4.1,0);

\addplot [
        thick,
        color=black,
        fill=black, 
        fill opacity=0.2
    ]
    fill between[
        of=g and h,
    ];

    \node at (axis cs:  0.77,1) {Small blowup};
    \node at (axis cs:  0.5,0.5) {Inflate $B$};
 \draw [red, -> ] (axis cs:0.52,0.99) -- (axis cs:0.01,0.01) ; 
 \draw [red, -> ] (axis cs:0.72,0.99) -- (axis cs:0.01,0.01) ; 
  \draw [red, -> ] (axis cs:0.92,0.99) -- (axis cs:0.01,0.01) ; 
  \draw [red, -> ] (axis cs:1.07,0.99) -- (axis cs:0.01,0.01) ; 
   
    \end{axis}
\end{tikzpicture}  
\caption{Inflation along class B}\label{fig:inflationalongB}
\end{figure}

We then proceed to choose $J^2$ on $X_{g,2}$ by the blowup of $(X_{g,1},J^1)$ at the intersection point between the exceptional sphere in class $F-E_1$ and the section in class $B$. This choice will produce the smooth curves in classes $F-E_1-E_2$ and $B-E_2$. The effect of inflation along $F-E_1-E_2$ will turn $(x,y,z)\in\mathcal{P}_1^2$ into $(x+t,y+t,z+t)$ for any $t<\frac{1-x-y}{2}$. Given any $(\dd_B,\dd_1,\dd_2)\in\mathcal{P}_g^2$, we can repeat the small blowup argument as above to choose small $\varepsilon$ such that $(\dd_B-\dd_2+\varepsilon,\dd_1-\dd_2+\varepsilon,\varepsilon)$ is a K\"ahler class since the K\"ahler cone is open. Then the inflation along $F-E_1-E_2$ with $t=\dd_2-\varepsilon$ will realize $(\dd_B,\dd_1,\dd_2)$ as a K\"ahler class.

\begin{figure}[ht]
		\centering\includegraphics*[height=6cm, width=12cm]{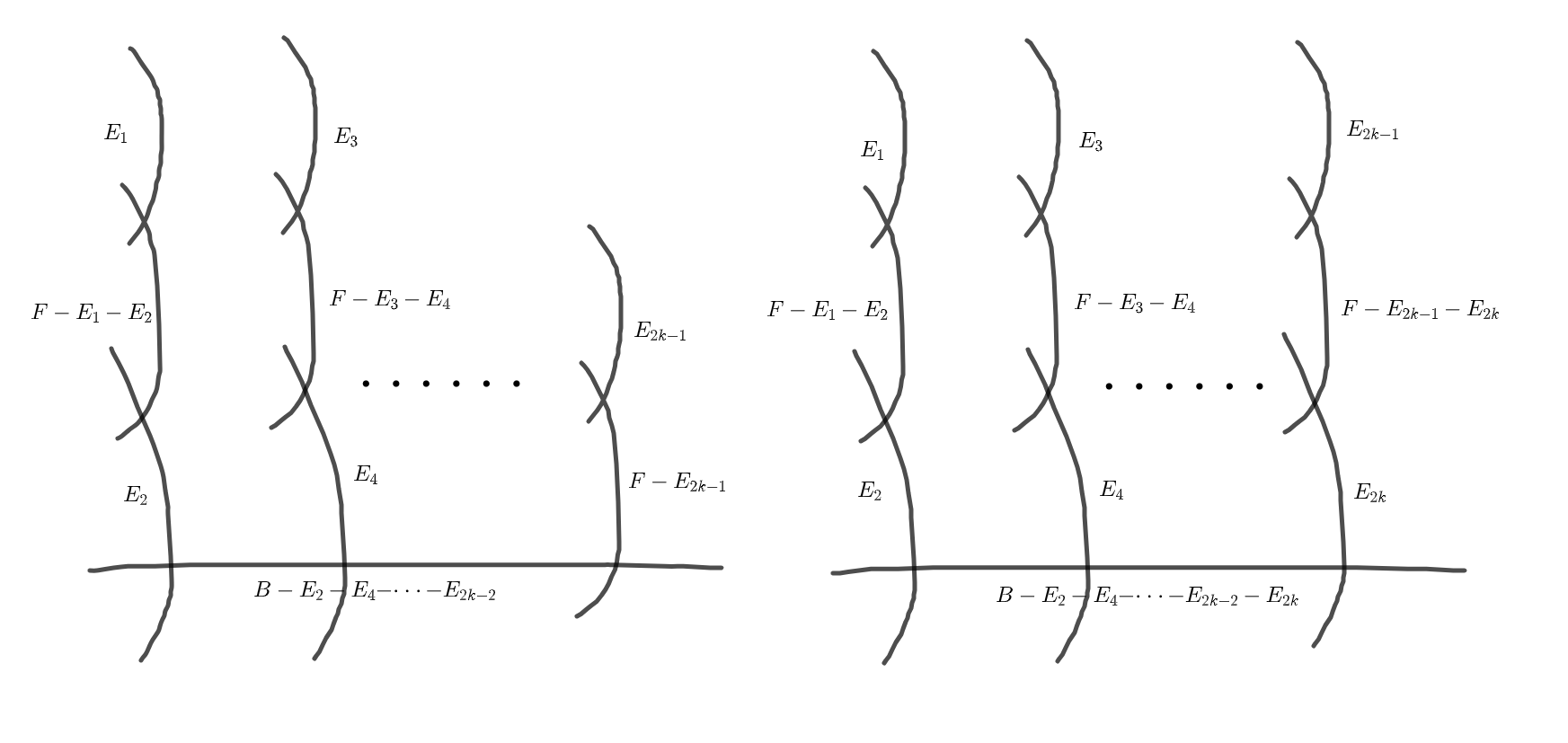}
 
		\caption{The curve configurations in $(X_{g,2k-1},J^{2k-1})$ and $(X_{g,2k},J^{2k})$.  \label{fig:inflationdraw}}
	\end{figure}
        
\begin{proof}[Proof of Proposition \ref{prop:c1positiveKahler}]
     For $k\geq 2$, we inductively choose $J^{2k-1},J^{2k}$ as follows. Assume that $(X_{g,2k-2},J^{2k-2})$ has smooth curves in classes $B-E_2-E_4-\cdots-E_{2k-2},F-E_1-E_2,F-E_3-E_4,\cdots,F-E_{2k-3}-E_{2k-2},E_1,\cdots,E_{2k-2}$. See Figure \ref{fig:inflationdraw} for the curve configurations that we construct to prepare for the inflation.
     
     We construct $J^{2k-1}$ as the blowup of $(X_{g,2k-2},J^{2k-2})$ at a point which is disjoint from these curves. This produces new exceptional spheres in classes $E_{2k-1}$ and $F-E_{2k-1}$. Suppose there is a vector $(\dd_B,\dd_1,\cdots,\dd_{2k-1})\in\mathcal{P}^{2k-1}_1$. Take $t$ to be a number less than but very close to $\frac{\dd_{2k-1}}{1-\dd_{2k-1}}$ and consider the vector \[v:=((1+t)\dd_B-(k-1)t,(1+t)\dd_1,(1+t)\dd_2-t,\cdots,(1+t)\dd_{2k-1}-t).\]
     When the last term $(1+t)\dd_{2k-1}-t$ is sufficiently close to $0$ (equivalently, $t$ is sufficiently close to $\frac{\dd_{2k-1}}{1-\dd_{2k-1}}$), we know $v\in \mathcal{P}_1^{2k-1}$ is a K\"ahler class by small blowup argument. Note that for any $i\neq j$, \begin{equation}\label{equation:Fij}
         (1+t)\dd_i+(1+t)\dd_j-t\leq (1+t)(\dd_1+\dd_2)-t<1+t-t=1.
     \end{equation}  
     This implies that $(1+t)\dd_{2k-1}<1$ which allows us to inflate along the classes $F-E_{2k-1}$ by $t$ to obtain the K\"ahler class  \[((1+t)\dd_B-(k-2)t,(1+t)\dd_1,(1+t)\dd_2-t,\cdots,(1+t)\dd_{2k-2}-t,(1+t)\dd_{2k-1}).\]
     Next, we want to inflate along the classes $E_{2l}$ and $F-E_{2l-1}-E_{2l}$ by $t$ for all $2\leq l\leq k-1$. Since there is no guarantee that $(1+t)\dd_{2l}-2t>0$ or $(1+t)(\dd_{2l-1}+\dd_{2l})<1$, a direct inflation by $t$ may not be allowed due to the area constraints. However, by (\ref{equation:Fij}), this can be overcome by a zig-zag inflation shown in Figure \ref{fig:zigzag}. Therefore, we can still arrive at a K\"ahler class
     \[((1+t)\dd_B,(1+t)\dd_1,(1+t)\dd_2-t,(1+t)\dd_3,(1+t)\dd_4-t,\cdots,(1+t)\dd_{2k-3},(1+t)\dd_{2k-2}-t,(1+t)\dd_{2k-1}).\]
     Finally, it remains to inflate along the class $B-E_2-E_4-\cdots-E_{2k}$ by $t$ which will give us $(\dd_B,\dd_1,\cdots,\dd_{2k-1})$. Note that there is no issue with the area since by the conditions for $\mathcal{P}_1^{2k-1}$, \[\dd_2+\dd_4+\cdots+\dd_{2k-2}\leq \frac{1}{2}(\dd_1+\dd_2+\cdots+\dd_{2k-2})<\frac{1}{2}(2\dd_B)=\dd_B.\]

\begin{figure}[ht]
\centering
\adjustbox{scale=0.8,center}{
\begin{tikzpicture} [xscale=1,yscale=1]
    \begin{axis}[
  axis x line=center,
  axis y line=center,
  xtick={0,1},
  ytick={0,1},
  xlabel={},
  ylabel={},
  xlabel style={below right},
  ylabel style={above left},
  xmin=0,
  xmax=1.1,
  ymin=0,
  ymax=1.1]
    \addplot[name path=f,samples=300, domain=0:4.1,dashed] {1-x};
   
    \node at (axis cs:  0.2,0.2) {$\bigstar$};
    \node at (axis cs:  0.6,0.2) {$\bullet$};
    \draw [blue, -> ] (axis cs:0.42,0.73) -- (axis cs:0.42,0.67) ;
     \node at (axis cs:  0.6,0.7) {Inflate $E_{2l}$};
     \draw [red, -> ] (axis cs:0.2,0.58) -- (axis cs:0.25,0.63) ; 
      \node at (axis cs:  0.6,0.6) {Inflate $F-E_{2l-1}-E_{2l}$};
    \node at (axis cs:  0.6,0.8) {$\bullet=((1+t)\dd_{2l-1},(1+t)\dd_{2l}-t)$};
    \node at (axis cs:  0.6,0.9) {$\bigstar=((1+t)\dd_{2l-1}-t,(1+t)\dd_{2l}-t)$};
    \draw [red, -> ] (axis cs:0.2,0.2) -- (axis cs:0.3,0.3) ; 
    \draw [red, -> ] (axis cs:0.3,0.2) -- (axis cs:0.4,0.3) ; 
    \draw [red, -> ] (axis cs:0.4,0.2) -- (axis cs:0.5,0.3) ; 
    \draw [red, -> ] (axis cs:0.5,0.2) -- (axis cs:0.6,0.3) ;
    \draw [blue, -> ] (axis cs:0.3,0.3) -- (axis cs:0.3,0.2) ;
    \draw [blue, -> ] (axis cs:0.4,0.3) -- (axis cs:0.4,0.2) ;
    \draw [blue, -> ] (axis cs:0.5,0.3) -- (axis cs:0.5,0.2) ;
    \draw [blue, -> ] (axis cs:0.6,0.3) -- (axis cs:0.6,0.2) ;
 
    \end{axis}
\end{tikzpicture}  }
\caption{ (\ref{equation:Fij}) guarantees that $\bigstar$ and $\bullet$ are inside the triangle. A direct inflation along $E_{2l}$ or $F-E_{2l-1}-E_{2l}$ by $t$ may extend beyond the triangle, which makes the area negative. A zig-zag inflation shown above would solve this issue.}\label{fig:zigzag}
\end{figure}

     We construct $J^{2k}$ as the blowup of $(X_{g,2k-1},J^{2k-1})$ at the intersection point between the curve in class $B-E_2-E_4-\cdots-E_{2k-2}$ and $F-E_{2k-1}$. This produces the smooth curve in class $F-E_{2k-1}-E_{2k}$ along which we can inflate. A similar argument as the case $k=1$ would realize all vectors in $\mathcal{P}^{2k}_1$ as K\"ahler classes.
\end{proof}

\begin{rmk}
    We have actually proved that any vector in $\mathcal{P}_1^n$ can be realized as the K\"ahler class of $X_{g,n}$ for all $g\geq 1$, which is stronger than the statement in Proposition \ref{prop:c1positiveKahler}.
\end{rmk}

\bibliographystyle{amsalpha}
\bibliography{mybib}{}

\end{document}